\documentclass[a4paper,11pt]{amsart}
\usepackage[T1]{fontenc}
\usepackage[latin9]{inputenc}
\usepackage{subfigure}
\usepackage{geometry}
\usepackage{amsmath,amsfonts,amssymb,amsthm}
\usepackage{color}
\usepackage{dsfont}
\usepackage{stmaryrd}
\usepackage{pb-diagram,graphicx}
\usepackage[colorlinks=true]{hyperref}
\usepackage{pdfsync}
\usepackage{bbm}

\newcommand{\p }{ \varphi}
\newcommand{\ph }{ \varphi}
\newcommand{\la }{ \lambda}

\newcommand{\Om }{ \Omega}
\newcommand{\be }{ \beta}

\newcommand{\dn}{ \partial_n}
\newcommand{\R }{\mathds{R}}
\newcommand{\vs}{\medskip}

\definecolor{darkgreen}{rgb}{0.0,0.6,0.0}

\newcommand{\tred}{\textcolor{red}}

\definecolor{trustcolor}{rgb}{0.71,0.14,0.07}

\newtheorem{lemma}{Lemma}

\newtheorem{proposition}{Proposition}
\newtheorem{theorem}{Theorem}
\newtheorem{remark}{Remark}

\begin{document}

\title[Properties of optimizers of the principal eigenvalue with indefinite weight]{Properties of optimizers of the principal eigenvalue with indefinite weight and Robin conditions}

\author{Jimmy Lamboley}
\address{CEREMADE, Universit\'e Paris-Dauphine, CNRS, PSL Research University, 75775 Paris, France\\
({\tt jimmy.lamboley@ceremade.dauphine.fr})}


\author{Antoine Laurain}
\address{Instituto de Matem\'atica e Estat\'istica, Universidade de S\~{a}o Paulo,
Rua do Mat\~{a}o, 1010, 05508-090 - S\~{a}o Paulo, Brazil  ({\tt laurain@ime.usp.br})}

\author{Gr\'egoire Nadin}
\address{CNRS, Universit\'e Pierre et Marie Curie (Univ. Paris 6), UMR 7598, Laboratoire Jacques-Louis Lions, F-75005, Paris, France ({\tt gregoire.nadin@upmc.fr}).}

\author{Yannick Privat}
\address{CNRS, Universit\'e Pierre et Marie Curie (Univ. Paris 6), UMR 7598, Laboratoire Jacques-Louis Lions, F-75005, Paris, France ({\tt yannick.privat@upmc.fr})}

\date{\today}
\thanks{
This work was partially supported by the project ANR-12-BS01-0007 OPTIFORM financed by the French Agence Nationale de la Recherche (ANR).
}

\begin{abstract}
In this paper, we are interested in the analysis of a well-known free boundary/shape optimization problem motivated by some issues arising in population dynamics. The question is to determine optimal spatial arrangements of favorable and unfavorable regions for a species to survive. The mathematical formulation of the model leads to an indefinite weight linear eigenvalue
problem in a fixed box $\Omega$ and we consider the general case of Robin boundary conditions on $\partial\Omega$. It is well known that it suffices to consider {\it bang-bang} weights taking two values of different signs, that can be parametrized by the characteristic function of the subset $E$ of $\Omega$ on which resources are located. Therefore, the optimal spatial arrangement is obtained by minimizing the positive principal eigenvalue with respect to $E$, under a volume constraint. By using symmetrization techniques, as well as necessary optimality conditions, we prove new qualitative results on the solutions. Namely, we completely solve the problem in dimension 1, we prove the counter-intuitive result that the ball is almost never a solution in dimension 2 or higher, despite what suggest the numerical simulations. We also introduce a new rearrangement in the ball allowing to get a better candidate than the ball for optimality when Neumann boundary conditions are imposed. We also provide numerical illustrations of our results and of the optimal configurations.
\end{abstract}

\maketitle

[0.1cm]
\noindent \textbf{AMS classication:} 49J15, 49K20,  	49R05, 49M05.

%

\section{Introduction: Elliptic Problem with indefinite weight}\label{sect:intro}
\subsection{The optimal design problem}\label{subsec:optDesPb}\label{ssect:pb}

In this paper, we are interested in the analysis of a well-known free boundary/shape optimization problem motivated by a model of population dynamics. The question is to determine the optimal spatial arrangement of favorable and unfavorable regions for a species to survive. We prove new qualitative results on the optimizer, using rearrangement techniques on the one hand, first order optimality conditions on the other hand.

More precisely, the following linear eigenvalue problem with indefinite weight is formulated in \cite{Fleming}:
\begin{equation}\label{eq:EV_main}
\left\{\begin{array}{rcl}
\Delta\ph+\lambda m\ph &=& 0  \mbox{ in }  \Omega ,\\
\partial_n\ph+\beta \ph &=& 0  \mbox{ on }  \partial\Omega ,
\end{array}\right.
\end{equation}
where $\Omega$ is a bounded domain (open and connected set) in $\mathds{R}^{N}$ with a Lipschitz boundary
$\partial\Omega$, $n$ is the outward unit normal vector on $\partial\Omega$, $\beta\in\R$ and the weight $m$ is a bounded measurable function which changes sign in $\Om$ (meaning that $\Om^+_{m} :=\{x\in\Om :m(x)>0\}$ has a measure strictly between 0 and $|\Om|$) and satisfies
\begin{equation}
\label{kappa} -1\le m(x)\le\kappa\quad\textrm{for almost every } x\mbox{ in }\Omega,
\end{equation}
where $\kappa>0$ is a given constant.

As we are motivated by a biological problem, we focus in this article on varying sign weights $m(\cdot)$, but most if not all our techniques can be applied to the case of positive $m$, as soon as a positive principal eigenvalue exists. Additional comments on positive weights can be found in Section \ref{sec:basics}.

 It is said that $\la$ is a \textit{principal eigenvalue} of \eqref{eq:EV_main} if the corresponding eigenfunction $\ph\in H^1(\Om)$ is positive. The existence of principal eigenvalues of \eqref{eq:EV_main} with respect to the parameter $\beta$ was discussed in \cite{AfrouziBrown,Bocher}. More precisely,
\begin{itemize}
\item in the Dirichlet case (``$\beta=+\infty$''), there are exactly two principal eigenvalues $\lambda^-<0<\lambda^+$, respectively associated with the eigenfunctions $\varphi^-$ and $\varphi^+$ satisfying
$$\int_{\Om}m(x)\varphi^-(x)^2dx<0, \;\;\;\int_{\Om}m(x)\varphi^+(x)^2dx>0,$$
\item the case $0<\beta<+\infty$ is similar to the Dirichlet case,
\item in the critical case $\beta=0$, which corresponds to Neumann boundary conditions, there are two principal eigenvalues, $0$ and $\lambda$, respectively associated with the eigenfunctions $1$ and $\varphi$; moreover $\lambda>0$ if and only if $\int_{\Omega}m(x)dx<0$, in which case we have 
$$\int_{\Om}m(x)\varphi(x)^2dx>0,$$
\item
for $\beta<0$, it was shown in \cite{AfrouziBrown} that, depending on $\beta$, (\ref{eq:EV_main}) has two, one or zero principal eigenvalues. In the case of two principal eigenvalues,  distinguishing them is achieved by considering the sign of $\int_\Om m(x) \ph(x)^2\, dx $.
\end{itemize}

In the rest of the paper , we focus on the case $\beta\geq 0$ which is relevant for applications in the context of species survival. Therefore, assuming that $|\Om^+_{m}|>0$, and besides that $\int_{\Om}m<0$ if $\beta=0$, there exists a unique positive principal eigenvalue for Problem \eqref{eq:EV_main}, denoted $\lambda(m)$. Moreover, $\lambda(m)$ rewrites also as
\begin{equation}
\label{eq:variational}\la (m)=\inf_{\p\in\mathcal{S}(m)}\Re_m[\varphi], 
\end{equation}
where
\begin{equation}\label{def:Re}
 \Re_m[\varphi]=\frac{\int_\Om |\nabla\p|^2 + \beta\int_{\partial\Omega}\p^2}{\int_\Om m\p^2}
\qquad \textrm{and}\qquad \mathcal{S}(m)=\left\{\p\in H^1(\Om):\int_\Om m\p^2>0 \right\},
\end{equation}
whenever $\beta<+\infty$. This has been proved for Neumann boundary conditions ($\beta=0$) in \cite{LoYa} and the reader can check that the extension to $\beta>0$ is straightforward.
Moreover, $\la(m)$ is simple, the infimum is reached, the associated eigenfunctions do not change sign in $\overline{\Om}$, and any eigenfunction belonging to $\mathcal{S}(m)$ and that do not change sign is associated with $\la(m)$.

In the Dirichlet case, where the boundary condition $\partial_n\varphi+\beta \varphi=0$ on $\partial\Om$ is replaced by $\varphi=0$ on $\partial\Om$, this formulation becomes
\begin{equation}\label{eq:variationalD}
\la (m)=\inf\left\{
\frac{\int_\Om |\nabla\p|^2}{\int_\Om m\p^2}
, \;\;\;\p\in H^1_{0}(\Om), \;\;\int_\Om m\p^2>0 \right\}.
\end{equation}
Indeed, according to Proposition \ref{prop:asympDiri} below, the Dirichlet eigenvalue can be obtained from the Robin eigenvalues by letting $\beta \to +\infty$.

\bigskip

Throughout this paper, we will analyze the following optimization problem, modeling the optimal arrangement for a species to survive.

\begin{quote}
\noindent{\bf Optimal arrangement of ressources for species survival.}
\textit{Let $\Omega$ be a bounded domain of $\R^N$. Given $\kappa>0$ and $m_{0}\in (-\kappa,1)$ if $\beta>0$ or $m_{0}\in (0,1)$ if $\beta =0$, we consider the optimization
 problem 
\begin{equation}\label{mini} 
\inf_{m\in \mathcal{M}_{m_0,\kappa}} \lambda(m)\quad\textnormal{where}\quad  \mathcal{M}_{m_0,\kappa}= \left\{ m\in L^{\infty}(\Omega): -1\leq m\leq \kappa,\; |\Om_{m}^+|>0,\;\int_{\Om}m\leq -m_{0}|\Om| \right\}.
\end{equation}
}
\end{quote}

In Section \ref{ssect:bio}, the biological motivations for considering such a problem, as formulated by Cantrell and Cosner in \cite{MR1014659,MR1112065}, are recalled. 

It is well known (see for example \cite{MR2886017, LoYa} and Section \ref{ssect:optimality}) that Problem \eqref{mini} has a solution $m^*$, and moreover there exists a measurable subset $E^*\subset\Omega$ such that, up to a set of zero Lebesgue measure, there holds 
$$
m^*=m_{E^*}\quad\textrm{ where }\quad m_{E^*}=\kappa \mathbbm{1}_{E^*}- \mathbbm{1}_{\Omega\backslash E^*}\textrm{ a.e. in }\Omega.
$$
In addition, one has $\int_\Om m^*=\int_\Om m_{E^*} = -m_0 |\Om|$, which is a direct consequence of a  comparison principle\footnote{Indeed, by comparing the Rayleigh quotients for $m_1$ and $m_2$, one gets
$$m_1>m_2 \ \Longrightarrow \ \lambda(m_1)<\lambda(m_2).$$
One can refer for instance to \cite[Lemma 2.3]{LoYa}.}.
In other words, the minimizer saturates at the same time the pointwise and integral constraints on $m$. As a consequence, the optimal design problem above can actually be rewritten as a shape optimization problem:
\begin{quote}
\noindent{\bf Shape optimization formulation of Problem \eqref{eq:variational}.} Using the same notations as above, let $c=\frac{1-m_{0}}{\kappa+1}\in(0,1)$. We investigate the optimal design problem
\begin{equation}\label{minishape} 
\inf_{E\in \mathcal{E}_{c,\kappa}} \lambda(E)\quad\textnormal{with} \quad \lambda(E)=\lambda(\kappa \mathbbm{1}_{E}- \mathbbm{1}_{\Omega\backslash E}),
\end{equation}
where $\mathcal{E}_{c,\kappa}$ denotes the set of Lebesgue measurable sets $E$ such that $0<|E|\leq c|\Om|$.
\end{quote}

Therefore and to sum up, given $\beta\geq 0$ and $\kappa\in(0,\infty)$, it is equivalent to choose either the parameter $m_{0}\in(-\kappa,1)$ (with, in addition, $m_{0}>0$ when $\beta=0$) or the parameter $c\in(0,1)$ (with $c<\frac{1}{\kappa+1}$ if $\beta =0$) 
and the solution is then naturally a function of three parameters (either $(\beta,\kappa,m_{0})$ or $(\beta,\kappa,c)$), once $\Omega$ is given. 

\subsection{New results}

In this paper, we obtain three new qualitative results on the shape optimization problem we described in the previous section. 

\bigskip

In our {\bf first result}, we provide a complete description of the optimal sets in the one-dimensional case.
\begin{theorem}\label{Robin_connected}
Assume $N=1$ and without loss of generality, let us consider $\Om=(0,1)$. Let $\beta\in[0,+\infty]$, $\kappa>0$, and $c\in(0,1)$ whenever $\beta>0$ or $c\in (0,\frac{1}{\kappa+1})$ whenever $\beta=0$. Define
\begin{equation}\label{eq:beta*}
\beta^*:=\left\{\begin{array}{ll}
\frac{2}{c\sqrt{\kappa}}\arctan\left(\frac{1}{\sqrt{\kappa}}\right)&\textrm{ if }\kappa>1\\
\frac{\pi}{2c}&\textrm{ if }\kappa=1\\
\frac{1}{c\sqrt{\kappa}}\left(\arctan\left(\frac{2\sqrt{\kappa}}{\kappa-1}\right)+\pi\right)&\textrm{ if }\kappa<1
\end{array}\right.
\end{equation}
Then 
\begin{itemize}
\item if $\beta>\beta^*$, the unique\footnote{Here the uniqueness must be understood up to some subset of zero Lebesgue measure. In other words if $E^*$ is optimal then the union of $E^*$ with any subset of zero measure is also a solution.} solution of \eqref{minishape} is the interval of length $c$ and centered \tred{at} $1/2$,
\item if $\beta<\beta^*$, the solutions  of \eqref{minishape} are exactly $(0,c)$ and $(1-c,1)$,
\item if $\beta=\beta^*$, the solutions  of \eqref{minishape} are exactly all intervals of length $c$,
\end{itemize}
\end{theorem}
This result is clear for $\beta=+\infty$ (i.e. for the Dirichlet case) using symmetrization, and is proven in \cite{LoYa} for $\beta=0$. In the more general situation $\beta\in(0,\infty)$, the minimizer among intervals has been computed in \cite{CantrellCosnerRobin} when $\kappa=1$ and in \cite{MR2886017} for $\kappa>0$. Therefore the previous result rests upon the fact that the optimal set is an interval. We prove this in Section \ref{1d_section}. Our method is based on a symmetrization argument, and is therefore closer to the case $\beta=+\infty$ than the method of \cite{LoYa}. We cannot use Steiner/Schwarz symmetrization since it may not decrease the gradient term $\int_{\Om}|\nabla\varphi|^2$ (except if $\beta=+\infty$ in which case $\varphi\in H^1_{0}(\Om)$). Therefore we use a sort of two sided decreasing rearrangement, whose center is chosen appropriately so that symmetrized functions are still admissible, and which decreases every term in the Rayleigh quotient. Notice that even in the case $\beta=0$, this gives a new proof of the result of \cite{LoYa}, which is more straightforward.
\bigskip

Our {\bf second result} deals with the case $N\geq 2$, and disproves 
the commonly stated conjecture that the ball is a minimizer for certain domains $\Omega$ and certain values of the parameters $\beta$, $\kappa$ and $c$.
This conjecture was also suggested by numerical computations and results (see \cite{Roques-Hamel}).


We prove that the conjecture is false, except maybe for very particular choices of the parameters such as the box $\Omega$. In particular, if $\Om$ is not a ball, a minimizer for \eqref{minishape} cannot be a ball, whatever the value of the parameters $\beta$, $\kappa$ and $c$ are. This means in particular that the optimal set does not minimize the surface area of its boundary.

More precisely, we have the following general result.

\tred{
}
\begin{theorem}\label{theo:optimball} 
Let $N\geq 2$, $\Om$ a domain of $\R^N$ such that its boundary $\partial\Om$ is connected and of class $\mathcal{C}^1$, $(\beta,\kappa,c)\in[0,+\infty]\times(0,+\infty)\times(0,1)$ with $c<\frac{1}{1+\kappa}$ if $\beta=0$, and $E$ an open subset of $\Om$ of measure $|E|=c|\Om|$. 
Assume that either $E$ or $\Om\backslash \overline{E}$ is rotationally symmetric (i.e. a union of concentric rings, whose center is denoted $O$) and has a finite number of connected components.
\begin{itemize}
\item If the set $E$ is a critical point\footnote{This means that $E$ satisfies the necessary first order optimality conditions of Problem \eqref{minishape}, in other words that $E$ is an upper level set of the eigenfunction $\varphi$ associated with the principal eigenvalue $\lambda(E)$
, more precisely that there exists $\alpha$ such that $E=\{\varphi>\alpha\}$, see also Section \ref{ssect:optimality}.} of the optimal design problem \eqref{minishape} then $\Om$ is a ball of center $O$.
\item There exists $\beta_0\geq 0$ such that if $\beta\geq \beta_0$, and if $E$ solves Problem \eqref{minishape}, then $E$ and $\Om$ 
 are concentric balls.
\end{itemize}
\end{theorem}

To our best knowledge, this result is completely new, even if $\beta=0$ or $\beta=+\infty$.
Theorem \ref{theo:optimball}  lets open the issue of knowing whether the optimal configuration $E$ is rotationally symmetric. 

Note that when $\partial\Om$ is disconnected, it is likely that the result is not valid anymore. For instance, if $\Om$ is an annulus, one would expect the existence of a rotationally symmetric critical set $E$.

The assumption on the finite number of connected components for $E$ ensures that $\partial E \cap \Om$ is analytic, which is crucial in our proof. It could thus be replaced by an analyticity assumption on $\partial E \cap \Om$.

Note that our result is also interesting if $\Om$ is a ball. 
It asserts that the only rotational symmetric domain which is a candidate for optimality is the centered ball whenever the parameter $\beta$ is large enough. It implies in particular that an annulus cannot be a minimizer, even if $\beta=0$.






The proof of Theorem \ref{theo:optimball} uses the first order optimality condition, namely that $\varphi$ is constant on $\partial E$, to infer that $\varphi$ is necessarily radial (i.e. $\varphi$ is a function of $|x|$) on the whole domain $\Om$. 
To that end, we built particular test functions that can be interpreted as angular derivatives of the function $\varphi$. 

Then, we rewrite the problem as an optimization problem bringing into play only functions of the polar variable $r$. This allows to conclude that $\Om$ must be a ball, proving that the associated eigenvalue on the largest inscribed ball and the smallest circumscribed ball are the same. 

The second part of the result is proven by using a symmetrization argument, which, as for Theorem \ref{Robin_connected}, works for large values of the parameter $\beta$ and despite the lack of the usual hypotheses for this kind of argument.



We also underline here that the converse of Theorem \ref{theo:optimball} is not true. More precisely, the radial symmetry of $\Om$ does not imply that a similar symmetry will hold for the minimizing set $E^{*}$. Indeed, an analytical example which shows that a radially symmetric set $E$ cannot be a minimizer has been provided in \cite[Theorem 2.5]{JhaPorru} when $\Om$ is a thin and large annulus, for Neumann boundary conditions. 
Even for Dirichlet boundary conditions, symmetry breaking can occur. In \cite{CGIKO}, this phenomenon is observed and explicit examples are provided for a closely related problem.

Finally, let us highlight that we prove in the second step of Theorem \ref{theo:optimball} the following interesting byproduct: among the set of rotationally symmetric open subsets $E$ of $\Omega$ of prescribed measure, the centered ball is the only minimizer for $\beta$ large enough.

\begin{proposition}\label{prop:geomprop}
Let $N\geq 2$, $\Om$ be the $N$-dimensional unit ball of $\R^N$ centered at the origin, $(\beta,\kappa,c)\in[0,+\infty]\times(0,+\infty)\times(0,1)$ with $c<\frac{1}{1+\kappa}$ if $\beta=0$. Let $E$ be a rotationally symmetric and concentric open subset of $\Om$ of measure $|E|=c|\Om|$. Then, any eigenfunction $\varphi$ associated with $\lambda(E)$ is radial. Moreover, there exists $\beta_0>0$ such that there holds $\lambda(E)\geq \lambda(E^S)$ for every $\beta\geq \beta_0$, where $E^S$ denotes the centered ball of volume $c|\Om|$, and $\lambda(E)= \lambda(E^S)$ if and only if $E=E^S$.
\end{proposition}

\bigskip

Our {\bf third result} is motivated by Theorem \ref{theo:optimball} which asserts in particular that a ball is a candidate for optimality only if $\Omega$ itself is a concentric ball. It remains to decide, in the case where $\Om$ is a ball, whether the centered ball is optimal or not. In the case $\beta=+\infty$ it is actually the case (classically, by using the so-called Schwarz symmetrization), but we expect that it is not the case for every values of $\beta$. 
We prove that the centered disk is not optimal in the case $\beta=0$ (Neumann boundary condition) and for $N\geq 2$.

\begin{theorem}\label{thm3_simple}
Let $\kappa>0$, $N = 2,3,4$, and $c\in(0,\frac{1}{1+\kappa})$. Assume $\beta = 0$ and $\Om = B(0,1)\subset\R^N$ is the disk of radius $1$ centered at the origin.
Then the centered ball of volume $c|\Om|$ is not a minimizer for Problem \eqref{minishape}.
\end{theorem}
This result is a particular case of the more general result stated in Theorem \ref{th:deformation}: we use a non-local deformation that decreases strongly the value of $\lambda$, more precisely if $E$ is a centered ball (or more generally a radially symmetric set), we build a set $\widehat{E}$ that ``sticks'' on the boundary of $\partial\Om$ and satisfies 
\begin{equation}
\label{eq:estimate_lambda_hat}\lambda(\widehat{E}) < c_N\lambda(E).
\end{equation}
with $$c_N = \frac{5N-4}{4N}.$$
We compute $c_2=6/8, c_3 = 11/12, c_4=1$ which yields Theorem  \ref{thm3_simple}. For $N>4$, we have $c_N>1$ so we cannot conclude from the estimate \eqref{eq:estimate_lambda_hat}.
For $\beta>0$, the situation is unclear: first, it seems our strategy cannot be adapted, even if $\beta$ is small, though it is reasonable to expect that the centered ball is not a solution in that case. For $\beta$ large, we do not know whether the situation is similar to the 1-dimensional case (that is there exists $\beta^*$, possibly depending on $\kappa$ and $c$, such that for $\beta>\beta^*$ the solution is a centered ball, in other words the same as if $\beta=+\infty$) or if it can be proven that the centered ball is a solution only if $\beta=+\infty$.

%



The article is organized as follows: in Section \ref{gen}, we provide some explanations about the biological model motivating our study, as well as a short survey on several existing results related to the problem we investigate and similar ones. Section \ref{1d_section} is devoted to the proof of Theorem \ref{Robin_connected}, solving completely Problem \eqref{minishape} in the case where $\Omega=(0,1)$ with Robin boundary conditions. The whole section \ref{geom_prop} is devoted to proving Theorem \ref{theo:optimball}. Finally, in Section \ref{sec:applications}, we provide qualitative properties of the minimizers of Problem \eqref{minishape} in the particular cases where Neumann boundary conditions are considered, $\Omega$ is a $N$-orthotope or a two dimensional euclidean disk. In this last case, we prove in Theorem \ref{th:deformation} a quantitative estimate showing symmetry-breaking for the minimizers. This allows in particular to derive Theorem \ref{thm3_simple}. All these results are illustrated by numerical simulations.


\section{Preliminaries and State of the art}\label{gen}

In this section, we gather several known facts about Problem \eqref{mini}, from the biological motivation of the model to deep and technical results about minimizers, mainly for two reasons. 
First of all, we will use several known results in our proofs, therefore we want to recall them for the convenience of the reader. 
Second, we want to highlight the novelty of our results, even when we will be led to state results for certain choices of parameters (such as $\Om$, $\beta$, the dimension, and so on).

\subsection{Biological model}\label{ssect:bio}
The main biological motivation for studying extremal properties of the principal eigenvalue
$\lambda=\lambda(m)$ with respect to the weight $m$ comes from the diffusive  logistic equation
\begin{equation}\label{eq: logistic}
\left\{
\begin{array}{ll}
u_t=\Delta u+\omega u [m(x)-u] \quad &\mbox{in } \Omega\times
\R^+,
\vs \\
\dn u+\beta u =0 \quad &\mbox{on }\partial\Omega\times \R^+,
\vs \\
u(0,x)\ge 0, \quad u(0,x)\not\equiv 0 \ \ &\mbox{in}\
\overline\Omega,
\end{array}
\right.
\end{equation}
introduced in \cite{MR0043440}, where $u(t,x)$ represents the density of a species at location $x$ and time $t$, and $\omega$
is a positive parameter. Concerning the boundary conditions on $\Omega$, the case $\beta=0$ corresponds to Neumann or no-flux boundary condition, meaning that the boundary acts as a barrier, i.e. any individual reaching the boundary returns to the interior. The case $\beta=+\infty$ corresponds to Dirichlet conditions and may be interpreted as a deadly boundary, i.e. the exterior environment is completely hostile and any individual reaching the boundary dies. For intermediate values $0<\beta<+\infty$, we are in the situation where the domain $\Om$ is surrounded by a partially inhospitable region, where inhospitableness grows with $\beta$. The weight $m$ represents the intrinsic growth rate of species: it is positive in the favorable part of habitat ($\Omega^+_{m}=\{m>0\}$) and negative in the unfavorable one ($\Omega^-_{m}=\left\{m<0\right\}$). The integral
of $m$ over $\Om$ measures the total resources in a spatially heterogeneous environment.

The logistic equation (\ref{eq: logistic}) plays an important role
in studying the effects of dispersal and spatial heterogeneity in
population dynamics; see, e.g. \cite{MR1014659,MR1112065,  MR2191264} and
the references therein. It is known that if $\omega \le \lambda(m)$,
then $u(t,x)\to 0$ uniformly in $\overline\Omega$ as $t\to \infty$
for all non-negative and non-trivial initial data, i.e., the species
go to extinction; if, however, $\omega>\lambda(m)$, then $u(t,x)\to u^*(x)$
uniformly in $\overline\Omega$ as $t\to \infty$, where $u^*$ is the
unique positive steady solution of  (\ref{eq: logistic}), i.e., the species survives. 


Since the species can be maintained if and only if
$\omega>\lambda(m)$, we see that the smaller $\lambda(m)$ is, the
more likely the species can survive. With this in mind, the
following question was raised and addressed by Cantrell and Cosner in
\cite{MR1014659,MR1112065}: {\it among all functions $m\in\mathcal{M}_{m_{0},\kappa}$, which $m$ will yield the smallest principal
eigenvalue $\lambda(m)$}, whenever it exists? From the biological point of view, finding
such a minimizing function $m$ is equivalent to determining the
optimal spatial arrangement of the favorable and unfavorable patches
of  the environment for species to survive. This
issue is important for public policy decisions on conservation of
species with limited resources.

\subsection{Other formulation}

In this section, we address a closely related optimal design problem. Let $(\mu_{-},\mu_{+})\in\R^2$ such that $\mu_{-}<0<\mu_{+}$. For $ \mu\in L^{\infty}(\Omega;[\mu_{-},\mu_{+}])$, the classical reaction-diffusion model in homogeneous environments of Fisher, Kolmogorov et al. \cite{Fisher,Kolmogorov} generalizes as:
\begin{equation}\label{eq: logistic2}
\left\{
\begin{array}{ll}
v_t=\Delta v+ v [\mu(x)-\nu(x) v] \quad &\mbox{in } \Omega\times
\R^+,
\vs \\
\dn v+\beta v =0 \quad &\mbox{on }\partial\Omega\times \R^+,
\vs \\
v(0,x)\ge 0, \quad v(0,x)\not\equiv 0 \ \ &\mbox{in}\
\overline\Omega,
\end{array}
\right.
\end{equation}
where $v(t,x)$ represents the population density at time $t$ and position $x$. The function $\mu$ stands for the intrinsic grow rate of the species whereas the function $\nu$ is the susceptibility to crowding and is chosen in $L^\infty(\Omega)$ and such that $\mathrm{essinf} \nu >0$.

According to \cite{BHR,Roques-Hamel} and similarly to the previous model, a necessary and sufficient condition of species survival writes $\gamma(\mu)<0$, where $\gamma(\mu)$ denotes the principal eigenvalue associated with the elliptic problem
\begin{equation}\label{eq:EV_main2}
\left\{\begin{array}{ll}
-\Delta \psi = (\mu(x) +\gamma) \psi &\mbox{ in }\Omega \\
\partial_n \psi +\beta \psi= 0 &\mbox{ on }\partial\Omega.
\end{array}\right.
\end{equation}
It is notable that this condition does not depend on the function $\nu(\cdot)$.

The principal eigenvalue $\gamma(m)$ of \eqref{eq:EV_main2} is unique, nonnegative and given by 
\begin{equation}
\label{eq:variational2}\gamma (\mu)=\inf_{\substack{\psi \in H^1(\Omega)\\ \psi \neq 0}}\frac{\int_\Om |\nabla\psi |^2-\int_{\Omega}\mu \psi^2 + \beta\int_{\partial\Omega}\psi^2}{\int_\Om\psi^2} .
\end{equation}
Moreover, $\gamma(\mu)$ is simple, and the infimum is attained only by associated eigenfunctions that do not change sign in $\overline{\Om}$.

As previously, a similar analysis of the biological model leads to the study of the following optimal design problem.

\begin{quote}
\noindent{\bf Optimal rearrangement of species problem, equivalent formulation.}
\textit{Given $(\mu_{-},\mu_{+})\in \R^2$ such that $\mu_{-}<0<\mu_{+}$ and ${\mu}_{0}\in (\mu_{-},\mu_{+})$, we are interested in 
\begin{equation}\label{mini2} 
\inf\left\{ \gamma(\mu) ;\;\;\mu\in L^\infty (\Omega;[\mu_{-},\mu_{+}])\textrm{ such that }|\Omega^+_{\mu}|>0,\textrm{ and }
 \int_{\Omega} \mu\le - \mu_0|\Omega|\right\}.
\end{equation}
}
\end{quote}

Following the same approach as for Problem \eqref{mini}, it is standard to prove that Problem \eqref{mini2} has a solution $\mu^*$ which is a bang-bang function:
$\mu_{+} \mathbbm{1}_{E^*}+\mu_{-}\mathbbm{1}_{\Omega\backslash E^*}$,
and the volume constraint is active.

Note that Problems \eqref{mini} and \eqref{mini2} have been considered independently in the literature on optimal arrangement of ressources for species survival. In \cite{CGIKO} these two similar problems have been investigated and it is shown that they are equivalent in a sense recalled below. The main difference with our case is, roughly speaking, that the weight $\mu(x)+\gamma$ is positive in \cite{CGIKO}.  However as far as the equivalence of   \eqref{mini} and \eqref{mini2} is concerned, the proof is the same and we recall the result here, for mainly two reasons: 
firstly because it allows us to use certain results from both literatures, and secondly, because while our statements and proofs deal with formulation \eqref{mini}, our new results are actually also valid for solutions of \eqref{mini2}.

\begin{theorem}(\cite[Theorem 13]{CGIKO}, Equivalence between the two formulations)\label{theo:compare}
Let $\Omega$ a bounded domain and $\beta\in\R_{+}$.
\begin{itemize}
\item Let $\kappa>0$ and $m_{0}\in(-\kappa,1)$, with in addition $m_0>0$ if $\beta =0$. Assume that $E^*$ is a solution of Problem \eqref{minishape} and let $\lambda^{*}=\lambda(\kappa \mathbbm{1}_{E^*}-\mathbbm{1}_{\Omega\backslash E^*})$ be the minimal eigenvalue. 
Then, $\mu_{E^*}=\mu_{+}\mathbbm{1}_{E^*}+\mu_{-}\mathbbm{1}_{\Om\backslash E^*}$ is a solution of Problem \eqref{mini2} with parameters $\mu_{-}=-\lambda^{*}$, $\mu_{+}=\kappa \lambda^{*}$, $\mu_{0}=\lambda^{*}m_{0}$, and moreover $\gamma^{*}=\gamma(\mu_{E^*})=-\mu_{-}-\lambda^{*}$.
\item Conversely, let $\mu_{-}<\mu_{+}$, $\mu_{0}\in(-\mu_{+},-\mu_{-})$, let $\mu=\mu_{E^*}$ be a bang-bang solution of Problem \eqref{mini2}, and let $\gamma^{*}=\gamma(\mu_{+} \mathbbm{1}_{E^*}+\mu_{-}\mathbbm{1}_{\Omega\backslash E^*})$ be the minimal eigenvalue. 
Then, $m_{E^*}=\kappa \mathbbm{1}_{E^{*}}-\mathbbm{1}_{\Om\backslash E^{*}}$ is a solution of Problem \eqref{mini} with parameters $\kappa=-\frac{\gamma^{*}+\mu_{+}}{\gamma^{*}+\mu_{-}}$,  $m_{0}=\frac{\mu_{0}-\gamma^{*}|\Omega|}{\lambda^{*}}$, and moreover $\lambda^{*}=\lambda(m_{E^{*}})=-\mu_{-}-\gamma^{*}$.
\end{itemize}
\end{theorem}


\subsection{About the class of admissible weights}\label{sec:basics}

In this section, we gather several comments related to the choice of constraints on the weight $m$: the pointwise one and the global (mass constraint) one.

As a first remark, there exists a wide literature concerning problems similar to \eqref{mini}, where one aims at minimizing the first eigenvalue of the operator $-\frac{1}{m}\Delta$ where $\Delta$ denotes the Dirichlet-Laplacian operator, with respect to functions $m$ satisfying the pointwise constraint $a\leq m(\cdot)\leq b$ a.e. in $\Om$ with $0<a<b$ as well as a global integral constraint. Such problems are motivated by optimal design issues with respect to structural eigenvalues. We refer for instance to \cite{Cox-MacL,kohn-strang1,kohn-strang2,kohn-strang3,krein,MR1613638} where Dirichlet boundary conditions are considered, and to \cite[Chapter 9]{henrot-eigbook} for a survey on these problems. 

We also mention the case of non-homogeneous membranes, similar to \eqref{mini} with Neumann boundary conditions, but with the positive weight (also called density) $0<a\leq m(\cdot)\leq b$ and without a mass constraint on the weight, see \cite[Theorem 1.4.2]{MR2641628} and \cite{MR0313648,MR572958}. See also \cite{colbois:hal-01330456} for a similar problem in the context of Riemannian manifolds with a mass preservation constraint.
Notice however that in the case where the weight $m$ is positive and where Neumann boundary conditions are imposed on the eigenfunction $\varphi$, there is no positive principal eigenvalue: indeed, assuming there exists a positive eigenfunction $\varphi$ associated to $\lambda$, we obtain by integration by parts that
$$
-\int_\Omega \frac{|\nabla \varphi|^2}{\varphi^2}=\lambda \int_\Omega m,
$$
which is a contradiction since this yields $\int_\Omega m<0$ if $\lambda>0$.

All the results presented in this section and in Section \ref{ssect:optimality} (i.e. the monotonicity of eigenvalues, the {\it bang-bang} property of minimizers) were established in \cite{LoYa} and \cite{JhaPorru} in one dimension for Neumann conditions (i.e. $\beta=0$) for Problem \eqref{mini}, and in \cite{Roques-Hamel} for periodic boundary conditions for Problem \eqref{mini2}. We claim that they can be straightforwardly extended to Robin conditions. Therefore we do not reproduce here the proof but rather refer to \cite[Theorem 1.1]{LoYa}  or \cite[Appendix A]{Roques-Hamel} for details. We also mention \cite{derlet}, for an extension of these results to principal eigenvalues associated to the one dimensional $p$-Laplacian operator.

Finally, concerning the constraint $\int_\Om m\leq -m_0|\Om|$, or equivalently $|E|\leq \frac{1-m_0}{\kappa+1}|\Om|$ (see Section \ref{ssect:pb}), we claim that it is active and therefore, it is similar to deal with the same optimal design problem where the inequality constraint is replaced by the equality one
$$
\int_\Om m= -m_0|\Om|\qquad \textrm{or}\qquad |E|= \frac{1-m_0}{\kappa+1}|\Om|.
$$
Indeed, it is a consequence of the comparison principle (see \cite[Lemma 2.3]{LoYa}) 
$$
m_1>m_2 \ \textrm{(resp. $E_1\subset E_{2}$)}\  \Longrightarrow \ \lambda(m_1)<\lambda(m_2) \ \textrm{(resp. $\lambda(E_{1}) > \lambda(E_{2})$)}.
$$
This comparison principle is obtained in an elementary way, by comparing the Rayleigh quotient for $m_1$ and $m_2$ (resp. $E_{1}$ and $E_{2}$).

\subsection{First order optimality conditions and \textit{bang-bang} property of minimizers}\label{ssect:optimality}

The minimizing set $E^{*}$ is a level surface of the principal eigenfunction $\varphi$. Indeed, this is proved in \cite{CGIKO} for Dirichlet boundary conditions but the arguments can be straightforwardly extended to Robin boundary conditions. Let us briefly recall the main steps. We denote by $\varphi$ the eigenfunction associated to the minimal principal eigenvalue $\lambda^*$. First, note that the optimal design problem $\max_{m\in \mathcal{M}_{m_0,\kappa}}\int_{\Om}m\varphi^{2}$, has a solution given by $m=\kappa 1_{E_{\alpha}}-1_{\Om \backslash E_{\alpha}}$, where $\{\varphi > \alpha\} \subset E_{\alpha}\subset \{\varphi \geq\alpha\}$. This is the so-called ``baththub principle'', see e.g. \cite[Theorem 1]{PTZ4}. Using a direct comparison argument and arguing by contradiction, one shows that $\lambda^*\geq \lambda \big(\kappa 1_{E_{\alpha}}-1_{\Om \backslash E_{\alpha}}\big)$ and therefore, $E_{\alpha}$ is a minimizing set for Problem \eqref{mini}.  
Next, it is standard, as $\varphi\in H^2(\Om)$ that $\Delta \varphi=0$~a.e. on $\{\varphi = \alpha\}$, which implies 
$\lambda^{*}\big(\kappa 1_{E_{\alpha}}-1_{\Om \backslash E_{\alpha}}\big)=0$ a.e. on $\{\varphi = \alpha\}$, which is impossible if $\{\varphi = \alpha\}$ is not negligible. 
Hence, one infers that $E^{*} =\{\varphi>\alpha\}$ up to a set of measure zero.

\subsection{Regularity theory}

Proving the regularity of the free boundary $\Gamma:=\partial E^{*}\backslash \partial \Om$ is a very difficult question in general. It follows from classical elliptic regularity that the principal eigenfunction $\varphi$ is $\mathcal{C}^{1,a}(\Om)$ for every $a\in [0,1)$. Hence, as $E^{*}=\{\varphi>\alpha\}$ up to a set of Lebesgue measure zero (see Section \ref{ssect:optimality}), the boundary $\Gamma$ is $C^{1,a}$-smooth at any point where $\nabla \varphi \neq 0$ and therefore, using a bootstrap argument, one infers the local analytic regularity of $\Gamma$ in this case, see \cite{CGK}. The regularity problem is thus reduced to the one of the degeneracy of the eigenfunction $\varphi$ on its level line $\Gamma$. 

When Dirichlet conditions are imposed on the boundary $\partial \Om$ (in other words, when ``$\beta =+\infty$''), then it has been proved in \cite{CKT}, when $N=2$, that $u\in \mathcal{C}^{1,1}(\Omega)$, that $\partial E$ does not hit the boundary and consists of finitely many disjoint, simple and closed real-analytic curves. 
We believe that the arguments involved in \cite{CKT} could be extended to our framework. Indeed, for Neumann boundary conditions, we expect $\partial E$ to hit the boundary, but most of the arguments of \cite{CKT} are local and do not see the Dirichlet boundary conditions. However, this is not the main topic of the present paper and we will thus leave this question open, since we do not need these results to obtain Theorems \ref{Robin_connected}, \ref{theo:optimball} and \ref{thm3_simple}.
In higher dimensions, it is only known that $\Gamma$ is smooth up to a closed set of Hausdorff dimension $N-1$ \cite{CK}. However, the situation is much more complicated since one could expect, as for some other free boundary problems, the emergence of stable singularities. 



\subsection{Dirichlet boundary conditions}
When Dirichlet boundary conditions are imposed on $\partial \Om$, it is possible to derive qualitative properties on $E^{*}$ from that of $\Om$. Symmetrization techniques apply (\cite{CGIKO}) and allow to show that, if $\Om$ is symmetric and convex with respect to some hyperplane, then so is $E^{*}$. Notice nevertheless that symmetry breaking phenomenon might arise if the convexity property with respect to the hyperplane is not satisfied, for example for annuli or dumbbells \cite{CGIKO}. 

For particular sets of parameters, it has been proved that $\Om\backslash E$ is connected if $\Om$ is simply connected and $E$ is convex if $\Om$ is convex \cite{CGIKO}.

We also mention \cite{HKK} where the authors investigate the related optimization problem of locating an obstacle of given shape, namely a ball, inside a domain $\Omega$ so that the lowest eigenvalue of the Dirichlet-Laplacian operator is minimized. Numerous symmetry results have been derived from the moving plane method. 

It is interesting to note that the Dirichlet case can be recovered by letting the parameter $\beta$ tend to $+\infty$.
\begin{proposition}\label{prop:asympDiri}
Let $\kappa>0$, $m\in L^\infty(\Om,[-1,\kappa])$ such that $\int_\Om m<0$. Let us denote temporarily by $\lambda (\beta,m)$ the principal eigenvalue for Robin boundary conditions defined by \eqref{eq:variational} and by $\lambda_D(m)$ the principal Dirichlet eigenvalue defined by \eqref{eq:variationalD}. The mapping $\R_+\ni \beta\mapsto \lambda (\beta,m)$ is concave, monotone increasing and converges to $\lambda_D(m)$ as $\beta\to +\infty$.
\end{proposition}
\begin{proof}
As an infimum of real affine functions, $\lambda(\cdot,m)$ is concave. Moreover, taking $0<\beta_1<\beta_2$ and comparing the principal eigenvalues thanks to the Rayleigh definition \eqref{eq:variational}  shows that $\lambda(\cdot,m)$ is monotone non-decreasing. 
Let us prove that $\lambda(\cdot,m)$ is moreover increasing. For that purpose, we argue by contradiction and consider $(\lambda,\varphi_1)$ and $(\lambda,\varphi_2)$ two eigenpairs solving \eqref{eq:variational} with respectively $\beta=\beta_1$ and $\beta_2$ and such that $\lambda(\beta_1,m)=\lambda(\beta_2,m)$. Since the mapping $\beta \mapsto \lambda(\beta,m)$ is concave and non-decreasing, we infer that $\beta \mapsto \lambda(\beta,m)$ is constant on $[\beta_1,+\infty)$. Notice that
$$
\lambda(\beta_2,m)=\frac{\int_\Om |\nabla \varphi_2|^2+\beta_2\int_{\partial\Om}\varphi_2^2}{\int_\Om m\varphi_2^2}>\frac{\int_\Om |\nabla \varphi_2|^2+\beta_1\int_{\partial\Om}\varphi_2^2}{\int_\Om m\varphi_2^2}\geq \lambda(\beta_1,m)
$$
whenever $\varphi_2$ does not vanish identically on $\partial\Om$. Hence, it follows that necessarily $\varphi_2=0$ on $\partial\Om$ and $\lambda(\beta_i,m)=\lambda_D(m)$, $i=1,2$. Hence, by simplicity of the principal eigenvalue, one has also $\varphi_1=\varphi_2=0$ on $\partial\Om$. In particular, according to the Robin boundary condition on $\partial\Om$, one also has $\partial _n\varphi_1=\partial_n\varphi_2=0$. One gets a contradiction by observing that the Neumann principal eigenfunction is positive in $\overline{\Omega}$ (see e.g. \cite{MR588690}).

Choosing test functions in $\mathcal{S}(m)\cap H^1_0(\Om)$ in \eqref{eq:variational} proves that $\lambda(\cdot,m)\leq \lambda_D(m)$. As a monotone non-decreasing bounded function, $\lambda (\beta,m)$ has a finite limit as $\beta\to +\infty$. Hence, the family of eigenpairs $\{(\lambda(\beta,m),\varphi_\beta)\}_{\beta>0}$, where $\varphi_\beta$ denotes a solution of  \eqref{eq:variational}, maximizes  $\lambda (\cdot,m)$ as $\beta\to +\infty$. Assuming moreover that $\int_\Om  \varphi_\beta^2=1$ (by homogeneity of the Rayleigh quotient), one has $\int_\Om |\nabla\varphi_\beta|^2+\beta \int_{\partial\Om}\varphi_\beta^2=\lambda(\beta,m)\int_\Om m\varphi_\beta^2$ and therefore $ \int_\Om m\varphi_\beta^2\leq \int_{\Om\cap \{m>0\}} m\varphi_\beta^2\leq \kappa$ showing that the $H^1$-norm of $\varphi_\beta$ is uniformly bounded with respect to $\beta$. Thus, there exists $\varphi_{\infty}\in H^1(\Om)$ with $\int_\Om m{\varphi_{\infty}}^2 =1$ such that, up to a subsequence, $(\varphi_\beta)_{\beta>0}$ converges to $\varphi_{\infty}$ weakly in $H^1$ and strongly in $L^2$, by using the Rellich-Kondrachov Theorem. Writing then
$$
\frac{1}{\beta}\int_\Om |\nabla\varphi_\beta|^2+ \int_{\partial\Om}\varphi_\beta^2=\frac{\lambda(\beta,m)}{\beta}\int_\Om m\varphi_\beta^2,
$$
and letting $\beta$ tend to $+\infty$ shows that  $\int_{\partial\Om}{\varphi_{\infty}}^2=0$, or in other words that $\varphi_{\infty}\in H^1_0(\Om)$. Moreover, by weak convergence of $(\varphi_\beta)_{\beta>0}$ in $H^1(\Omega)$, there holds 
$$
\lambda_D(m)\geq \liminf_{\beta\to +\infty}\lambda(\beta,m)\geq \frac{\int_\Om |\nabla \varphi_\infty|^2}{\int_\Om m \varphi_\infty^2}\geq \lambda_D(m).
$$
Since $\int_\Om m \varphi_\infty^2\geq \int_\Om |\nabla \varphi_\infty|^2$, one has necessarily $\int_\Om m \varphi_\infty^2\neq 0$ and therefore the quotient above is well defined. The expected conclusion follows.

\end{proof}
According to Proposition \ref{prop:asympDiri}, we will consider that the Dirichlet case \eqref{eq:variationalD} corresponds to the choice of parameter $\beta=+\infty$.

\subsection{Periodic boundary conditions}\label{sec:period}

When $\Om = \Pi_{i=1}^{N}(-L_{i},L_{i})$ is embedded with periodic boundary conditions, 
then the optimal set $E^{*}$ is Steiner symmetric, that is, convex and symmetric with respect to all the hyperplanes $\{x_{i}=0\}$ \cite{BHR}. It follows that the restriction of the set $E^{*}$ to $\widetilde{\Om}=\Pi_{i=1}^{N}(0,L_{i})$ is a minimizer for problem \eqref{mini2} for the set $\widetilde{\Om}$ embedded with Neumann boundary conditions \cite{Roques-Hamel}. Hence, there is a bijection between the 
minimization problem for the periodic principal eigenvalue in the square $\Om = \Pi_{i=1}^{N}(-L_{i},L_{i})$ and for the Neumann principal eigenvalue in the restricted square 
$\widetilde{\Om}=\Pi_{i=1}^{N}(0,L_{i})$, and thus, one derives easily corollaries of our results to the periodic framework. 

Moreover, it has been proved that the strip is a local minimizer for certain parameters sets \cite{KaoLouYanagida}, and that the ball is not always a global minimizer \cite{Roques-Hamel}.

We will apply our results to the framework where $\Om$ is a rectangle in dimension $2$ in Section \ref{sec:rectangle}, and prove in particular that $\partial E^{*}\cap \Om$ cannot have a part of its boundary with constant curvature when $\beta=0$. 

\subsection{Numerics}\label{sec:numerics}
As the minimizing set $E^{*}$ is a level set of the principal eigenfunction $\varphi$, thresholding methods based on the so-called {\it bathtub principle} provide very fast algorithms in order to compute $E^{*}$. Indeed, starting with an arbitrary set $E_{k}$ of measure $c|\Om|$ and computing the eigenfunction $\varphi_{k}$ associated with the principal eigenvalue, one then defines recursively $E_{k+1}:= \{ \varphi_{k}>\alpha\}$ where $\alpha$ is a positive number chosen in such a way that $|E_{k+1}|=c|\Om|$ (\cite{CGIKO,MR2886017}), and so on. Note that $\alpha$ is unique since one shows in particular that the level sets $\{\varphi_k=C\}$ have zero Lebesgue measure for every $C>0$. This algorithm converges to a critical point $E$ in a small number of iterations for reasonable parameters.

This method has been used to compute optimal sets for Dirichlet boundary conditions \cite{CGIKO}, Neumann boundary conditions in squares and ellipses \cite{KaoLouYanagida}, Robin boundary conditions in squares \cite{MR2886017}. In general these solutions look like stripes, balls, or complementary of balls, depending on the parameters. However, as already underlined above, very few analytical results confirmed these simulations. In particular, it was not clear whether balls could be minimizing sets or not and we provide a negative answer to this problem in the present paper. We also refer to Section \ref{sec:applications} where we provide numerical investigations and illustrations of our results in the particular cases where $\Omega$ is either the two-dimensional unit square or the unit disc.



\section{The one-dimensional case (Proof of Theorem \ref{Robin_connected})}\label{1d_section}

In \cite{LoYa},  the authors solve the one-dimensional version of Problem \eqref{mini} in the particular case where $\beta=0$. The proof methodology was to first exhibit the solutions when the sets are intervals, and then to show that the optimizers must be intervals. Subsequently, in \cite{JhaPorru}, a much simpler proof of the same result was obtained using increasing or decreasing rearrangements. Here we generalize this result to the case of Robin boundary conditions, in other words for every $\beta\geq 0$.
Throughout this section we will assume without loss of generality that $\Omega=(0,1)$.



The optimization of $\lambda (m_E)$ in the class of intervals has been solved, we cite the following result from
\cite{CantrellCosnerRobin, MR2886017}.  
\begin{proposition}\label{beta_interval}
Take $c\in (0,1)$, $\kappa>0$ and $\beta\in\R_{+}$. Let $\la_{\be}(a):=\lambda (m)$ with $m=\kappa \mathbbm{1}_{[a,a+c]} - \mathbbm{1}_{(0,a)} -\mathbbm{1}_{(a+c,1)}$, in order to highlight the dependence of the eigenvalue on $a\in[0,1-c]$. The function $a\mapsto\la_{\beta}(a)$ is symmetric with respect to $a=(1- c )/2$, and moreover, with $\beta^*$ defined in \eqref{eq:beta*}, we have:
\begin{itemize}
\item if $\be> \be^*$, then $a\mapsto \la_{\be}(a)$ is strictly decreasing on $[0, (1- c )/2]$; in particular its minimum is reached for $a=(1- c )/2$.
\item if $\be< \be^*$, then $a\mapsto\la_{\be}(a)$ is strictly increasing on $[0,(1- c )/2]$; in particular its minimum is reached for $a=0$ and $a=1- c $.
\item if $\be=\be^*$, then $a\mapsto\la_{\be}(a)$ is constant and any $0\leq a\leq 1- c $ is a global minimum for $a\mapsto\la_{\be}(a)$.
\end{itemize}
\end{proposition}
Therefore to complete this result and prove Theorem \ref{Robin_connected}, we need to show that the solution of the optimal design problem \eqref{mini}  has the expression
$$
m=\kappa \mathbbm{1}_{[a,a+c]} - \mathbbm{1}_{(0,a)} -\mathbbm{1}_{(a+c,1)}
$$
for given parameters $a$ and $c$. The next result is devoted to proving this claim.
\begin{proposition}\label{prop:Robin_interval}
If $\Om=(0,1)$, then any optimal set $E^*$ for Problem \eqref{minishape} is an interval.
\end{proposition}

\begin{proof}[Proof of Proposition \ref{prop:Robin_interval}.]
Assume in a first step that $\beta>0$.
As recalled in the sections \ref{sect:intro} and \ref{ssect:optimality}, we know there exists $E^*$ solution of \eqref{minishape}. We denote by $m^*=\kappa\mathbbm{1}_{E^*}-\mathbbm{1}_{\Om\backslash E^*}$ the associated weight and $\varphi$ the corresponding eigenfunction, solution of \eqref{eq:EV_main}. The Rayleigh quotient $\Re_{m^*}$ defined by \eqref{def:Re} rewrites in this case
$$
\Re_{m^*} [\varphi]  :=\frac{\int_0^1\varphi'^2+\beta \varphi^2(0)+\beta \varphi^2(1)}{\int_0^1 m^*\varphi^2},
$$
one has $\la(m^*)=\Re_{m^*}[\varphi]$. Since $-\varphi'' = \lambda m^* \varphi \in L^\infty(0,1)$ we have $\varphi'\in W^{1,\infty}(0,1)$ and $\varphi'\in C^0([0,1])$. We  also have $\varphi'(0) =-\dn \varphi(0)= \beta \varphi(0) > 0 $ and $\varphi'(1) =\dn \varphi(1)= -\beta \varphi(1) < 0 $, so $\varphi$ reaches his maximum inside $(0,1)$. 
Let 
$$
\alpha :=  \min \{\xi\in(0,1) \ |\ 
\varphi(\xi) =\| \varphi\|_{\infty} \}\in(0,1).
$$

%

Introduce the function 
$\varphi^R$ defined on $(0,1)$ by
$$
\varphi^R(x)=\left\{\begin{array}{ll}
\varphi^\nearrow(x) & \textrm{on }(0,\alpha),\\
\varphi^\searrow(x) & \textrm{on }(\alpha,1),
\end{array}
\right.
$$
where $\varphi^{\nearrow}$ denotes the monotone increasing rearrangement of $\varphi$ on $(0,\alpha)$ and 
$\varphi^\searrow$ denotes the monotone decreasing rearrangement of $\varphi$ on $(\alpha,1)$ (see for instance \cite{MR2455723}). Thanks to the choice of $\alpha$, it is clear that this symmetrization does not introduce discontinuities, and more precisely that $\varphi^R\in H^1(0,1)$.\\
Similarly, we also introduce the rearranged weight $m^R$, defined by 
$$
 m^R(x)=\left\{\begin{array}{ll}
m^\nearrow(x) & \textrm{on }(0,\alpha),\\
m^\searrow(x) & \textrm{on }(\alpha,1),
\end{array}
\right.
$$
with the same notations as previously. In other words, $m^\nearrow$ (resp. $m^\searrow$) is {\it bang-bang}, equal to $-1$ or $\kappa$ almost everywhere, such that $|\{m^\nearrow=\kappa\}\cap (0,\alpha)|=|\{m^*=\kappa\}\cap (0,\alpha)|$ and $|\{m^\nearrow=-1\}\cap (0,\alpha)|=|\{m^*=-1\}\cap (0,\alpha)|$ (resp. $|\{m^\searrow=\kappa\}\cap (\alpha,1)|=|\{m^*=-1\}\cap (\alpha,1)|$ and $|\{m^\searrow=-1\}\cap (\alpha,1)|=|\{m^*=-1\}\cap (\alpha,1)|$).

We aim at proving now that $m^R$ is admissible for the optimal design problem \eqref{mini}, that $\varphi^R$ is an admissible test function of the Rayleigh quotient $\Re_{m^R}$, and that the Rayleigh quotient decreases for the symmetrization, in other words that $\Re_{m^R}[\varphi^R] \leq\Re_{m^*}[\varphi]$. 

First it is clear that $m^R\in\mathcal{M}_{m_{0},\kappa}$. Indeed, every integral on $(0,1)$ can be written as the sum of the integral on $(0,\alpha)$ and $(\alpha,1)$, and we use the equimeasurability property of monotone symmetrizations on each of these intervals. Then, writing $\int_{0}^1 m\varphi^2=\int_{0}^1(m+1)\varphi^2-\int_{0}^1\varphi^2$ and using Hardy-Littlewood inequality (again on $(0,\alpha)$ and $(\alpha,1)$), we also have $\int_{0}^1m^R(\varphi^R)^2\geq \int_{0}^1 m\varphi^2>0$, and the expected conclusion follows. 

Also, we easily see that 
$$(\varphi^R)^2(0)=\min_{[0,\alpha]}\varphi^2\leq \varphi^2(0)\;\;\;\textrm{ and }\;\;\;(\varphi^R)^2(1)=\min_{[\alpha,1]}\varphi^2\leq \varphi^2(1).$$
Using now Poly\`a's inequality twice: $\int_0^\alpha ({\varphi^R}')^2\leq \int_0^\alpha {\varphi}'^2$ and $\int_\alpha^{1} ({\varphi^R}')^2\leq \int_\alpha^{1} {\varphi}'^2$, we obtain that
\begin{eqnarray*}
\Re_{m^R}[ \varphi^R]&\leq &\Re_{m^*}[\varphi]\\
&=& \lambda(m^*)=\inf_{m\in\mathcal{M}_{m_0,\kappa}}\inf\left\{\Re_{m}(\psi), \psi\in H^1(0,1)\textrm{ such that }\int_{0}^1m\psi^2>0\right\}\\
&\leq &\Re_{m^R}[ \varphi^R].
\end{eqnarray*}
Investigating the equality case of Poly\`a's inequality, it follows that $\varphi$ first increases up to its maximal value and then decreases on $(\alpha,1)$ (see for example \cite{BLR} and references therein). As a consequence, since $E^*$ is a upper level-set of $\varphi$, it is necessarily an interval, which concludes the proof. 

Finally, the case $\beta=0$ is simpler. A direct adaptation of this proof shows that the claim remains valid in that case.


\end{proof}
\section{Non-optimality of the ball for Problem 
\texorpdfstring{\eqref{mini}}{Lg} 
(Proofs of Theorem~\texorpdfstring{\ref{theo:optimball}}{Lg} 
and Proposition~\texorpdfstring{\ref{prop:geomprop}}{Lg})}\label{geom_prop}

In this section, we investigate the optimality of a ball, or more generally of rotationally symmetric sets (i.e. a union of concentric rings). This question naturally arises, in particular according to numerical results in \cite{MR2886017} for dimension $N=2$, where for some values of the parameters, the solutions seem to take the shape of a ball; also in the case of periodic boundary condition, H. Berestycki stated that the solution might be a ball, for some values of the volume constraint \cite{berest}. We prove that for every $\beta\in [0,+\infty]$ and  $c\in(0,1)$, the ball is not optimal, except possibly if $\Om$ itself is a ball having the same center as $E$.



Let us first assume that $m_{E}$ is a critical point of the optimal design problem \eqref{mini}.
We recall that ``$m_{E}$ is a critical point of the optimal design problem \eqref{mini}'' means that $m_E$ satisfies the necessary first order optimality conditions (see Section \ref{ssect:optimality}), in other words that the associated principal eigenfunction $\varphi$ is constant on $\partial E \backslash \Om$.

\medskip

\noindent {\bf Part 1: the function $\varphi$ is radially symmetric.} Assume that $E$ is rotationally symmetric, and is a critical point.  In what follows, we will denote by  
$\lambda$ the principal eigenvalue $\lambda (m_{E})$ and by $\varphi$ the associated eigenfunction that solves \eqref{eq:EV_main}. 
We follow the following steps: we first prove that $\varphi$ is radial in $E$, then is also radial in $\Om\backslash E$, and we conclude that $\Om$ must be a centered ball.
Generalizing the methods used in \cite{HenrotOudet,Henrot-Privat}, take $i,j\in \llbracket1,N\rrbracket$ with $i\neq j$, and define 
\begin{equation} \label{def:vradiale}
 v_{ij}:= x_i\partial_{x_j} \varphi - x_j \partial_{x_i} \varphi \quad \hbox{ in } \Om.
\end{equation}
This function lies in $W^{1,p}_{loc}(\Om)$ 
for all $p\in (1,+\infty)$ and a straightforward computation yields that $v_{ij}$ verifies in the sense of distributions the partial differential equations
\begin{align*}
-\Delta v_{ij} &= \lambda \kappa v_{ij} \quad  \hbox{ in } E  \\
 \hbox{ and } \quad \Delta v_{ij} &= \lambda v_{ij} \quad  \hbox{ in } \Om \backslash \overline{E}.
\end{align*}
Let us now prove that $v_{ij}$ vanishes in $\Om$. 

Since $m_E$ is a critical point of Problem \eqref{mini} and since $\varphi$ solves \eqref{eq:EV_main} in a variational sense, there exists a real number $\alpha>0$ such that $E=\{\varphi>\alpha\}$ up to a set of measure $0$, and therefore using continuity of $\varphi$, we obtain
$$
\varphi=\alpha \quad\textrm{on }\partial E\backslash \partial\Om.$$

As a consequence and according to \eqref{def:vradiale}, since the set $E$ is rotationally symmetric, the function $v_{ij}$ vanishes on $\partial E\backslash \partial\Om$. We now assume $\beta<\infty$ and leave the case $\beta=+\infty$ at the end of this part.
Moreover, since  $E$ is rotationally symmetric, a straightforward computation shows that, 
\begin{eqnarray}
|x|\partial_n v_{ij} &=& \sum_{k=1}^N x_k\partial_{x_k}v_{ij}  =(x_i\partial_{x_j}  - x_j \partial_{x_i})\left(\sum_{k=1}^N x_k\partial_{x_k}\right)\varphi \nonumber \\
&=& -\beta |x| (x_i\partial_{x_j}\varphi  - x_j \partial_{x_i}\varphi)= -\beta |x| v_{ij} \qquad \textrm{on } \partial E \cap \partial \Om.\label{eq:bordv}
\end{eqnarray}
for all $x\in \partial E \cap \partial \Om$. Therefore, the function $v_{ij}$ solves in a variational sense the partial differential equation
\begin{equation}\label{eq:varray}
\left\{ \begin{array}{rll} -\Delta v_{ij} &= \lambda \kappa v_{ij} \quad &\hbox{ in } E ,\\
\partial_n v_{ij} + \beta v_{ij} &=0 \quad &\hbox{ on } \partial E \cap \partial \Om,\\
           v_{ij} &=0 \quad &\hbox{ on } \partial E\backslash\partial \Om. 
          \end{array}\right.
\end{equation}
\begin{figure}[!ht]
\begin{center}
\includegraphics[height=6cm]{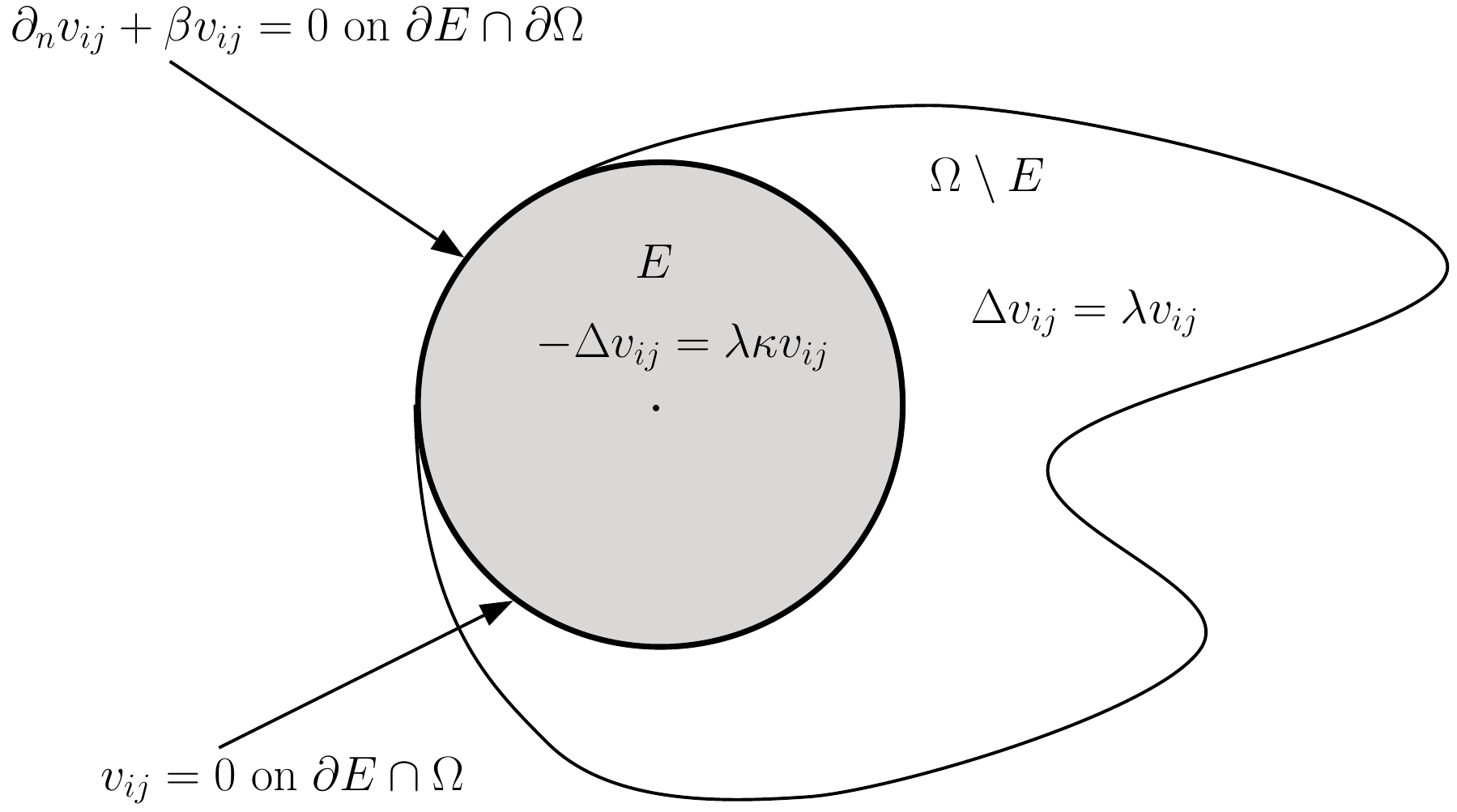}
\caption{The PDE solved by $v_{ij}$.}
\end{center}
\end{figure}

On the other hand, according to the minimax Courant-Fischer principle, there holds
\begin{equation} \label{eq:Rayleigh}
 \lambda = \min_{\substack{\psi \in H^1(\Om)\\ \int_{\Om}m_{E}\psi^2>0}} \Re _{m_E}[\varphi]
\end{equation}
where the Rayleigh quotient $\Re _{m_E}$ is defined by \eqref{def:Re} and this minimum is reached only by the multiples of $\varphi$. 
Assume by contradiction that $v_{ij}$ does not vanish identically in $\overline{E}$, then we can take as a test function
$$
\widetilde{v_{ij}}:= \left\{ 
\begin{array}{lcl} 
v_{ij} &\hbox{ in }& E ,\\
0 &\hbox{ in }& \Om\backslash E .
\end{array}\right.
$$
 in the Rayleigh quotient $\Re _{m_E}$. The function $\widetilde{v_{ij}}$ belongs to $H^1(\Om)$ since $v_{ij}=0$ on $\partial E\backslash\partial\Om$, satisfies 
 $$
 \int_{\Om}m_{E}\widetilde{v_{ij}}^2=\kappa \int_{E}{v_{ij}}^2>0
 $$ 
 and is not a multiple of $\varphi$ since $\widetilde{v_{ij}}= 0$ 
in $ \Om\backslash E$. As a consequence, one has
$$
\lambda <  \frac{\int_\Om |\nabla \widetilde{v_{ij}}|^2 +\beta \int_{\partial\Om} \widetilde{v_{ij}}^2}{ \int_\Om m_{E}\widetilde{v_{ij}}^2}
=\frac{  \int_{E} |\nabla v_{ij}|^2 +\beta \int_{\partial E\cap \partial\Omega}v_{ij}^2}{ \int_{E} m_{E}v_{ij}^2}=\lambda,
$$
the last equality following from an integration by parts in \eqref{eq:varray}. This contradiction yields that $v_{ij}\equiv 0$ in $E$. 

Hence $x_i\partial_{x_j} \varphi \equiv x_j \partial_{x_i} \varphi$ for all $i\neq j$, which implies  
that $\varphi$ is radially symmetric inside $E$.
In other words, there exists a function $U$ such that $\ph(x) = U(|x|)$, for all $x\in E$. 

Notice moreover that $\partial_n \varphi = \pm U'(|x|)$ 
on $\partial E$, where $n$ stands for the outward normal to $E$. Therefore, the function $\sum_{k=1}^N x_k\partial_{x_k}\varphi$ which, up to some multiplicative constants, is equal to $\partial_{n}\varphi$ on each connected component of $\partial E$, is constant on each connected component of $\partial E$.

Let us now prove that the function $\varphi$ is in fact radially symmetric on the whole domain $\Omega$. For that purpose, let us show that for every $i,j\in \llbracket1,N\rrbracket$ with $i\neq j$, the function $v_{ij}$ also vanishes in $\Om\backslash \overline{E}$. Similar computations as in \eqref{eq:bordv} lead to 
$\partial_n v_{ij} \equiv 0$ on $\partial E$. 
Hence, $v_{ij}$ satisfies the following overdetermined partial differential equation
\begin{equation}
 \label{eq:varray2}\left\{ \begin{array}{rll} 
 \Delta v_{ij}  &= \lambda v_{ij} \quad &\hbox{ in } \Om\backslash \overline{E},\\
           v_{ij} &=0 \quad &\hbox{ on } \partial (\Om\backslash \overline{E})\backslash \partial\Om  =\partial E\cap\Om,\\ 
           \partial_n v_{ij} &=0 \quad &\hbox{ on } \partial (\Om\backslash \overline{E})\backslash \partial\Om =\partial E\cap\Om.\\ 
          \end{array}\right.
\end{equation}
Moreover, $\partial E\cap \Omega$ is analytic. Indeed, as $E$ is rotationally symmetric and has a finite number of connected components, it is a finite union of rings. It thus satisfies an interior sphere condition and the Hopf Lemma thus gives $\nabla \varphi \neq 0$ on $\partial E \cap \Om$. The implicit function theorem thus yields that this boundary is $C^{1,1}$ since $E$ is a level set of $\varphi$, which is $W^{2,p}$ for all $p>1$, and the conclusion follows from a bootstrap argument (see \cite{CGK}).

The Cauchy-Kowalevski theorem yields that $v_{ij}=0$ in a neighborhood of $\partial E\cap\Om$ in $\Om\backslash \overline{E}$. Moreover, using the hypoellipticity of the Laplacian operator (see e.g. \cite{nelson}), we claim that $v_{ij}$ is analytic (in $\Om\backslash\overline{E}$), implying that $v_{ij}\equiv 0$ in $\Om\backslash \overline{E}$ and then in the whole domain $\Om$. 
Since it is true for all $i\neq j$, this means that $\varphi$ is radially symmetric over the full domain $\Om$ and we can define $U(r)=\varphi(x)$ where $r=|x|$ for some $x\in \Om$ and $r\in[0,b)$ where $b=\max\{|x|, x\in\Om\}$.

In the case $\beta=+\infty$, the set $\partial\Om\cap\partial E$ is empty as $\varphi$ vanishes on $\partial\Om$ and equals $\alpha$ on $\partial E$, so the proof is similar by just dropping any consideration involving the set $\partial\Om\cap\partial E$.


\medskip 

\noindent {\bf Part 2: $\Omega$ is necessarily a centered ball.} 
In the case $\beta=+\infty$, this part is easy as $\Omega$ is a level set of $\varphi$. We therefore focus on the case $\beta<\infty$.
Assume by contradiction that $\Om$ is not a centered ball.
Let $B_a$ and $B_b$ the largest and, respectively, the smallest open balls centered at $O$ such that $B_a \subset \Om\subset B_b$ (see Figure \ref{fig:dessin21721}), with $0<a<b$. Notice that the assumption that $\partial \Om$ be connected guarantees that $\Om$ contains the origin. The existence of $B_a$ and $B_b$ is then a consequence of the boundedness of $\Om$ combined with the main assumption of Theorem \ref{theo:optimball}, namely that $E$ is rotationally symmetric and centered at $O$.

\begin{figure}[!ht]
\begin{center}
\includegraphics[height=8cm]{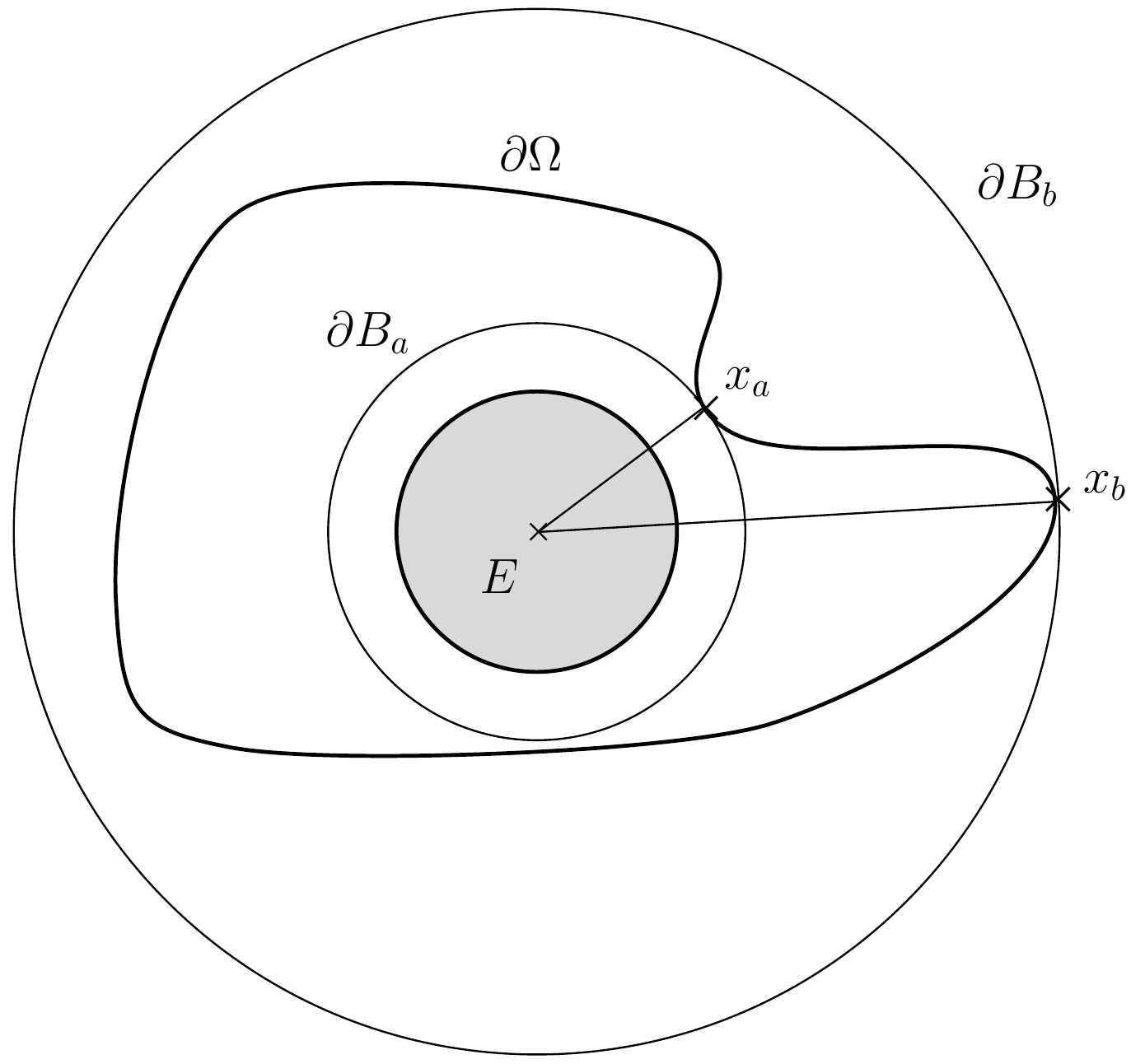}
\caption{The set $\Om$, the two balls $B_a$ and $B_b$.}\label{fig:dessin21721}
\end{center}
\end{figure}

 There exist $x_a, x_b \in \partial \Om$ such that $|x_a|=a$ and $|x_b|=b$. The definitions of $B_a$ and $B_b$ yield that the exterior normals of $\Om$ at $x_a$ and $x_b$ are respectively 
$x_a/|x_a|$ and  $x_b/|x_b|$. Hence, 
\begin{eqnarray*}
\partial_n \varphi (x_a)+\beta \varphi (x_a)& =& U'(a)+\beta U (a)=0,\\
\textrm{and }\quad \partial_n \varphi (x_b)+\beta \varphi (x_b)&=&U'(b)+\beta U (b)=0.
\end{eqnarray*}

In the sequel, by analogy with polar variables, we will denote $r=|x|$. 
For every $c\in (a,b)$, let us denote by $B_c$ the centered ball of radius $c$, and let us associate an element $x_c\in \partial B_c\cap\partial\Omega$ such that $|x_c|=c$. Introduce $\mathcal{N}(x)=x/|x|$ and note that  if $x\in\partial B_c$, $\mathcal{N}(x)$ is the outer normal vector to $B_c$. Moreover, since $\varphi$ is radially symmetric, there holds $\nabla \varphi (x)=U'(|x|)\mathcal{N}(x)$ for every $x\in B_b$.

In what follows, we will treat separately the cases ``$\beta=0$'' and ``$\beta >0$'' for the sake of clarity.

Let us first consider the case of Neumann boundary conditions, in other words the case ``$\beta=0$''. Note that  $\mathcal{N}(x_a)\cdot n(x_a) =1$ so by continuity there exists an interval $[a,a+\delta]$ such that $\mathcal{N}(x_c)\cdot n(x_c) >0$  for every $c\in [a,a+\delta]$. Writing   $0=\dn\varphi(x_c) = U'(c) \mathcal{N}(x_c)\cdot n(x_c)$  leads to $U'(c)=0$ for all $c\in [a,a+\delta]$ since $\mathcal{N}(x_c)\cdot n(x_c) >0$. Thus $U(r)$ is a positive constant  and $U''(r)=0$ on $(a,a+\delta)$. This leads to a contradiction with the equation 
\begin{equation}\label{eq:Ur}
U''(r) + \frac{N-1}{r}U'(r) = \lambda U(r)
\end{equation}
satisfied by $U$ according to \eqref{eq:EV_main}, and the fact that $U(r)>0$ for every $r\in (a,b)$. As a consequence, with Neumann boundary conditions,  we must have $a=b$ which shows that $\Omega$ is a disk.

Let us now investigate the general case of Robin boundary conditions, in other words the case ``$\beta>0$''. We first prove that $\mathcal{N}(x_c)\cdot n(x_c) > 0$, where $n(x_c)$ is the outer normal vector to $\Omega$ at $x_c$. Assume by contradiction that $\mathcal{N}(x_c)\cdot n(x_c)\leq 0$ for some $x_c$. Since $\Omega$ and $B_a$ are tangent at $x_a$ we have $\mathcal{N}(x_a)\cdot n(x_a) = 1$ and since $\Omega$ is of class $\mathcal{C}^1$, we have $n(\cdot)\in \mathcal{C}^0(\partial\Omega)$. Therefore, by continuity, there exist $d\in (a,b)$ and $x_d\in \partial B_d\cap\partial\Omega$ with $\mathcal{N}(x_d)\cdot n(x_d)= 0$. Writing the Robin boundary condition at $x_d$, one gets 
$$ 
-\beta \varphi(x_d)=\dn\varphi(x_d) = U'(r_d) \mathcal{N}(x_d)\cdot n(x_d) =0,
$$
which is impossible since $\varphi(x_d)>0$ and $\beta>0$. It follows that $\mathcal{N}(x)\cdot n(x) > 0$ for all $x\in\partial\Omega$.

Let us now introduce the function $V$ defined by 
$$
V:[a,b]\ni r\mapsto -\frac{U'(r)}{ U(r)}.
$$ 


Note that $V$ is well-defined since $U(r)$ is positive for every $r\in (a,b)$. Rewriting the Robin boundary condition in terms of the function $U$ yields
$$ -\beta U(r)= -\beta \varphi(x)=\dn\varphi(x) = U'(r) \mathcal{N}(x)\cdot n(x), $$
for every $r\in (a,b)$ and $x$ such that $|x|=r$. Therefore, there holds
\begin{align} 
\label{V1} V(r) & := -\frac{U'(r)}{ U(r)} =\frac{\beta}{\mathcal{N}(x)\cdot n(x)}\geq \beta.\\
\label{V2}  V(a) &= V(b) =\beta.
\end{align}

We will reach a contradiction by exhibiting $r^*\in (a,b)$ such that $V(r^*)<\beta$. To prove this, we will investigate the sign of the derivatives of $V$ at $r=a$ or $r=b$. According to \eqref{V1} and \eqref{V2},  the function $V$ is non decreasing at $r=a$ (resp. $V$ is non increasing at $r=b$), otherwise one could find $r^*$ in the neighborhoods of $r=a$ (resp. $r=b$) with $V(r^*)<\beta$. 

The derivative of $V$ writes
$$
V'(r) = -\frac{U''(r)}{U(r)} +\frac{U'(r)^2}{U(r)^2} = -\lambda +\frac{(N-1)U'(r)}{rU(r)} + \frac{U'(r)^2}{U(r)^2} = -\lambda - \frac{N-1}{r}V(r) + V(r)^2.
$$
by using that the function $U$ solves \eqref{eq:Ur}. Moreover, the boundary conditions \eqref{V2} yields
\begin{align*} 
V'(a) &= -\lambda - \frac{N-1}{a}V(a) +V(a)^2  = -\lambda -\frac{(N-1)\beta}{a} +\beta^2,\\
V'(b) &= -\lambda - \frac{N-1}{b}V(b)+ V(b)^2 = -\lambda - \frac{(N-1)\beta}{b} +\beta^2.
\end{align*}
Since the two zeros of the polynomial $P_a(X)=-\lambda - (N-1)X/a + X^2$ are
$$X_a^{\pm} = \frac{N-1}{2 a} \pm \sqrt{\frac{(N-1)^2}{(2 a)^2} + \lambda},$$
it follows that $V'(a) < 0$ whenever $\beta\in (0,X_a^+)$, which would be a contradiction in view of the above discussion.
 

In a similar way, the two zeros of the polynomial $P_b(X)=-\lambda - (N-1)X/b + X^2$ are
$$
 X_b^{\pm} = \frac{N-1}{2 b} \pm \sqrt{\frac{(N-1)^2}{(2 b)^2} + \lambda}.$$
Hence, it follows that $V'(b) > 0$ whenever $\beta \in (X_b^+,+\infty)$, which would also be a contradiction.


Moreover, since $b > a$, one has $X_b^+<X^+_a$, so that whatever the value of $\beta > 0$, we can arrive at one of the two contradictions above or even both. Eventually, we have reached a contradiction which implies $a=b$ and the domain $\Omega$ is necessarily a ball.

\medskip

\medskip


\noindent{\bf Part 3: if $E$ is rotationally symmetric, and also a minimizer of $\lambda$, then $E$ is necessarily a centered ball.}
Let us assume now that $\Omega$ is equal to $B_R$, $E$ is rotationally symmetric, and it is not only a critical point but also a minimizer. Within this part, we will denote by $m_\beta$ a minimizer for the shape optimization problem
\begin{equation}\label{pbmini752}
\min \{\lambda(\beta,m)\ |\ m\in \mathcal{M}_{m_0,\kappa}\textrm{ and }\ m\textrm{ is radially symmetric}\},
\end{equation}
by $\lambda(\beta,m_\beta)$ the optimal eigenvalue and by $\varphi_\beta$ the associated principal eigenfunction. We will also assume that $\int_\Om \varphi_\beta^2\, dx=1$ by homogeneity of the Rayleigh quotient.
According to the two previous steps, we know that the function $\varphi_\beta$ is radially symmetric since $\Om$ is a centered ball. 

It is notable that in the case where a solution $m$ of Problem \eqref{mini} is radially symmetric, it also solves Problem \eqref{pbmini752}. As a consequence, our claim is equivalent to showing that for $\beta$ large enough, the solution of Problem \eqref{pbmini752} writes
$$
m_\beta=(\kappa+1)\chi_E-1,
$$
where $E$ denotes a centered ball with radius $r_0$ such that $\int_\Om m_\beta=-m_0|\Om|$.

In the sequel, we will use that the family $(\varphi_\beta)_{\beta>0}$ converges up to a subsequence to the function $\varphi_\infty$ weakly in $H^1(\Om)$ and strongly in $L^2(\Om)$, as $\beta \to +\infty$, where $\varphi_\infty$ is the eigenfunction associated to $\lambda(\infty,m_\infty)$ (principal eigenfunction associated to the solution of Problem \eqref{mini} in the Dirichlet case). This fact is easy to see by slightly adapting the proof of Proposition \ref{prop:asympDiri}. Moreover, by using a standard rearrangement argument involving the Schwarz symmetrization, the function $\varphi_\infty$ is radial decreasing and
$$
m_\infty=(\kappa+1) \chi_{\{|x|< r_0\}}-1,
$$
where $r_0$ is uniquely determined by the condition $\int_\Om m_\infty=-m_0|\Om|$.

In the sequel, the precise knowledge of $m_\infty$ and $\varphi_\infty$ is at the heart of its proof. More precisely, we will use the two following facts that stem obviously from the fact that $\varphi_\infty$ is invariant by the Schwarz symmetrization:
\begin{enumerate}
\item Set $r=|x|$. There exists a monotone decreasing differentiable function $U_\infty$ such that
$$
\varphi_\infty(x)=U_\infty(r)\qquad \textrm{ for a.e. }x\in \Om.
$$
\item Fix $\varepsilon\in (0,R)$. There exists $c_\infty>0$ such that
$U'_\infty(r)\leq -c_\infty$ for a.e. $r\in [\varepsilon,R]$.
\end{enumerate}

\medskip

To prove the expected result, we need the following lemma.

\begin{lemma}\label{lem1958}
There exists $\beta_0\geq 0$ such that
$$
\beta\geq \beta_0\quad \Longrightarrow \quad \min_{x\in \overline{\Omega}}\varphi_\beta(x)={\varphi_\beta}_{\mid {\partial\Omega}}. 
$$
\end{lemma}
\begin{proof}
Within this proof, we will write similarly with a slight abuse of notation a sequence and any subsequence, for the sake of simplicity.

Set $\tilde \varphi_\beta=\varphi_\beta-{\varphi_{\beta_n}}_{\mid {\partial\Omega}} $. Since $\varphi_\beta$ solves Equation \eqref{eq:EV_main}, one has
\begin{eqnarray*}
\Vert \Delta \tilde\varphi_\beta\Vert_{L^2(\Omega)} &=& \Vert \Delta \varphi_\beta\Vert_{L^2(\Omega)}=\lambda(\beta,m_\beta) \Vert m_\beta\varphi_\beta\Vert_{L^2(\Omega)}\\
&\leq & \lambda(\beta,m_\beta) \Vert m_\beta\Vert_{L^\infty(\Omega)}\Vert \varphi_\beta\Vert_{L^2(\Om)}\leq \max \{1,\kappa\}\lambda(\infty,m_\infty)
\end{eqnarray*}
by using the fact that $\varphi_\beta$ is $L^2$-normalized. Indeed, recall that, according to Proposition \ref{prop:asympDiri}, the family $(\lambda(\beta,m_\beta))_{\beta\geq 0}$ is non-decreasing and converges to $\lambda(\infty,m_\infty)$, in other words the optimal value for Problem \eqref{mini} where Dirichlet boundary conditions are imposed on the partial differential equation \eqref{eq:EV_main}. Since the boundary of $\Omega$ is smooth, the norms $\Vert \cdot\Vert_{H^2(\Om)}$ and 
$\Vert \Delta \cdot\Vert_{L^2(\Om)}$ are equivalent in $H^2(\Om)\cap H^1_0(\Om)$. It follows that the family $(\Vert \tilde \varphi_\beta\Vert_{H^2(\Om)})_{\beta\geq 0}$ is bounded. Moreover, since the Rayleigh quotient $(\Re _{m_\beta}[\varphi_\beta])_{\beta>0}$ is bounded, the sequence of real numbers $({\varphi_{\beta_n}}_{\mid {\partial\Omega}})_{\beta>0}$ is also bounded. We thus easily infer that the sequence $(\Vert \varphi_\beta\Vert_{H^2(\Om)})_{\beta\geq 0}$ is bounded.
 
 It follows from classical bootstrap arguments that 
$(\varphi_{\beta_n})_{n\in\mathds{N}}$ converges to $\varphi_\infty$ in $W^{2,p} (\Om)$ for all $p\in (1,\infty)$ and in $C^{1,\alpha}(\overline{\Om})$ for all $\alpha \in (0,1)$ along a subsequence as $n\to +\infty$. 

Now, let us write $\varphi_\beta (x)=U_\beta(r)$, with $r=|x|\in (0,R]$. 
Recall that there exists $c_\infty> 0$ such that $U'_\infty(r)\leq -c_\infty$ for every $r\in [\varepsilon,R]$. The convergence in $C^{1,\alpha}(\overline{\Om})$ yields:
$$ \sup_{r\in [\varepsilon,R]} U_{\beta_n}'(r)\leq -\frac{c_\infty}{2}<0, \hbox{ when $n$ is large enough.}$$
Similarly, fix $\varepsilon\in (0,R)$ and let $\eta>0$ be such that
$$
\min_{x\in B(0,\varepsilon)}\varphi_\infty>2\eta . 
$$
One has for large $n$:
\begin{equation}\label{eq0902}
\min_{x\in B(0,\varepsilon)}\varphi_{\beta_n}\geq \min_{x\in B(0,\varepsilon)}\varphi_\infty -\eta>\eta.
\end{equation}

Finally, since $U'_{\beta_n}(R)=-\beta_n U_{\beta_n}(R)$ for every $n\in \mathds{N}$, one has
\begin{eqnarray*}
|U_{\beta_n}(R)| & = & \frac{\lambda({\beta_n},m_{\beta_n})}{{\beta_n}}\left|\int_0^Rs^{N-1}\hat m_{\beta_n}U_{\beta_n} (s)\, ds\right| \\
&\leq & \frac{R^{N/2}\max\{1,\kappa\} \lambda(\infty,m_\infty)}{N^{1/2}\beta_n}\left(\int_0^R U_{\beta_n}(s)^2\, s^{N-1}\, ds\right)^{1/2}\\
&= & \frac{R^{N/2}\max\{1,\kappa\} \lambda(\infty,m_\infty)}{N^{1/2}\beta_n}.
\end{eqnarray*}

It follows that $U_{\beta_n}(R)=\operatorname{O}\left(1/\beta_n\right)$. As a consequence, one has\begin{equation}\label{eq0904}
{\varphi_{\beta_n}}_{\mid \{|x|=R\}}<\min_{x\in B(0,\varepsilon)}\varphi_\infty -\eta
\end{equation}
when $n$ is large enough,
and using the fact that $U_{\beta_n}$ is decreasing on $[\varepsilon,R]$, we get
\begin{equation}\label{eq0905}
{\varphi_{\beta_n}}_{\mid \{|x|=R\}}= \min_{x\in B_R}\varphi_{\beta_n}.
\end{equation}
The desired result follows.

\end{proof}

 Let us now prove that $E$ is a centered ball. From now on, the parameter $\beta$ is assumed fixed and such that $\beta\geq \beta_0$, where $\beta_0$ is the real number defined in Lemma \ref{lem1958}.
 
 We set $\varphi=\varphi_\beta$, $m=m_\beta$ and $\lambda=\lambda(\beta,m_\beta)$. According to Lemma \ref{lem1958}, there holds
$$
\min_{x\in \overline{\Omega}}\varphi(x)={\varphi}_{\mid {\partial\Omega}}.
$$
 
We consider $\varphi^S $ (resp. $m_E^S$) the Schwarz rearrangement of $\varphi$ (resp. $m_E$), and $E^S$ the centered ball of volume $|E|$. We prove that every term is the Rayleigh quotient will decrease with this rearrangement, and that $\varphi^S$ is admissible for the variational formulation of $\lambda(m_{E^S})$. Note that as $\varphi$ is radial, the Schwarz rearrangement can be seen as the monotone decreasing rearrangement of $U$ defined as $U(r)=\varphi(x)$ where $r=|x|$ and $x	\in B_{R}$. First, according to the Poly\`a-Szego inequality, there holds
$$
\int_{B_{R}}|\nabla\varphi|^2dx\geq \int_{B_{R}}|\nabla\varphi^S|^2dx,
$$
This inequality is valid for functions in $H^1_{0}(B_{R})$ but not, in general, for functions in $H^1(\Om)$. Nevertheless, being radial, $\varphi$ is constant on $\partial B_{R}$, and according to Lemma \ref{lem1958} the function $\varphi-\varphi_{|\partial B_{R}}$ belongs to $H^1_{0}(B_{R})$. The inequality above remains then valid in that case by considering $\varphi-\varphi_{|\partial B_{R}}$ instead of $\varphi$. The boundary term satisfies
$$
\int_{\partial B_{R}}\varphi^2\, dx\geq \int_{\partial B_{R}}(\varphi^S)^2\, dx
$$
since $\varphi$ is radial and $\varphi^S_{|\partial B_{R}}=\min_{B_{R}} \varphi$.
Using also the Hardy-Littlewood inequality (which does not require a boundary hypothesis, though it requires a sign condition, which can be overcome as in the proof of Theorem \ref{Robin_connected} by writing $m\varphi^2=(m+1)\varphi^2-\varphi^2$, we finally obtain
$$ \lambda \geq  \frac{\int_{B_{R}}|\nabla\varphi^S|^2dx+\beta \int_{\partial B_{R}}(\varphi^S)^2}{\int_{B_{R}} m_{E}^S(\varphi^S)^2} \geq \lambda (m_{E^S}),$$
the last inequality following from the Courant-Fisher principle, the fact that $\varphi^S$ is admissible in the formulation of $\lambda(m_{E^S})$ since $\int_{B_{R}}m_{E}^S(\varphi^S)^2\geq \int_{B_{R}}m_{E}\varphi^2>0$ and the fact that $m_{E}^S=m_{E^S}$.
But since $E$ is a minimizing set, $\lambda (m_{E^S}) \geq \lambda (m_E)= \lambda$. Hence all the previous inequalities are equalities, which is possible if and only if $U$ and $m_E$ are decreasing. It implies in particular that $E$ is a ball.

\medskip 

\noindent{\bf Part 4: case where $\Om\setminus E$ is rotationally symmetric.} Let us assume now that $E^c=\Om\backslash \overline{E}$ is rotationally symmetric. Then similar results occur. We do not give all details since the proof is then very similar to the one written previously.  We only underline the slight differences in every step. 
\begin{itemize}
\item {\bf Part 1:} introducing the function $v_{ij}$ defined by \eqref{def:vradiale}, one shows using the same computations and the fact that $\varphi$ is constant on $\partial E^c\cap\Om$ that for every $i\neq j$, $v_{ij}$ solves the partial differential equation
\begin{equation}\label{eq:varray3}
\left\{ \begin{array}{rll} \Delta v_{ij} &= \lambda v_{ij} \quad &\hbox{ in } E^c ,\\
\partial_n v_{ij} + \beta v_{ij} &=0 \quad &\hbox{ on } \partial E^c \cap \partial \Om,\\
           v_{ij} &=0 \quad &\hbox{ on } \partial E^c\backslash\partial \Om. 
          \end{array}\right.
\end{equation} 
Multiplying the main equation by $v_{ij}$ and integrating by parts leads to
$$
\lambda \int_{E^c}v_{ij}^2=-\int_{E^c}|\nabla v_{ij}|^2-\beta \int_{\partial E^c\cap \partial \Omega}v_{ij}^2.
$$
It follows that $v_{ij}$ vanishes in $E^c$ and that $\varphi$ is radial in $E^c$. The end of Part 1 remains then unchanged and we obtain that $\varphi$ is radial in the whole domain $\Om$.
\item {\bf Part 2:} It can be adapted directly by changing the term $\lambda$ into $-\kappa \lambda$ everywhere. It can be noticed that we did not use the sign of the left-hand side in the ordinary differential equation satisfied by $U$, namely \eqref{eq:Ur}. We only used the fact that $U$ does not vanish. 
\item {\bf Part 3:} Once we know that $\Om$ is a ball, we know that both $E^c$ and $E$ are rotationally symmetric, so the same arguments as in Part 3 above apply.
\end{itemize}
\qed

\medskip

\noindent{\bf Proof of Proposition \ref{prop:geomprop}.} We will apply the chain of arguments of Part 3. For that purpose, it suffices to prove that, if $m$ is a radially symmetric function, then so is the principal eigenfunction $\varphi$. Let us argue by contradiction, by considering a radial function $m$, the associated principal eigenvalue $\lambda(m)$, and assuming that $\varphi$, a principal eigenfunction, is not radial. Following the method introduced in Part 1, we claim that this statement is equivalent to the existence of two integers $i,j\in \llbracket1,N\rrbracket$ with $i\neq j$ such that the function $v_{ij}$ defined by \eqref{def:vradiale} does not vanish identically in $\Om$ and solves in a variational sense the system
$$
\left\{ \begin{array}{rll} -\Delta v_{ij} &= \lambda m v_{ij} \quad &\hbox{ in } \Omega ,\\
\partial_n v_{ij} + \beta v_{ij} &=0 \quad &\hbox{ on } \partial \Om,
          \end{array}\right.
$$
meaning that $v_{ij}$ is also a principal eigenfunction associated to the principal eigenvalue $\lambda(m)$. In particular, $v_{ij}$ has a constant sign. Since $v_{ij}$ stands for the derivative of $\varphi$ with respect to $\theta$, the angular variable associated to the polar coordinates in the plane $(Ox_ix_j)$, we get a contradiction with the periodicity of $\varphi$ with respect to $\theta$. The first part of the proposition hence follows. Note that the second one is easily obtained by applying Theorem \ref{theo:optimball}. 
\qed

\begin{remark}
According to Section \ref{sec:period}, there is a correspondance between the optimal shape $E$ for Neumann boundary conditions in a hypercube and the optimal shape in a periodic cell. One can refer for instance to \cite{BHR,Roques-Hamel}. As a consequence, the proof of Theorem \ref{theo:optimball} can be extended to the case where $\Omega$ is a periodic-cell (in $\R^2$, $\Om$ is a square on which we impose periodic boundary conditions) by noting that in such a case, the eigenfunction $\varphi$ satisfies homogeneous Neumann boundary conditions on the boundary of the periodic cell. Indeed, this is a direct consequence of the symmetry of $E$ and the periodic boundary conditions. We then have to investigate the case of a square $C$ on which homogeneous Neumann boundary conditions are imposed and $E$ is a centered ball. Notice that Theorem \ref{theo:optimball} cannot be directly applied in that case since $\partial C$ is not $C^1$. Nevertheless, we claim that the proof and statement of Theorem \ref{theo:optimball} can be adapted in the periodic framework. Indeed, only Part 2 of the proof has to be modified. Using the same notations as in the proof of Theorem \ref{theo:optimball}, it is enough to notice that one can choose (for instance) $x_a=(0,\dots,0, L_N)$. Since the boundary is locally flat around $x_a$, the same argument applies and the conclusion of Theorem  \ref{theo:optimball} remains true in that case.

\end{remark}

\section{Applications}\label{sec:applications}

In the two next sections, we provide hereafter several numerical simulations based on the algorithm described in Section \ref{sec:numerics}. Recall that for reasonable parameters, the convergence of this algorithm to a local minimizer has been shown in \cite{CGIKO,MR2886017}. 

Our implementation relies on the Matlab Partial Differential Equation Toolbox using piecewise linear and globally continuous finite elements. We worked on a standard desktop machine and the resulting code works out the solution very quickly (see the convergence curves). 

In the case where $\beta=0$ and $\kappa=1/2$ (the sets of parameters that we have chosen in the sequel), there exists a principal eigenvalue for Problem \eqref{eq:EV_main} if, and only if $c\in (0,2/3)$. 

\subsection{The case of a \texorpdfstring{$N$}{Lg}-orthotope with Neumann boundary conditions}\label{sec:rectangle}
We assume in this section that $\beta=0$ and $\Om= \Pi_{k=1}^{N}(0,L_{k})$, and we aim at describing more precisely $E^{*}$ solution of \eqref{minishape} in this framework. 

We have already recalled in Sections \ref{sec:period} and \ref{sec:numerics} that a common conjecture in dimension $2$  in this framework is that the minimizing set has constant curvature, that is, it would be a quarter of ball, a stripe, or the complementary of a quarter of ball depending on the parameters (see \cite{KaoLouYanagida, Roques-Hamel}). We will prove that this conjecture is false when $\beta=0$.

\begin{proposition} \label{thm:rectangle}
Assume that $N\geq 2$, $\Om= \Pi_{k=1}^{N}(0,L_{k})$, $\beta=0, \kappa\in (0,+\infty)$ and $c\in(0,\frac{1}{\kappa+1})$.
If $E^*$ is a minimizing set for \eqref{minishape}, then
\begin{enumerate}
\item (Steiner symmetry) $m_{E^{*}}$ is monotonic with respect to $x_{k}\in (0,L_{k})$ for all $k$. 
\item If $\partial E^{*}\cap \Om$ is analytic, then $\partial E^{*}\cap \Om$ does not contain any piece of sphere. 
\end{enumerate}
\end{proposition}

Note that the hypothesis that $\partial E^*\cap \Om$ is analytic is always satisfied if $N=2$ and $\beta=+\infty$, as it is shown in \cite{CKT}.
When $\beta>0$, we know from Theorem \ref{theo:optimball} that $E^{*}$ cannot be a ball. But we do not know how to conclude when $E^{*}$ hits the boundary of $\Om$ in that case since there is no link with the periodic framework as for Neumann boundary conditions. 

\begin{proof}
(1) Let $\widetilde{E^{*}}$ be the reflection of $E^{*}$ with respect to $\{x_{k}=0\}$ for $k=1,...,N$. According to  \cite[Appendix C]{Roques-Hamel}, this set minimizes the periodic principal eigenvalue in the cell $C=\Pi_{k=1}^{N}(-L_{k},L_{k})$:
$$
\lambda^{*} = \min \left\{\frac{\int_{C} |\nabla \psi|^{2} }{\int_{C}m_{\widetilde{E^{*}}}\psi^{2}}, \ \psi \in H^{1}_{per}(C)\textrm{ and }\int_\Om m\psi^2>0\right\}.
$$
Let $\varphi$ be the associated periodic principal eigenfunction, which is also the minimizer of the above Rayleigh quotient. 
Then, classical rearrangement inequalities yield 
$$\lambda^{*}=\frac{\int_{C} |\nabla \varphi|^{2} }{\int_{C}m_{\widetilde{E^{*}}}\varphi^{2}} \geq \frac{\int_{C} |\nabla {\varphi^{S}}|^{2} }{\int_{C}m_{(\widetilde{E^{*}})^{S}}({\varphi}^{S})^{2}}\geq \lambda^*,$$
where $\varphi^S$ and $(\widetilde{E^*})^S$ denote the successive Steiner symmetrization of $\varphi, \widetilde{E^*}$ with respect to $x_{1}=0$, $x_{2}=0$,..., $x_{N}=0$. As the equality holds, this yields that ${\varphi} = {\varphi}^{S} (\cdot+X)$ for some $X\in \R^{N}$. Let $k\in \llbracket1,N\rrbracket$. Since ${\varphi}$ is symmetric with respect to $\{x_{k}=0\}$ by construction, this necessarily implies that ${\varphi}$ is either nonincreasing or nondecreasing with respect to $x_{k}$. The conclusion follows by using that $E^{*}= \{\varphi > \alpha\}$ up to a set of zero measure.

\smallskip

(2)  If $\partial E^*\cap \Om$ contains a piece of sphere, then so does $\partial \widetilde{E^{*}}$ and, thanks to analyticity, $\partial \widetilde{E^{*}}$ is itself a sphere.
Then $\widetilde{E^{*}}$ or $C\backslash \operatorname{clos}(\widetilde{E^{*}})$ is a ball. Notice that Theorem \ref{theo:optimball} does not apply in that case since $\partial C$ is not $C^1$. Nevertheless, we claim that the proof and statement of Theorem \ref{theo:optimball} can be adapted in the periodic framework. Indeed, only Part 2 of the proof has to be modified. Using the same notations as in the proof of Theorem \ref{theo:optimball}, it is enough to notice that one can choose (for instance) $x_a=(0,\dots,0, L_N)$. Since the boundary is locally flat around $x_a$, the same argument applies and the conclusion of Theorem  \ref{theo:optimball} remains true in that case.
This yields that the set $C=\Pi_{k=1}^{N}(-L_{k},L_{k})$ must be a sphere, whence a contradiction.
\end{proof}

The numerical results for the square in the Neumann case are gathered on Figure \ref{fig:neusquare} and two convergence curves illustrating the efficiency of the method are drawn on Figure \ref{fig:CVcurveNeusq}.

\begin{figure}[h!]
\centering
\subfigure[$c=0.2$ - optimal domain]{\includegraphics[scale=0.35]{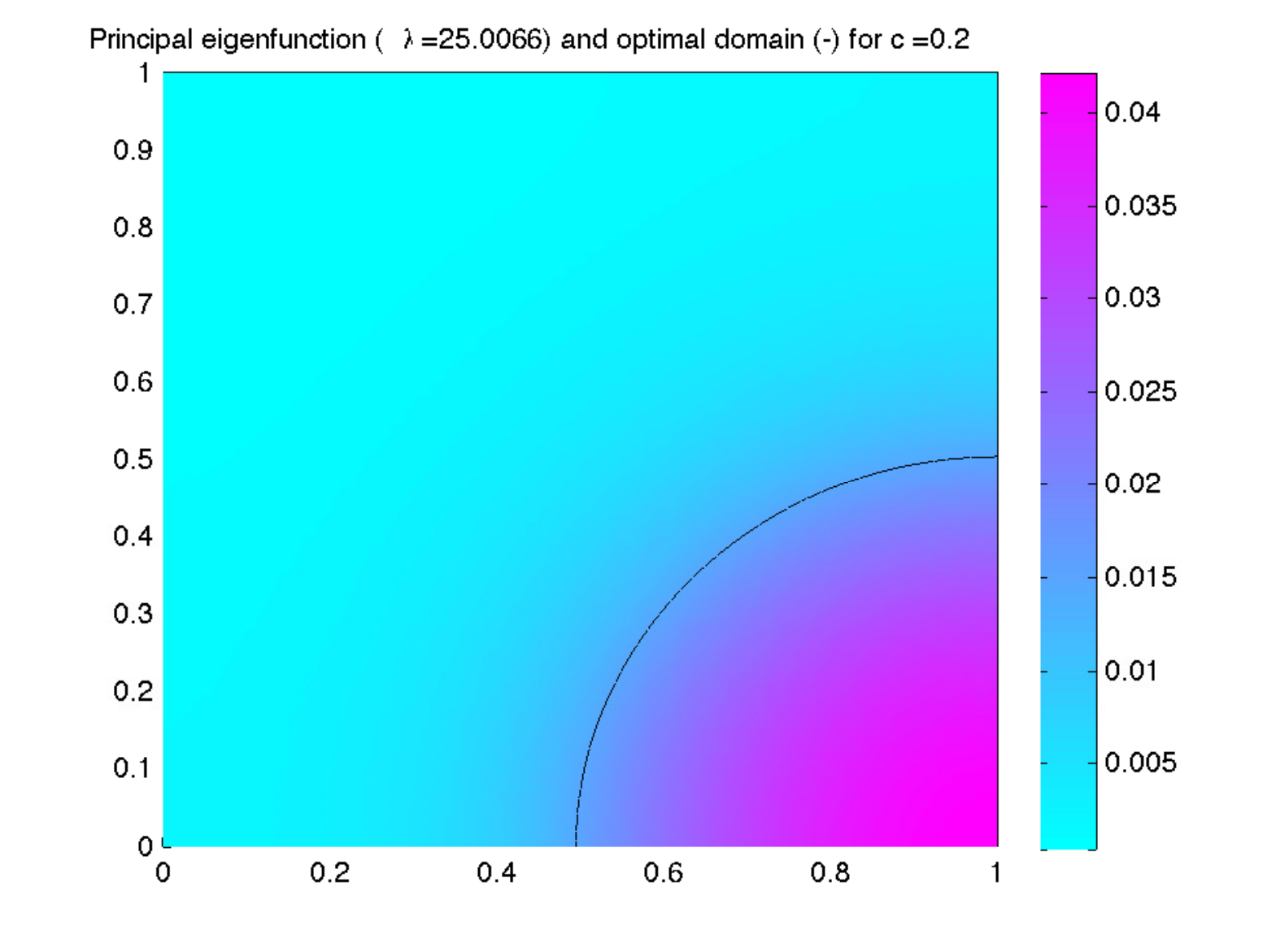}}
\subfigure[$c=0.3$ - optimal domain]{\includegraphics[scale=0.35]{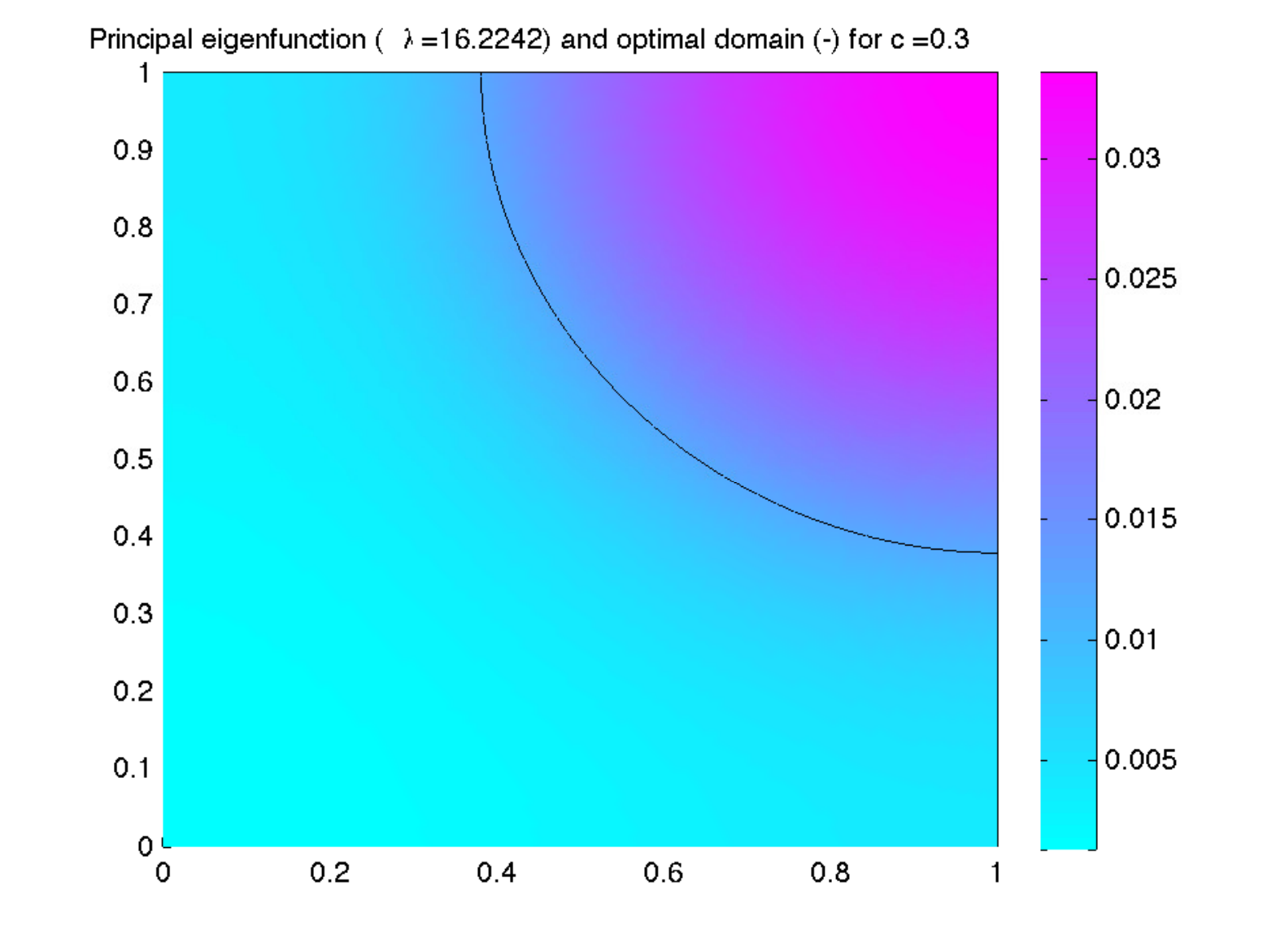}}
\subfigure[$c=0.4$ - optimal domain]{\includegraphics[scale=0.35]{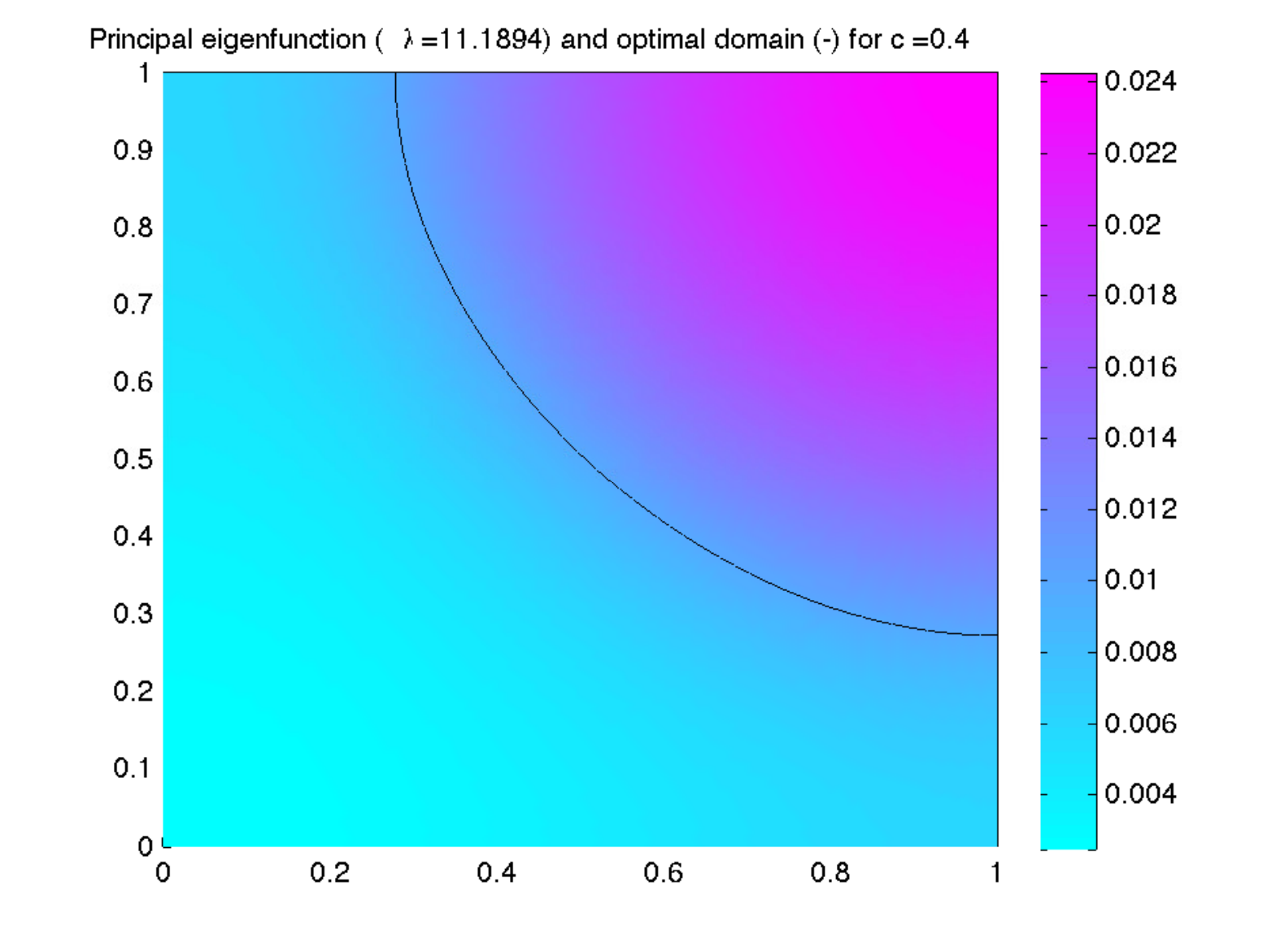}}
\subfigure[$c=0.5$ - optimal domain]{\includegraphics[scale=0.35]{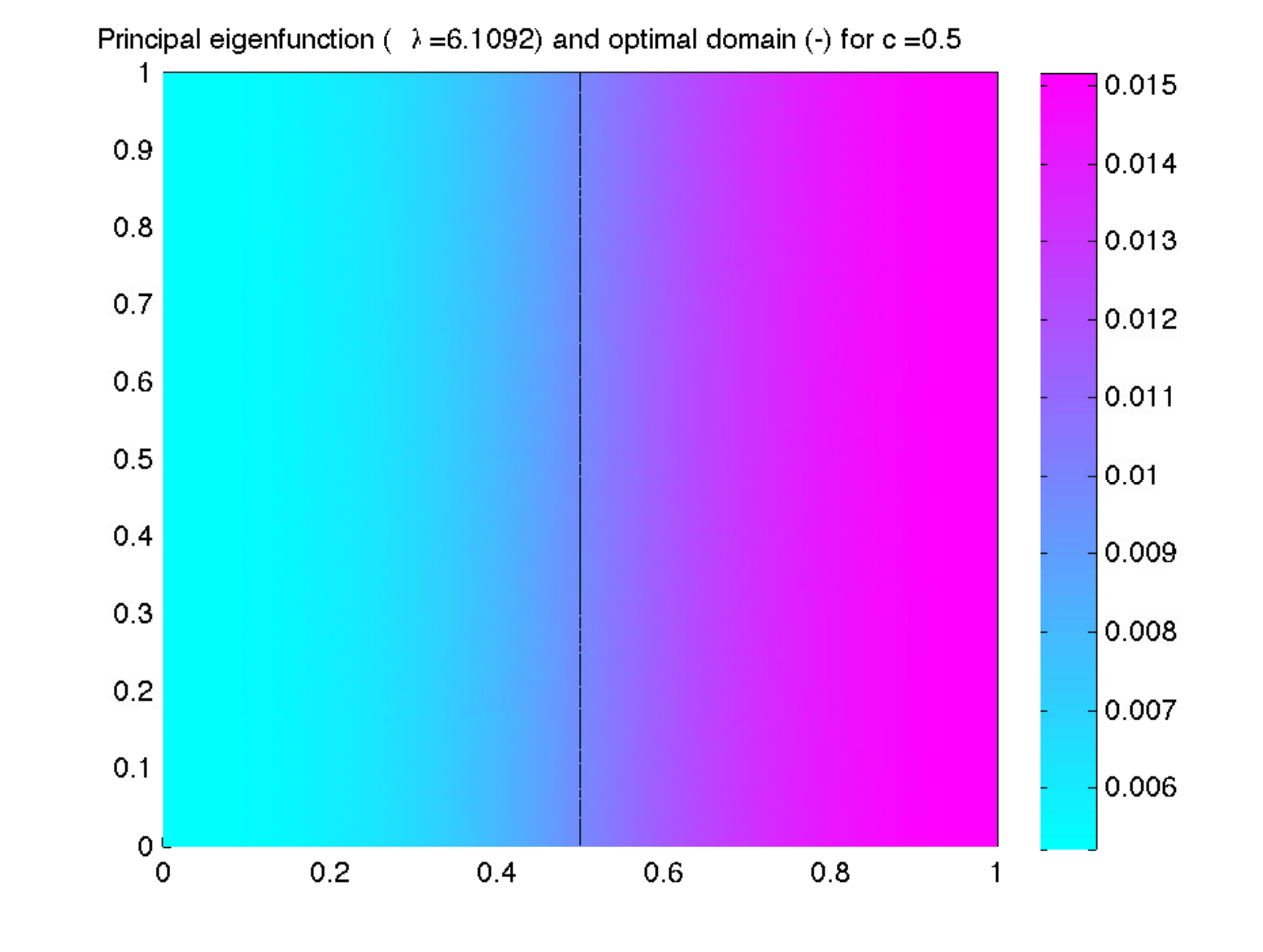}}
\subfigure[$c=0.6$ - optimal domain]{\includegraphics[scale=0.35]{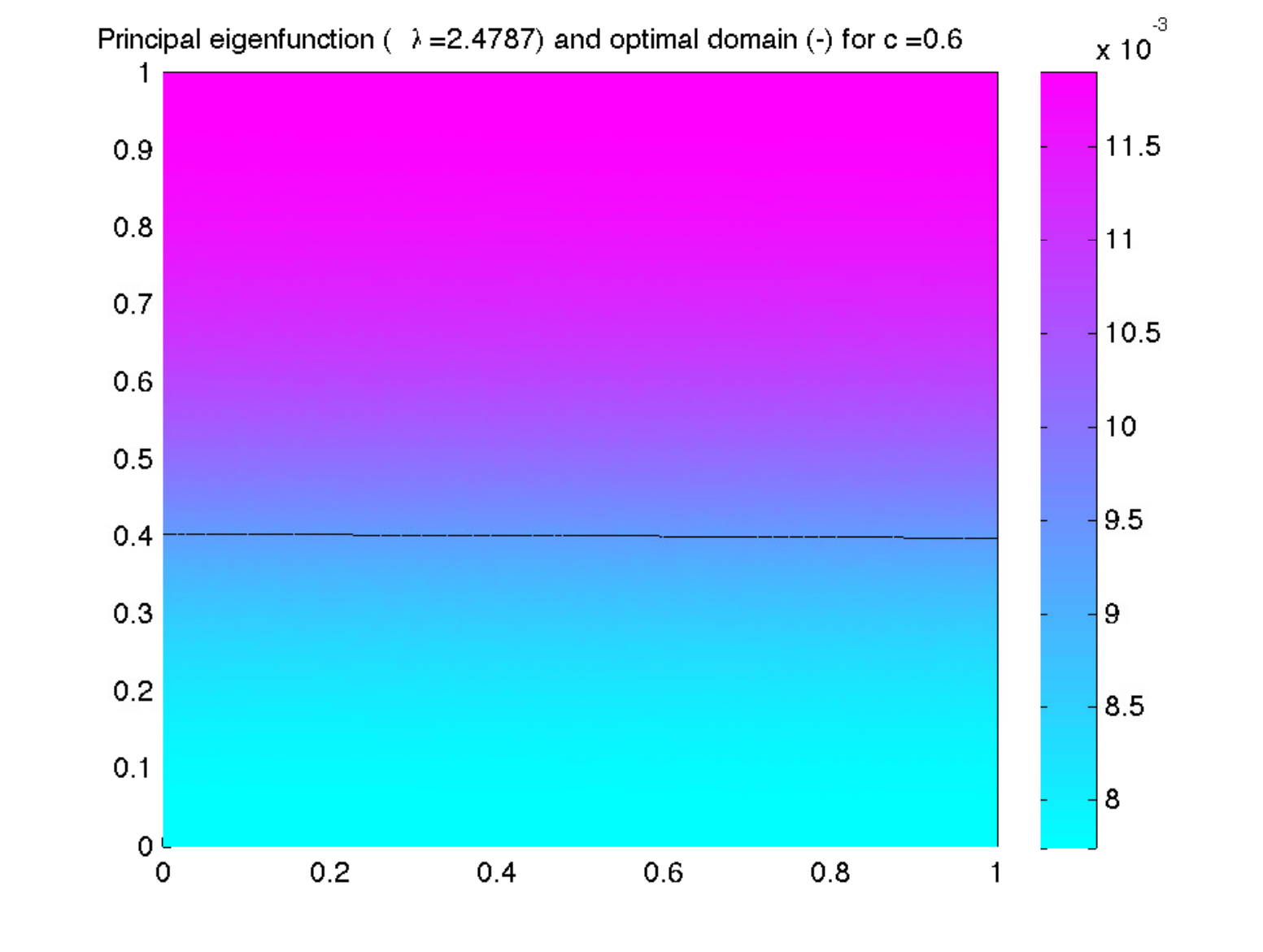}}
\caption{$\Omega=(0,1)^2$. Optimal domains in the Neumann case ($\beta=0$) with $\kappa=0.5$ and $c\in \{0.2,0.3,0.4,0.5,0.6\}$\label{fig:neusquare}} 
\end{figure}

\begin{figure}[h!]
\centering
\subfigure[$c=0.2$ - convergence curve]{\includegraphics[scale=0.35]{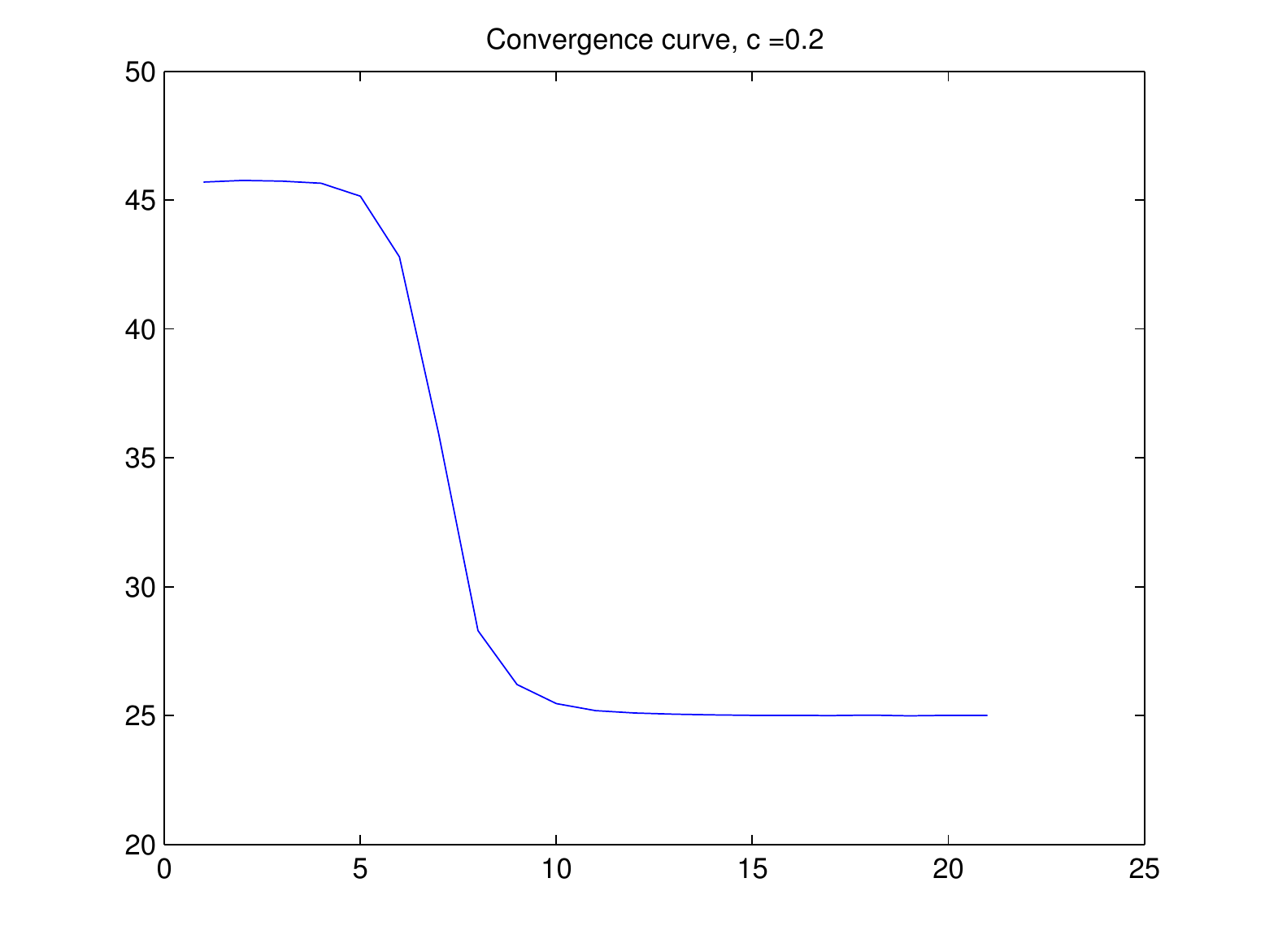}}
\subfigure[$c=0.6$ - convergence curve]{\includegraphics[scale=0.35]{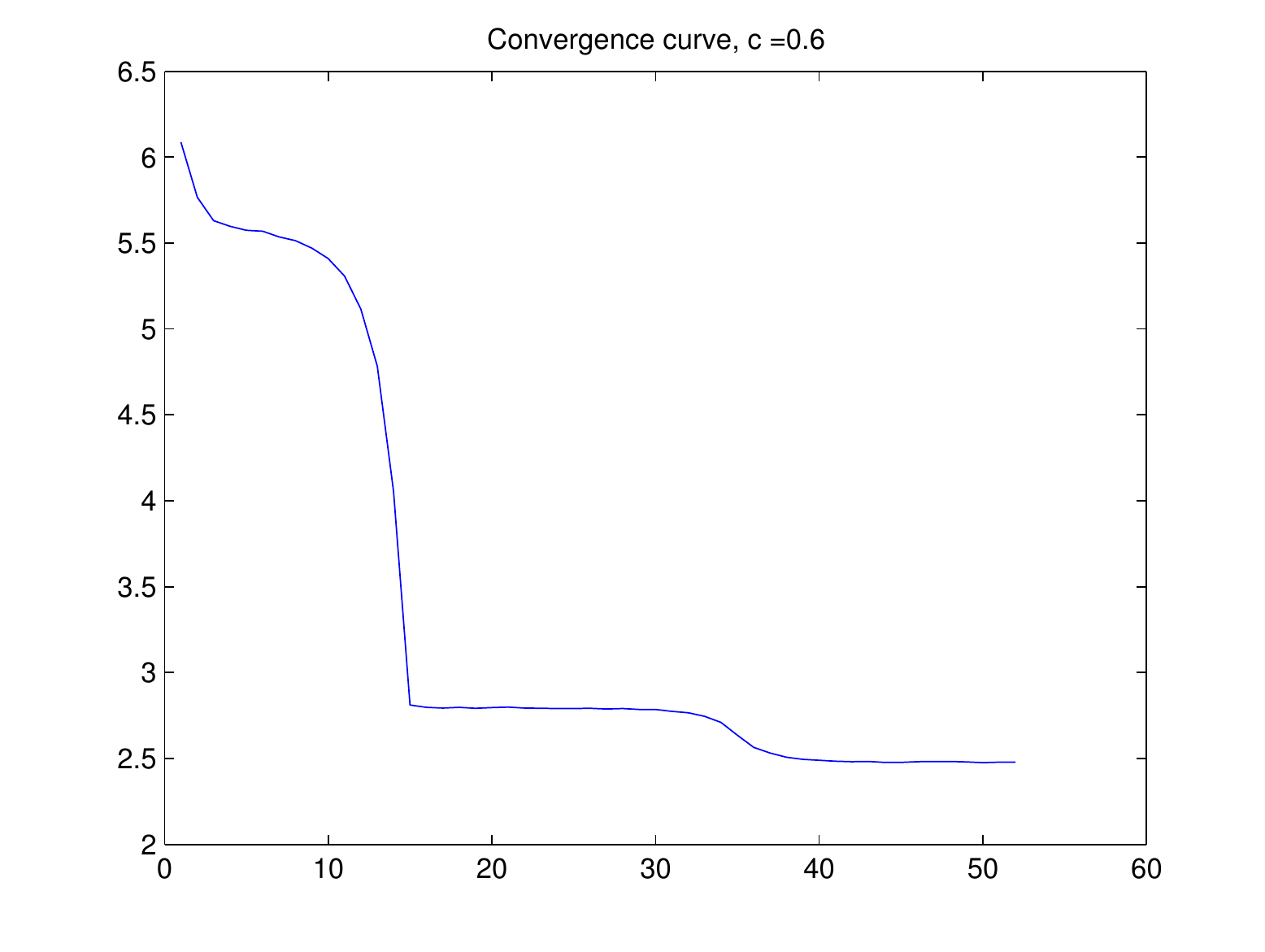}}
\caption{$\Omega=(0,1)^2$. Two examples of convergence curves in the Neumann case ($\beta=0$) with $\kappa=0.5$, $c=0.2$ (left) and $c=0.6$ (right)\label{fig:CVcurveNeusq}} 
\end{figure}

\subsection{The case of a ball}

We assume in this section that $\Om= B(0,1)$, and we aim at describing more precisely $E^{*}$ solution of \eqref{minishape} in this framework. 

\begin{proposition} \label{thm:ball}
Assume that $N=2$, $\Om= B(0,1)$, $\beta=0$, $\kappa\in(0,+\infty)$ and $c\in(0,\frac{1}{\kappa+1})$.
If $E^{*}$ is a minimizing set for \eqref{minishape}, then 
\begin{enumerate}
\item (Circular Symmetry) There exists $\theta_{0}\in [0,2\pi)$ such that $E^{*}$ is symmetric with respect to the half straight line $\{\theta=\theta_{0}\}$ in the radial coordinates $(r,\theta)$. Moreover, for all $r\in (0,1)$, $\{\theta \in [0,2\pi), (r,\theta) \in E^{*}\}$ is an interval. 
\item If $\beta =0$, then $E^{*}$ is not a ball. 
\end{enumerate}
\end{proposition}

The only related numerical simulations we know in this framework have been performed in \cite{KaoLouYanagida}, when $\Om$ is an ellipse. Only one set of parameters has been tested in this earlier work and in that case $E^{*}$ looks like a portion of disk hitting the boundary. 

\begin{proof}
(1) This follows from similar arguments as in the proof of Proposition \ref{thm:rectangle}. We refer to 
\cite{MR846062} for details on circular symmetrization. We just notice here that the term $\beta \int_{\partial \Om}\varphi^{2}$ is preserved with respect to circular symmetrization. 

(2) If $E^{*}$ was a ball, then it would be a centered one according to Theorem \ref{theo:optimball}. Then the result follows from Theorem \ref{th:deformation} below.
\end{proof}

Finally, we complete the theoretical analysis of the situation where $\Omega$ is the $N$-dimensional Euclidean unit ball and $\beta=0$ by showing that a radially symmetric set (and in particular the centered ball) cannot solve Problem \eqref{mini} for $N=2,3,4$. The next result is the most involved of this section. Our argument rests upon a particular rearrangement technique that breaks the radial symmetry in the disk and decreases the Rayleigh quotient $\Re_m$ defined by \eqref{def:Re}.
\begin{figure}[!ht]
\begin{center}
\includegraphics[height=5cm]{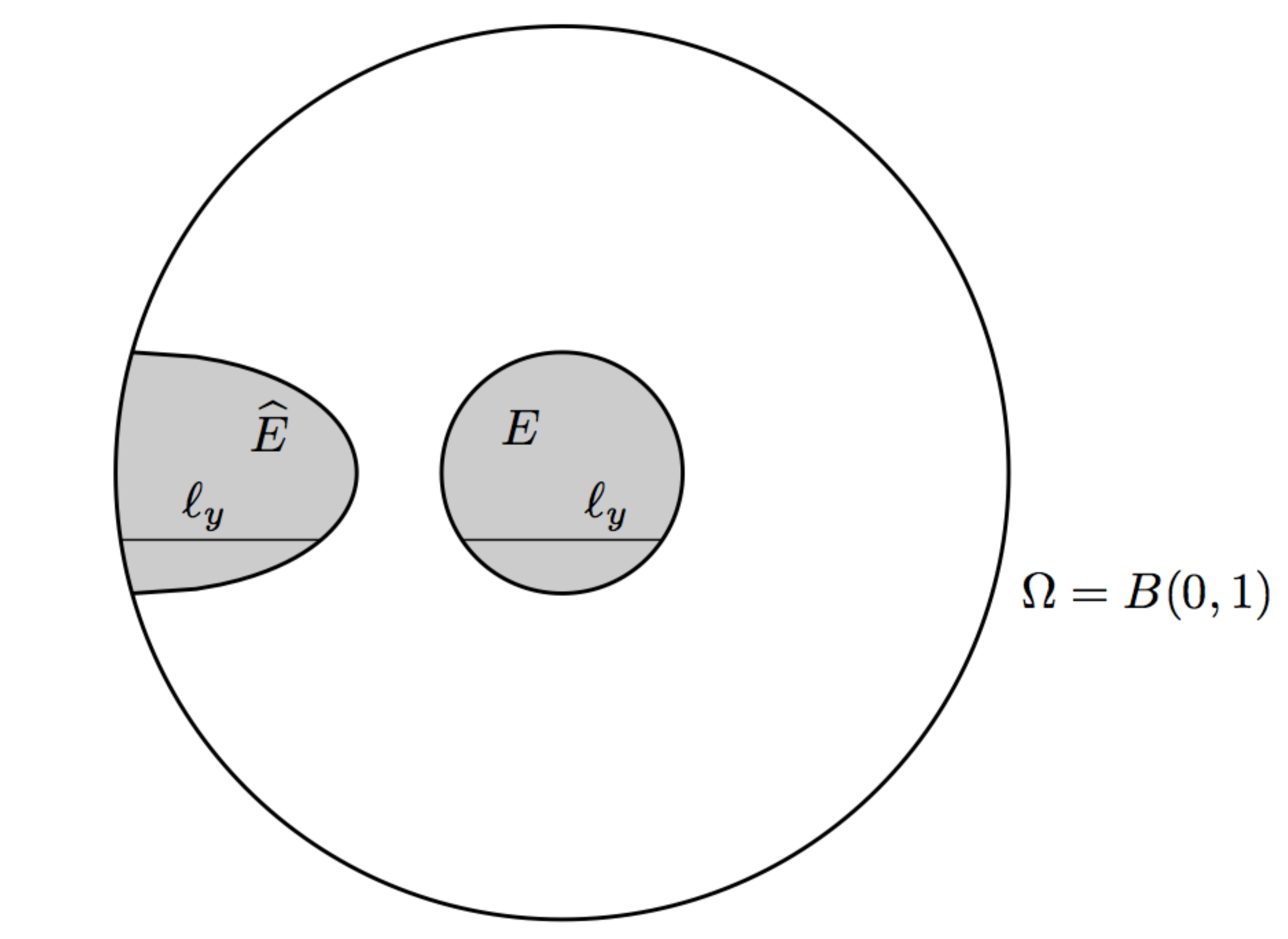}
\caption{Construction of the set $\widehat{E}$ from $E$\label{fig:hatE0906}}
\end{center}
\end{figure}

For the following theorem  we introduce the variable $x' = (x_2,...,x_N)$, the plane $P':= \{x_1=0\}$ and the ball 
$$\Om' :=\left\{ x'\in \R^{N-1} :\sum_{i=2}^N x_i^2 \leq 1\right\}$$ 
of dimension $N-1$.
\begin{theorem}\label{th:deformation}
Assume $N\geq 2$, $\Om= B(0,1)$, $\beta=0$, $\kappa\in(0,\infty)$ and $c\in(0,\frac{1}{\kappa+1})$. Let $E$ be a radially symmetric set, centered at $0$. Then there exists a set $\widehat{E}$ such that
\begin{equation}\label{ineq_stretch_lambda}
\lambda(\widehat{E})<\left(\frac{5N-4}{4N}\right)\lambda(E),
\end{equation}
and moreover: 
\begin{itemize}
\item $\partial \Om\cap \partial\widehat E\neq \emptyset$, 
\item $\widehat E$ and $\Omega\setminus \widehat E$ are both convex in the direction $x_1$, i.e. $\widehat E\cap\mathcal D_y$ and $(\Omega\setminus \widehat E)\cap\mathcal D_y$ are intervals, where $\mathcal D_y$ is the line passing through $(0,x')$, $x'\in\Om'$ and parallel to the $x_1$-axis,
\item the set $\widehat E$ is symmetric with respect to every hyperplane $P_i =\{x_i = 0\}$ for $i\geq 2$ but is not symmetric with respect to $P'$.
\end{itemize}
\end{theorem}
Notice that the constant $(5N-4)/4N$ is positive for all $N\geq 2$.
\begin{proof}
Defining $f:x'\in\Om'\mapsto f(x')=\sqrt{1-\sum_{i=2}^N x_i^2}$ we have
\begin{equation}\label{omega_graph}
\Om = B(0,1) = \{ (x_1,x'): x'\in \Om'\mbox{ and } -f(x')< x_1 < f(x')  \} .
\end{equation}
As $E$ is symmetric with respect to $P'$,
we consider the following transformations of $m$, $\varphi$ and $E$: 
$$\widehat m(x_1,x'):= m\left( \frac{x_1 + f(x')}{2},x' \right),\qquad \widehat \varphi(x_1,x'):= \varphi\left( \frac{x_1 + f(x')}{2},x' \right),\qquad \widehat E = \{ \widehat m =\kappa\}.$$ 
This transformation corresponds to considering the restrictions of $m$ and $\varphi$ to the right half  of $\Omega$, and then ``stretching'' these restrictions to all of $\Om$ (see Figure \ref{fig:hatE0906}). 

We show, using that $E$ is radially symmetric, that 
this transformation decreases the eigenvalue. 
We start by proving  that the transformation preserves the constraints of the problem. We have
\begin{align*}
\int_{\Omega} \widehat m \widehat\varphi^2 
& = \int_{\Om'} \int_{-f(x')}^{f(x')}  \widehat m(x_1,x') \widehat\varphi(x_1,x')^2 dx_1 dx'\\
& = \int_{\Om'} \int_{-f(x')}^{f(x')}  m\left( \frac{x_1 + f(x')}{2},x' \right) \varphi\left( \frac{x_1 + f(x')}{2},x' \right)^2 dx_1 dx' \\
& = \int_{\Om'} \int_{0}^{f(x')}  m\left( \widehat x,x' \right) \varphi\left( \widehat x,x' \right)^2 2d\widehat x dx'\\
& = \int_{\Om'} \int_{-f(x')}^{f(x')}  m\left( \widehat x,x' \right) \varphi\left( \widehat x,x' \right)^2 d\widehat x dx' 
=\int_{\Omega}  m \varphi^2.
\end{align*}
where we have used the change of variable $\widehat x_1 = (x_1 + f(x'))/2$ and the symmetry of $m$ and $\varphi$ with respect to the hyperplane $P'$, which follows from proposition \ref{prop:geomprop}. As a consequence, one has $\int_{\Omega} \widehat m \widehat\varphi^2>0$. With a similar calculation we can prove that $|\widehat E| = |E|$.

Now we prove that the $L^2$-norm of the gradient decreases. We have
\begin{align*}
\int_{\Omega} | \nabla\widehat\varphi |^2 
& = \int_{\Om'} \int_{-f(x')}^{f(x')}  \frac{1}{4} \left| \partial_{x_1}\varphi\left(  \frac{x_1 + f(x')}{2},x' \right) \right|^2 \\
&
\qquad +
\left|   \partial_{x_1}\varphi\left(  \frac{x_1 + f(x')}{2},y \right) \frac{\nabla_{x'} f(x')}{2}  + \nabla_{x'}\varphi\left(  \frac{x_1 + f(x')}{2},x' \right) \right|^2    dx_1dx'\\
& = 2\int_{\Om'} \int_{0}^{f(x')}  \frac{1}{4} \left| \partial_{x_1}\varphi\left( \widehat x_1,x' \right) \right|^2 
+
\left|   \partial_{x_1}\varphi\left(  \widehat x_1,y \right) \frac{\nabla_{x'} f(x')}{2}  + \nabla_{x'}\varphi\left(  \widehat x_1 ,x' \right) \right|^2    d\widehat x_1dx'\\
& = \int_{\Om'} \int_{-f(x')}^{f(x')}  \frac{1}{4} \left| \partial_{x_1}\varphi\left( \widehat x_1,x' \right) \right|^2 
+
\left|   \partial_{x_1}\varphi\left(  \widehat x_1,y \right) \frac{\nabla_{x'} f(x')}{2}  + \nabla_{x'}\varphi\left(  \widehat x_1 ,x' \right) \right|^2    d\widehat x_1dx'.
\end{align*}
As $\Om = B(0,1)$ and $E$ is radially symmetric, the level sets of $\varphi$ are also radially symmetric according to Proposition \ref{prop:geomprop}. Let us first compare the signs of $\partial_{x_1}\varphi\left( x_1,x' \right) $ and $\partial_{x_i}\varphi\left( x_1,x' \right) $ for $x_1\geq 0$ and $x_i\geq 0$ for some $i\geq 2$.

First of all note that $\varphi$, in hyperspherical coordinates, is $\mathcal C^1$ with respect to the variable $r$, this can be seen by writing the equation for $\varphi$ in these coordinates.
On one hand if  there exists $\widehat h_{x_1}>0$ such that $\varphi(x_1+h,x')\geq \varphi(x_1,x')$ for all $h\in [0,\widehat h_{x_1}]$ then
$$ \partial_{x_1}\varphi\left( x_1,x' \right)  = \lim_{h\to 0} \frac{\varphi(x_1+h,x') -\varphi(x_1,x') }{h}\geq 0$$
and due to the radial symmetry of the level sets of $\varphi$ and $x_i\geq 0$, there exists also $\widehat h_{x_i}>0$ such that $\varphi(x_1,..,x_{i-1},x_i+h,x_{i+1},..,x_N)\geq \varphi(x_1,x')$ for all $h\in [0,\widehat h_{x_i}]$, and consequently $\partial_{x_i}\varphi\left( x_1,x' \right)\geq 0$.
On the other hand if  there exists $\widehat h_{x_1}$ such that $\varphi(x_1+h,x')\leq \varphi(x_1,x')$ for all $h\in [0,\widehat h_{x_1}]$, then we similarly get $ \partial_{x_1}\varphi\left( x_1,x' \right)  \leq 0$ and $\partial_{x_i}\varphi\left( x_1,x' \right)\leq 0$.

If neither of these two situations do happen, then there exists a sequence $h_k\to 0$, $h_k>0$ such that 
$$\varphi(x_1+h_{2k},x')\geq \varphi(x_1,x')\mbox{ and }\varphi(x_1+h_{2k+1},x')\leq \varphi(x_1,x').$$
Passing to the limit in the differential quotient for the two subsequences $h_{2k}$ and $h_{2k+1}$  we get $ \partial_{x_1}\varphi\left( x_1,x' \right)  \geq 0$ and $ \partial_{x_1}\varphi\left( x_1,x' \right)  \leq 0$  and consequently  $ \partial_{x_1}\varphi\left( x_1,x' \right)  = 0$.

Therefore, as $\partial_{x_i} f(x')\leq 0$ for $x_i\geq 0$, the two terms $\partial_{x_i} f(x') \partial_{x_1}\varphi( x_1,x' )$ and $   \partial_{x_i}\varphi\left(  x_1,x' \right)$ always have opposite signs for  $0\leq x_1\leq f(x')$ and $0\leq x_i$
(this includes the case $\partial_{x_i} f(x') \partial_{x_1}\varphi( x_1,x' )=0$ as a limit case).
Due to the symmetries of $\Om$, $\partial_{x_i} f(x') \partial_{x_1}\varphi( x_1,x' )$ and $   \partial_{x_i}\varphi\left(  x_1,x' \right)$ always have opposite signs. 
Therefore we have
$$
\left| \frac{\partial_{x_i} f(x')}{2} \partial_{x_1}\varphi( x_1,x' )+  \partial_{x_i}\varphi (  x_1,x' ) \right|^2 
\leq 
\left|  \frac{\partial_{x_i} f(x')}{2} \partial_{x_1}\varphi( x_1,x' )\right|^2
+\left|  \partial_{x_i}\varphi (  x_1,x' )  \right|^2 .
$$
This yields the estimate
\begin{align*}
\int_{\Omega} | \nabla\widehat\varphi |^2 &   = \int_{\Om'} \int_{-f(x')}^{f(x')}  \frac{1}{4} \left| \partial_{x_1}\varphi\left(  x_1,x' \right) \right|^2 
+
\left|   \partial_{x_1}\varphi\left(  x_1,y \right) \frac{\nabla f(x')}{2}  + \nabla_{x'}\varphi\left( x_1 ,x' \right) \right|^2   dx_1dx'\\
& \leq \int_{\Om'} \int_{-f(x')}^{f(x')}  
\frac{1}{4} \left| \partial_{x_1}\varphi\left(  x_1,x' \right) \right|^2 
+\sum_{i\geq 2} \Big(
 \left|  \frac{\partial_{x_i} f(x')}{2} \partial_{x_1}\varphi( x_1,x' )\right|^2
 +\left|  \partial_{x_i}\varphi (  x_1,x' )  \right|^2  \Big) dx_1dx'.
\end{align*}
To continue with the main estimate, let us write down the hyperspherical coordinates in dimension $N$. Let $\theta_k\in [0,\pi]$ for $1\leq k\leq N-2$ and $\theta_{N-1}\in [0,2\pi]$. The relation with Cartesian coordinates is given by
\begin{align*}
x_1 = r\cos\theta_1, \qquad
x_i  = r\cos\theta_i \prod\limits_{k=1}^{i-1} \sin \theta_k\quad \mbox{ for } 2\leq i\leq N-1,\qquad
x_N = r \prod\limits_{k=1}^{N-1} \sin \theta_k  
\end{align*}
and $r =|x|$. 
A simple calculation shows that $|\partial_{x_i} f(x')|^2 = f(x')^{-2}|x_i|^2$ for $i\geq 2$.
Using hyperspherical coordinates we get
\begin{align*}
f(x')^{2} = 1 - \sum_{k= 2 }^{N} x_k^2 = 1 - \sum_{k= 2 }^{N-2} x_k^2  - r^2 \prod_{k=1}^{N-2}(\sin\theta_k)^2
= ... = 1-r^2\sin^2\theta_1.
\end{align*}
Since we have the radial symmetry we define $U(r) := \varphi ( x_1,x')$. 
We have then $\partial_{x_1}\varphi ( x_1,x') = U'(r)\cos\theta_1$. 
Introduce for $i\geq 2$ the integral
\begin{align*}
K_i: &= \int_{\Om'} \int_{-f(x')}^{f(x')} \left|  \partial_{x_i} f(x') \partial_{x_1}\varphi( x_1,x' )\right|^2  dx_1dx'  \\
& = \int_{\theta_{N-1}=0}^{2\pi} \int_{\theta_{N-2}=0}^{\pi} ... \int_{\theta_{1}=0}^{\pi} \int_{r=0}^{1}   \frac{r^2\cos^2\theta_i \prod_{k=1}^{i-1} \sin^2 \theta_k}{1-r^2\sin^2\theta_1} U'(r)^2  \cos^2\theta_1  r^{N-1}  \prod_{k=1}^{N-2} (\sin \theta_k)^{N-k-1}  drd\theta_1...d\theta_{N-1} . 
\end{align*}
Clearly the function
$$ 
[0,1]\ni r\mapsto \frac{r^2}{1-r^2\sin^2\theta_1}. 
$$
is increasing, for $\theta_1\in (0 ,\pi)$ fixed. Hence, one has
$$ \frac{r^2}{1-r^2\sin^2\theta_1}  
<   \frac{1}{1-\sin^2\theta_1} = \frac{1}{\cos^2\theta_1}\quad \mbox{ for } r < 1,$$
and rearranging the other terms in $K_i$ leads to 
$$K_i < \int_{\theta_{N-1}=0}^{2\pi} \int_{\theta_{N-2}=0}^{\pi} ... \int_{\theta_{1}=0}^{\pi} \int_{r=0}^{1}   r^{N-1}U'(r)^2  \cos^2 
\theta_i \left(\prod_{k=1}^{i-1} (\sin \theta_k)^{N-k+1}  \right) 
\left( \prod_{k=i}^{N-2} (\sin \theta_k)^{N-k-1}\right)  drd\theta_1...d\theta_{N-1} .
$$
To estimate the integral $K_i$, we can compute the various integrals above separately. We start with
\begin{align*}
\prod_{k=1}^{i-1} \int_{\theta_k =0}^{\pi}(\sin \theta_k)^{N-k+1} d\theta_k =  \prod_{k=1}^{i-1} 2W_{N-k+1}
\end{align*}
where $W_{N-k+1}$ denotes the $N-k+1$-th Wallis integral\footnote{The Wallis integrals are the terms of the sequence 
$(W_n)_{n \in \mathds{N}}$ defined by
$$
    W_n = \int_0^{\frac{\pi}{2}} \sin^nx\,dx.
    $$
    }. Using the well-known relation of Wallis integrals for $q\in\mathds{N}^*$
\begin{equation}
\label{wallis_relation}
 qW_q = (q-1)W_{q-2} 
\end{equation}
we obtain
\begin{align}\label{wallis1}
\prod_{k=1}^{i-1} \int_{\theta_k =0}^{\pi}(\sin \theta_k)^{N-k+1} d\theta_k = 2^{i-1} \prod_{k=1}^{i-1} W_{N-k-1}\frac{N-k}{N-k+1} = 2^{i-1} \frac{N-i+1}{N} \prod_{k=1}^{i-1} W_{N-k-1}.
\end{align}
Pursuing the estimate of $K_i$, we study the term 
$$\int_{0}^{\pi} \cos^2\theta_i(\sin \theta_i)^{N-i-1} d\theta_i 
= \int_{0}^{\pi} (\sin \theta_i)^{N-i-1} d\theta_i
- \int_{0}^{\pi} (\sin \theta_i)^{N-i+1} d\theta_i
=2W_{N-i-1} - 2W_{N-i+1}.
$$
Using relation \eqref{wallis_relation} we get
\begin{equation}
\label{wallis2}
\int_{0}^{\pi} \cos^2\theta_i(\sin \theta_i)^{N-i-1} d\theta_i 
=\frac{2}{N-i+1}W_{N-i-1} .
\end{equation}
Gathering \eqref{wallis1} and \eqref{wallis2} we obtain
$$K_i\leq 2\pi \left(\int_{r=0}^{1}   r^{N-1}U'(r)^2 dr\right) 
\frac{2^{N-2}}{N}\prod_{k=1}^{N-2}W_{N-k-1}.
$$
On the other hand we have
\begin{align*}
\| \partial_{x_1}\varphi \|^2_{L^2(\Omega)}
& =  \int_{\theta_{N-1}=0}^{2\pi} \int_{\theta_{N-2}=0}^{\pi} ... \int_{\theta_{1}=0}^{\pi} \int_{r=0}^{1}  
  U'(r)^2  \cos^2\theta_1  r^{N-1}  \prod_{k=1}^{N-2} (\sin \theta_k)^{N-k-1}  drd\theta_1...d\theta_{N-1}  \\
& = 2\pi\left(\int_{r=0}^{1} 
  U'(r)^2  r^{N-1} dr\right)
  \prod_{k=2}^{N-2} 2W_{N-k-1} \int_{0}^\pi \cos^2\theta_1 (\sin\theta_1)^{N-2}  d\theta_1\\
& = 2\pi\left(\int_{r=0}^{1} 
  U'(r)^2  r^{N-1} dr\right)
  (2 W_{N-2} - 2 W_{N})\prod_{k=2}^{N-2} 2W_{N-k-1}. \\  
\end{align*}
Using \eqref{wallis_relation} we get $2 W_{N-2} - 2 W_{N} = 2N^{-1} W_{N-2}$ which yields
\begin{align*}
\| \partial_{x_1}\varphi \|^2_{L^2(\Omega)} & = 2\pi\left(\int_{r=0}^{1} 
  U'(r)^2   r^{N-1} dr\right)
 \frac{2^{N-2}}{N}\prod_{k=1}^{N-2} W_{N-k-1}. 
\end{align*}
Thus gathering the results above we have obtained the estimate
$$ K_i < \| \partial_{x_1}\varphi \|^2_{L^2(\Omega)}.$$
Now using this estimate yields
\begin{align*}
\int_{\Omega} | \nabla\widehat\varphi |^2 
< 
\frac{1}{4}  \| \partial_{x_1}\varphi \|^2_{L^2(\Omega)}
+\sum_{i\geq 2}\Big(
\frac{1}{4}  \| \partial_{x_1}\varphi \|^2_{L^2(\Omega)}  +   \| \partial_{x_i}\varphi \|^2_{L^2(\Omega)} \Big).
\end{align*}
Now we observe that  $\| \partial_{x_{i_1}}\varphi \|_{L^2(\Omega)} =  \| \partial_{x_{i_2}}\varphi \|_{L^2(\Omega)}$ for all indices $i_1$ and $i_2$ due to the radial symmetry of $\varphi$. 

Thus we get
\begin{align*}
\int_{\Omega} | \nabla\widehat\varphi |^2 
& <
\frac{1}{4}  \| \partial_{x_1}\varphi \|^2_{L^2(\Omega)}
+(N-1)(1+\frac{1}{4})
 \| \partial_{x_1}\varphi \|^2_{L^2(\Omega)} 
 = \frac{5N-4}{4}\| \partial_{x_1}\varphi \|^2_{L^2(\Omega)}\\
& < \frac{5N-4}{4N} \| \nabla\varphi \|^2_{L^2(\Omega)} 
\end{align*}
and the expected conclusion follows.

Finally, if $E$ is also a centered ball,  then denoting $E_R$ the part of $E$ which is on the right of the plane $P'$, $E_R$ is a half-ball and is convex. The transformation $E_R\mapsto\widehat E $ obviously preserves the convexity in the $x_1$-direction since intervals are mapped onto intervals for fixed $x'\in \Om'$.   
Thus it is clear that $\widehat E$ and $\Om\setminus\widehat E$ are both convex in the $x_1$-direction.
\end{proof}

\begin{remark}
In the proof of Theorem \ref{th:deformation}, there is some room to improve the estimate for $K_i$, and in turn the estimate for the eigenvalue, using the estimate
$$ \frac{r^2}{1-r^2\sin^2\theta_1}  
<   \frac{r^2}{1-\sin^2\theta_1}\quad \mbox{ for } r < 1,\theta_1\in (0,\pi).$$ 
However, to obtain a practical estimate, one needs more informations about the spatial distribution of $U'(r)^2$.
\end{remark}


The numerical results for the disk in the Neumann case are gathered on Figure \ref{fig:disk1900} and two convergence curves illustrating the efficiency of the method are drawn on Figure \ref{fig:CVcurveneudisk}.

\begin{figure}[h!]
\centering
\subfigure[$c=0.2$ - optimal domain]{\includegraphics[scale=0.35]{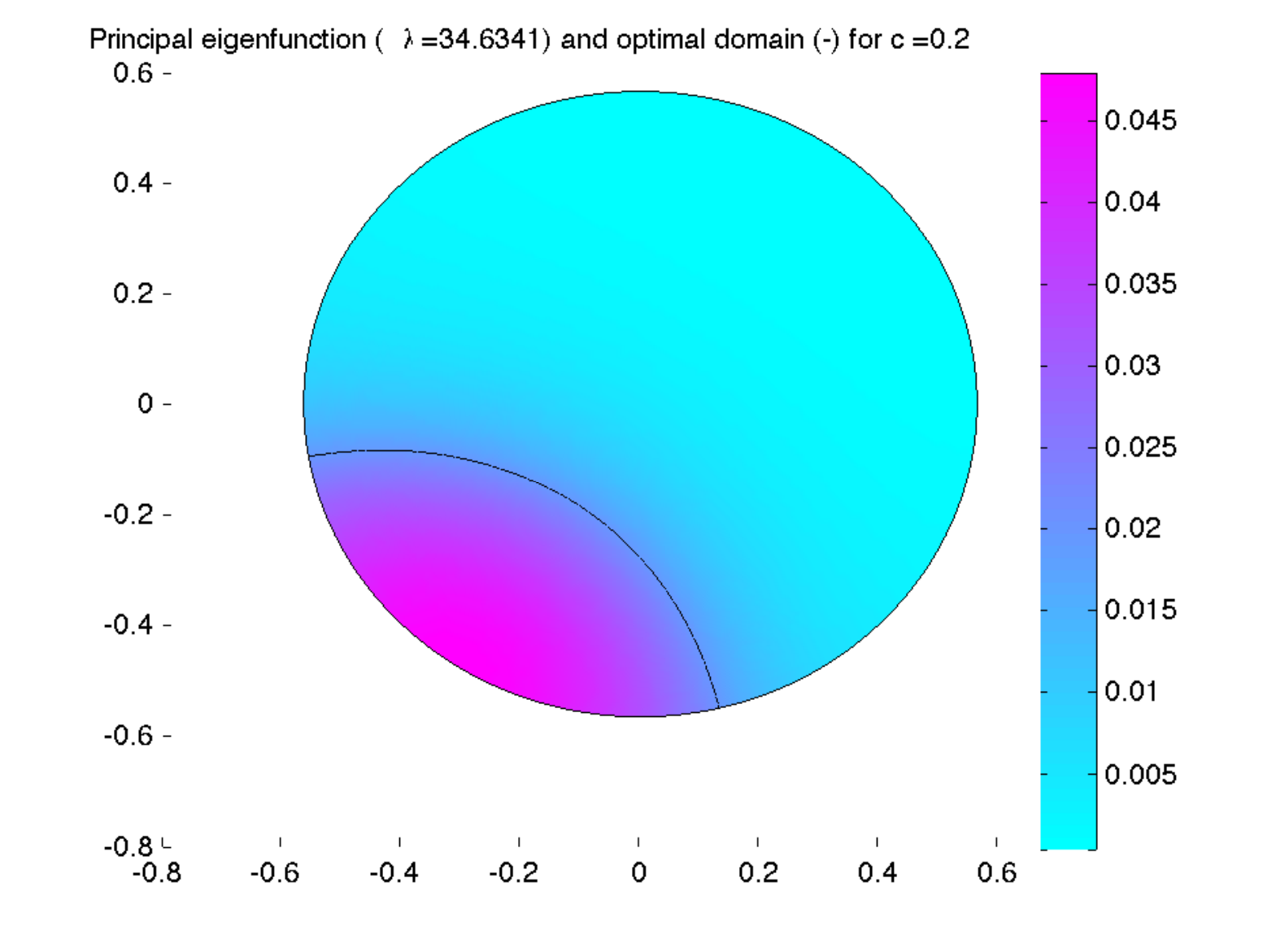}}
\subfigure[$c=0.3$ - optimal domain]{\includegraphics[scale=0.35]{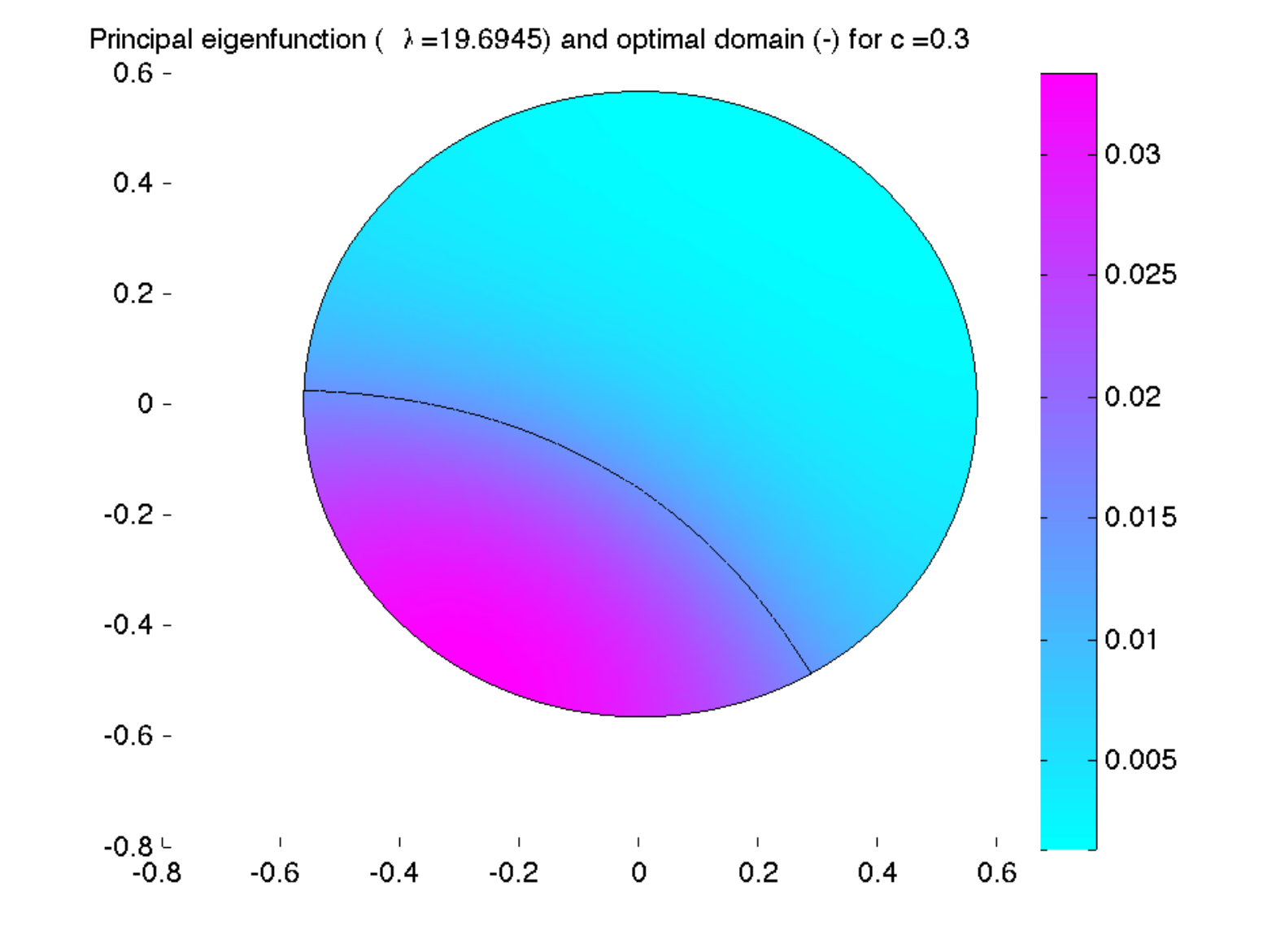}}
\subfigure[$c=0.4$ - optimal domain]{\includegraphics[scale=0.35]{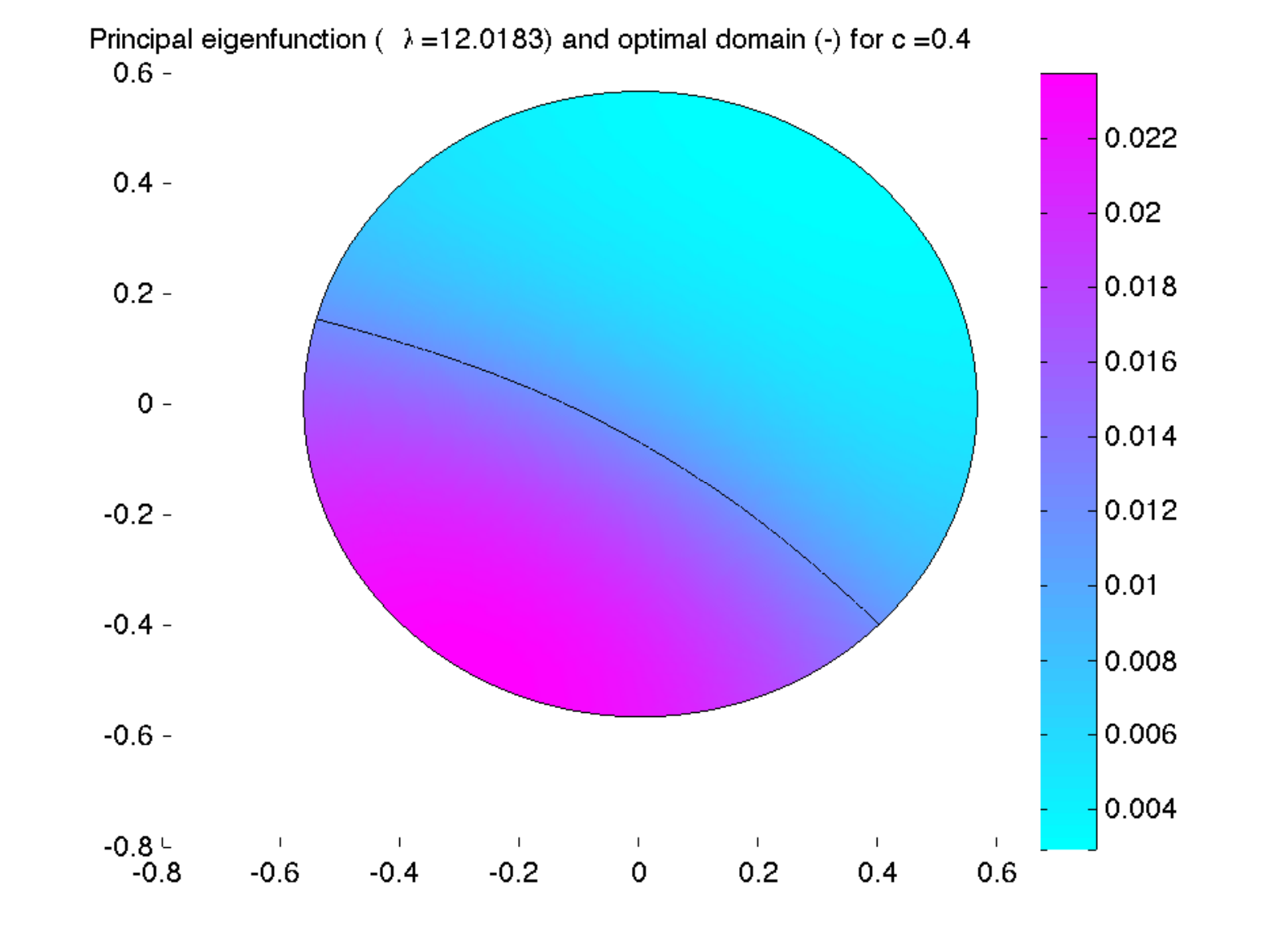}}
\subfigure[$c=0.5$ - optimal domain]{\includegraphics[scale=0.35]{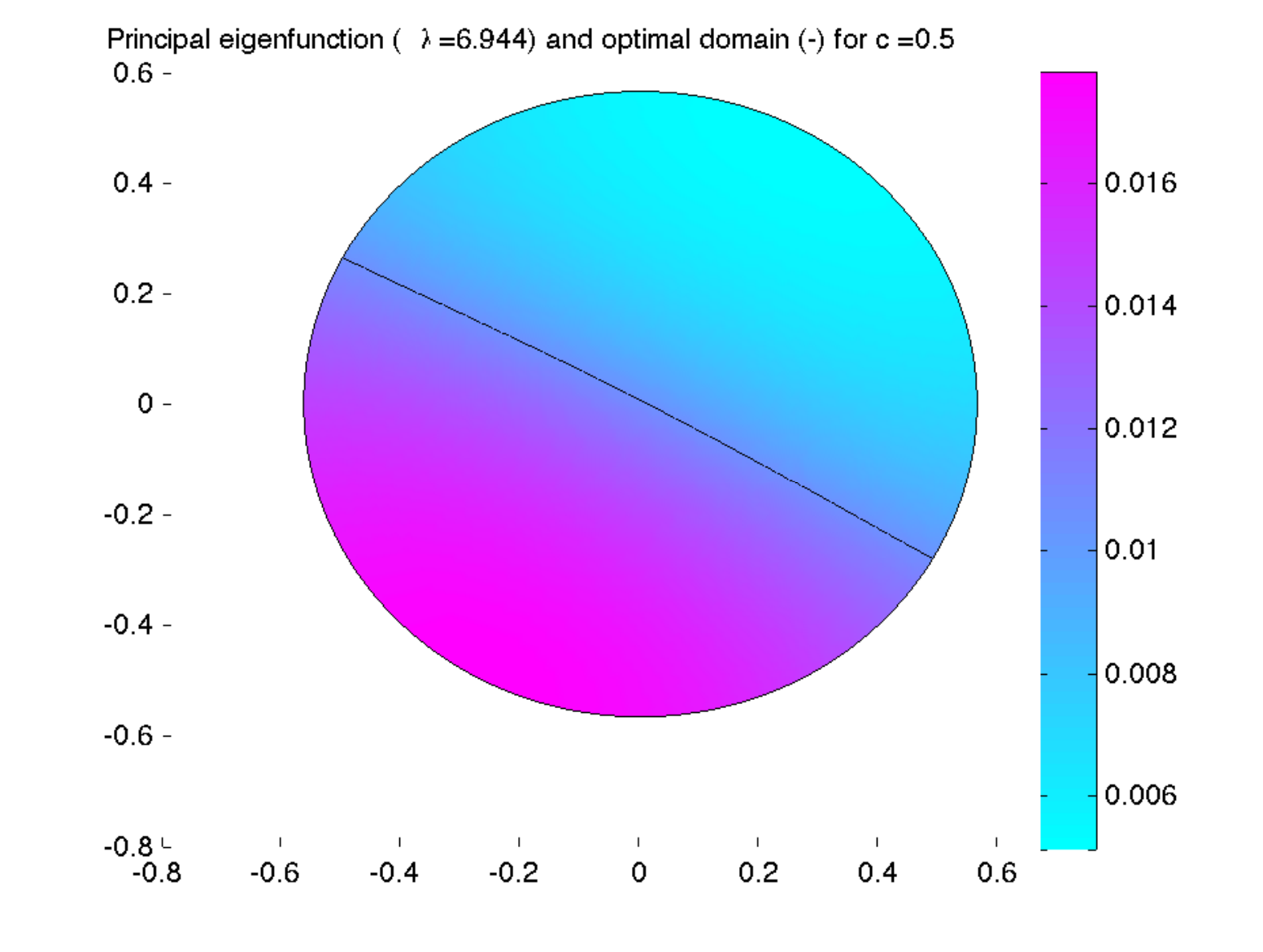}}
\subfigure[$c=0.6$ - optimal domain]{\includegraphics[scale=0.35]{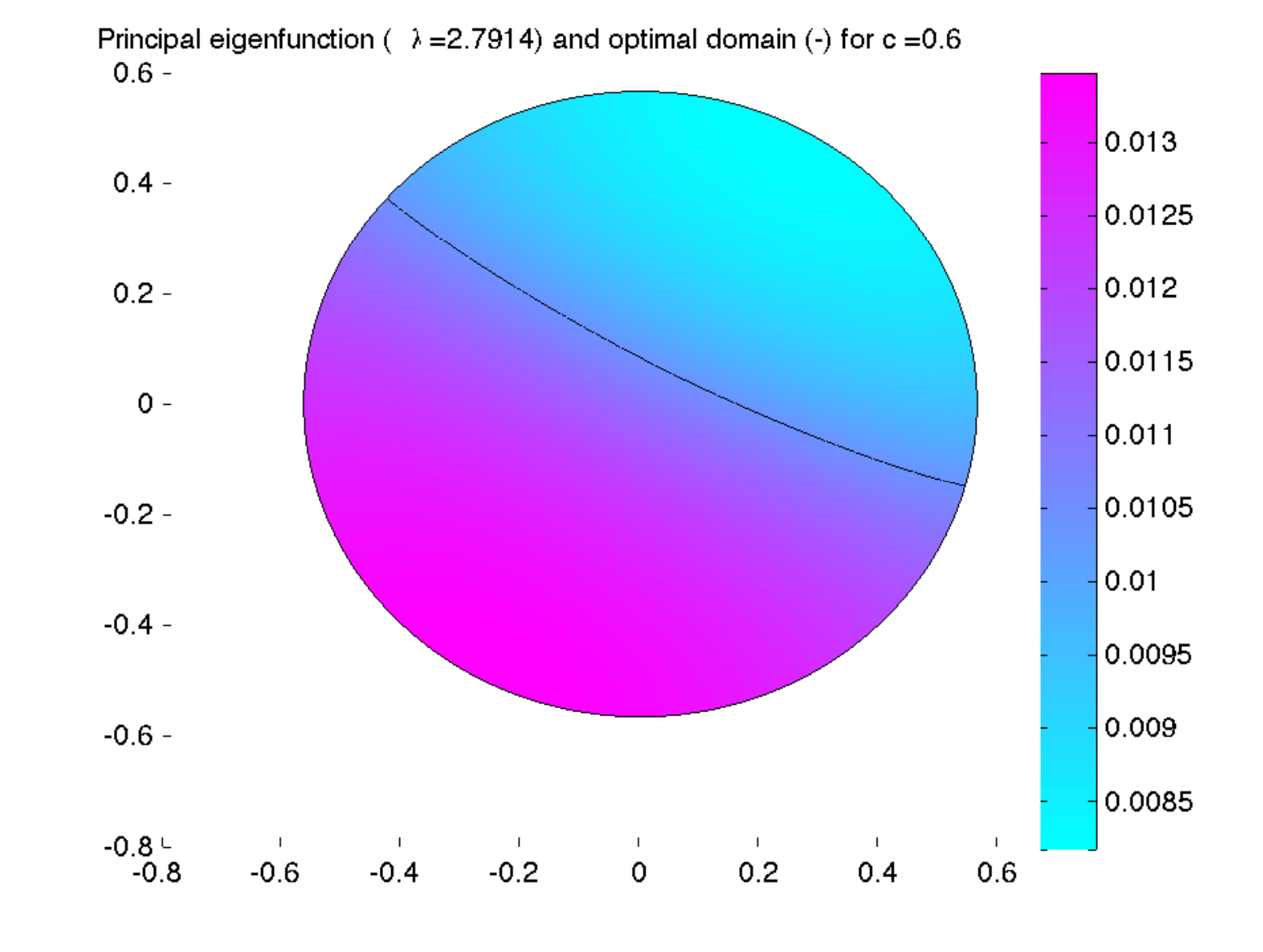}}
\caption{$\Omega=B(0,1/\sqrt{\pi})$. Optimal domains in the Neumann case ($\beta=0$) with $\kappa=0.5$ and $c\in \{0.2,0.3,0.4,0.5,0.6\}$} \label{fig:disk1900}
\end{figure}

\begin{figure}[h!]
\centering
\subfigure[$c=0.2$ - convergence curve]{\includegraphics[scale=0.35]{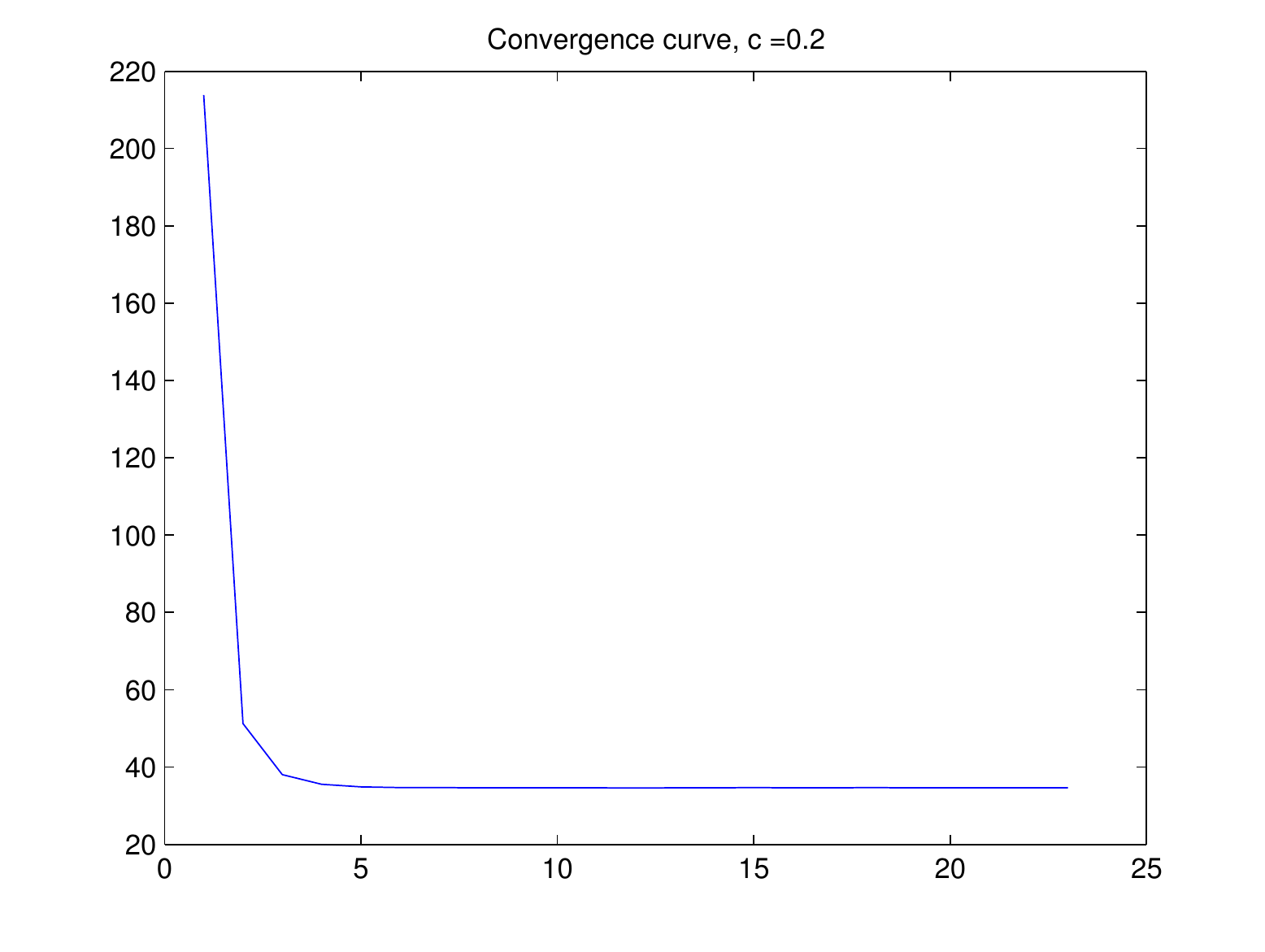}}
\subfigure[$c=0.6$ - convergence curve]{\includegraphics[scale=0.35]{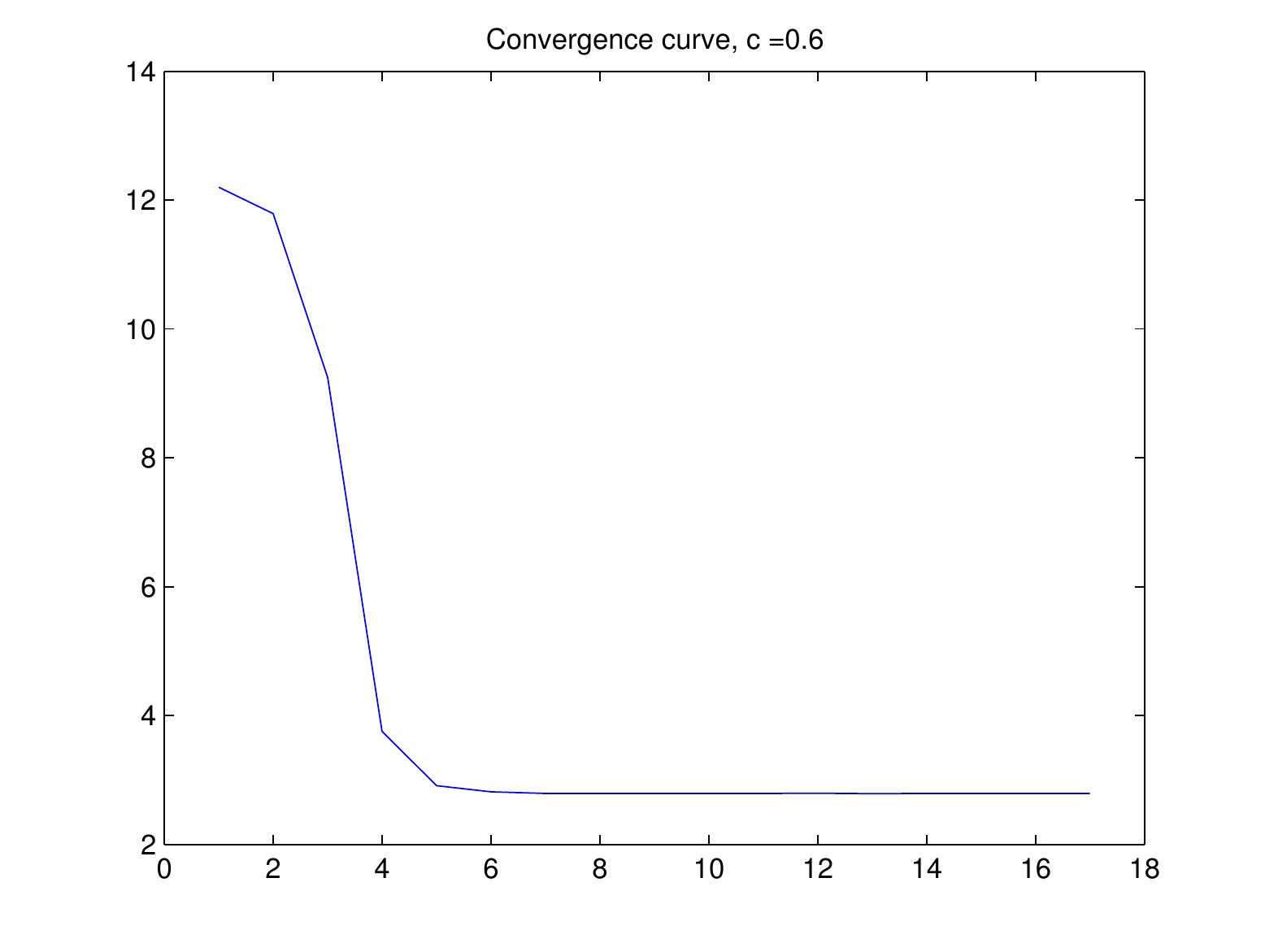}}
\caption{$\Omega=B(0,1/\sqrt{\pi})$. Two examples of convergence curves in the Neumann case ($\beta=0$) with $\kappa=0.5$ and $c=0.2$ (left) or $c=0.6$ (right)\label{fig:CVcurveneudisk}} 
\end{figure}

According to these simulations, the optimal set $E^*$ looks like a portion of disk intersecting $\Om$, but we did not manage to confirm nor to invalidate this observation theoretically.

To end this section, let us provide some numerical hints suggesting that  the optimal set $E^*$ are not portions of disks. 
Assume from now on that $\Om$ is a disk of radius $R$ (and $R=1/\sqrt{\pi}$ on Figure \ref{fig:disk1900} so that $|\Om|=1$). We expect from the Neumann boundary conditions that the boundary $\partial E\cap \Om$ will hit $\partial \Om$ with angle $\pi/2$. It follows from Pythagora's theorem that the distance between the center of $\Om$ and the center of $E$ is $\sqrt{R^2+r_c^2}$. Therefore, an easy but tedious computation shows that
$$
|E|=c|\Om|=R^2\arcsin\left(\frac{r_c}{\sqrt{R^2+r_c^2}}\right)+r_c^2\arcsin\left(\frac{R}{\sqrt{R^2+r_c^2}}\right)-r_cR.
$$
The mapping $r_c\mapsto R^2\arcsin\left(\frac{r_c}{\sqrt{R^2+r_c^2}}\right)+r_c^2\arcsin\left(\frac{R}{\sqrt{R^2+r_c^2}}\right)-r_cR$ is increasing on $\R_+$ and it follows that $r_c$ is determined in a unique way from $c$.

The numerical results presented on the table below suggest that $E_c$, the piece of disk of radius $r_{c}$, is not optimal for most of the possible values of $c$. This conjecture is tested for $R=1 / \sqrt{\pi}$, $\kappa=0.5$ and several values of the parameter $c$. However, for $c=0.15$, the algorithm we used did not manage to exhibit a set which is better than the piece of disk $E_c$. In all cases, it would be interesting to lead in a separate study a refined numerical investigation in order to validate or invalidate the conjecture that a piece of disk does not solve Problem \eqref{mini} when $\Omega$ stands for the unit disk and for $\beta\geq 0$.

\medskip

\begin{center}
\begin{tabular}{|c||c|c|c|c|c|c|c|c|}
\hline
 &  $c=0.1$ & $c=0.15$ & $c=0.2$ & $c=0.25$ & $c=0.3$ & $c=0.35$ & $c=0.4$\\
\hline
\hline
$r_c$  & 0.3408  &  0.4714   & 0.6234 &   0.8166  &  1.0869  &  1.5149  &  2.3408 \\
\hline
$\lambda(E_c)$  & 80.2483 & 49.5896 &34.6791 & 25.6912 & 19.7057 & 15.3542 & 12.0286 \\
\hline
$\lambda(E^*)$  & 80.2435  &  49.5912 &  34.6341  & 25.6727  & 19.6945  & 15.3520  & 12.0260\\
\hline
\end{tabular}
\end{center}

\medskip

\subsection{Some additional numerical investigations for \texorpdfstring{$\beta>0$}{Lg}}\label{sec:finnum}

In this section, we gather the optimal domains we have obtained for several values of the parameter $\beta$.  As previously, we consider the cases where $\Omega$ is the unit square or the disk with radius $1/\sqrt{\pi}$. As previously, the algorithm described in Section \ref{sec:numerics} is used to determine the optimal domain $E^*$. 

On Figure \ref{fig:squarebeta} (resp. on Figure \ref{fig:diskbeta}), the optimal domain $E^*$ are plotted for $c=0.2$, $\beta \in  \{1,5,50,1000\}$ and $\Omega=(0,1)^2$ (resp. $\Omega=B(0,1/\sqrt{\pi})$), as well as several convergence curves. 

\begin{figure}[h!]
\centering
\subfigure[$c=0.2$ - $\beta=1$]{\includegraphics[scale=0.35]{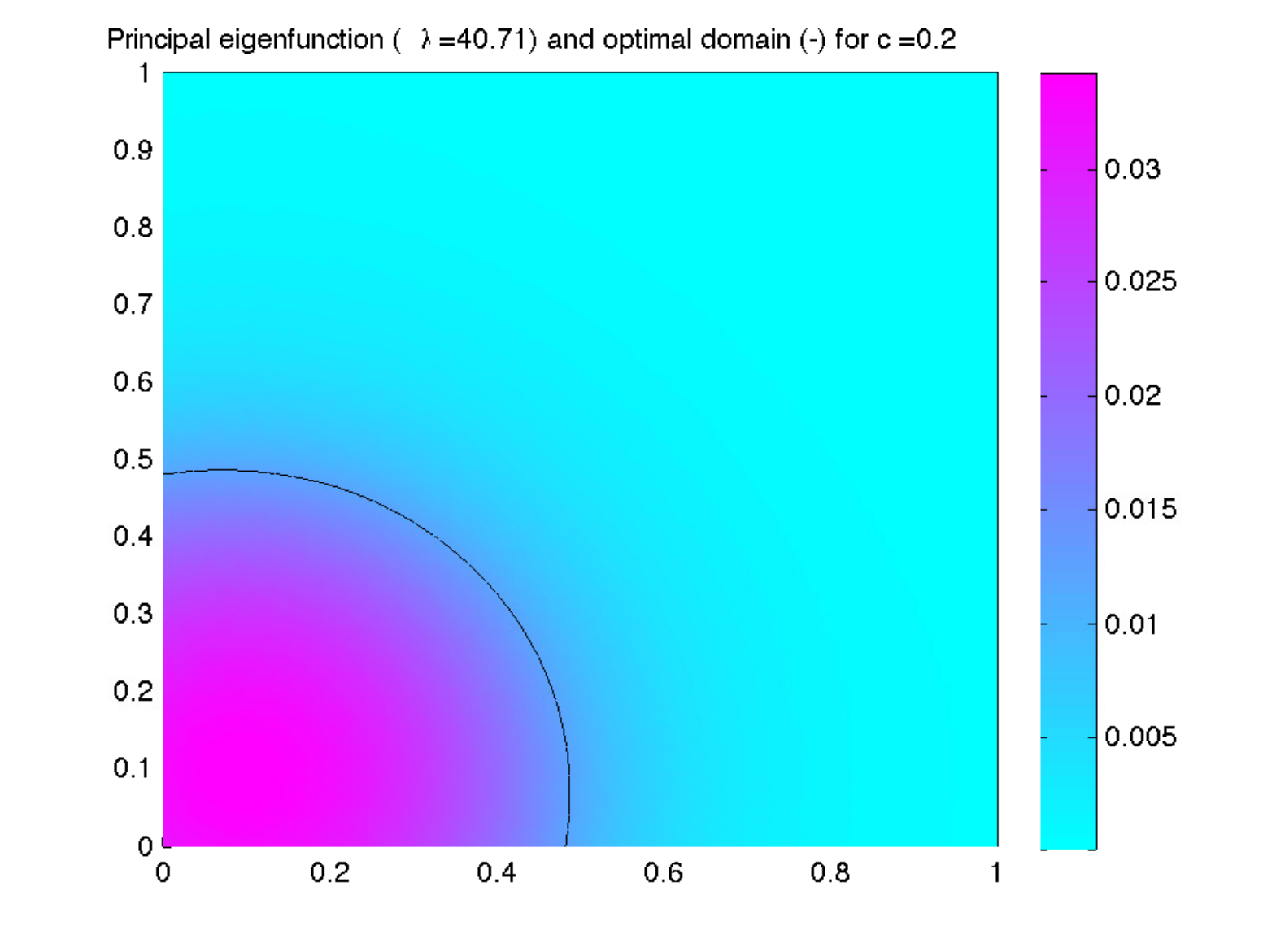}}
\subfigure[$c=0.2$ - $\beta=5$]{\includegraphics[scale=0.35]{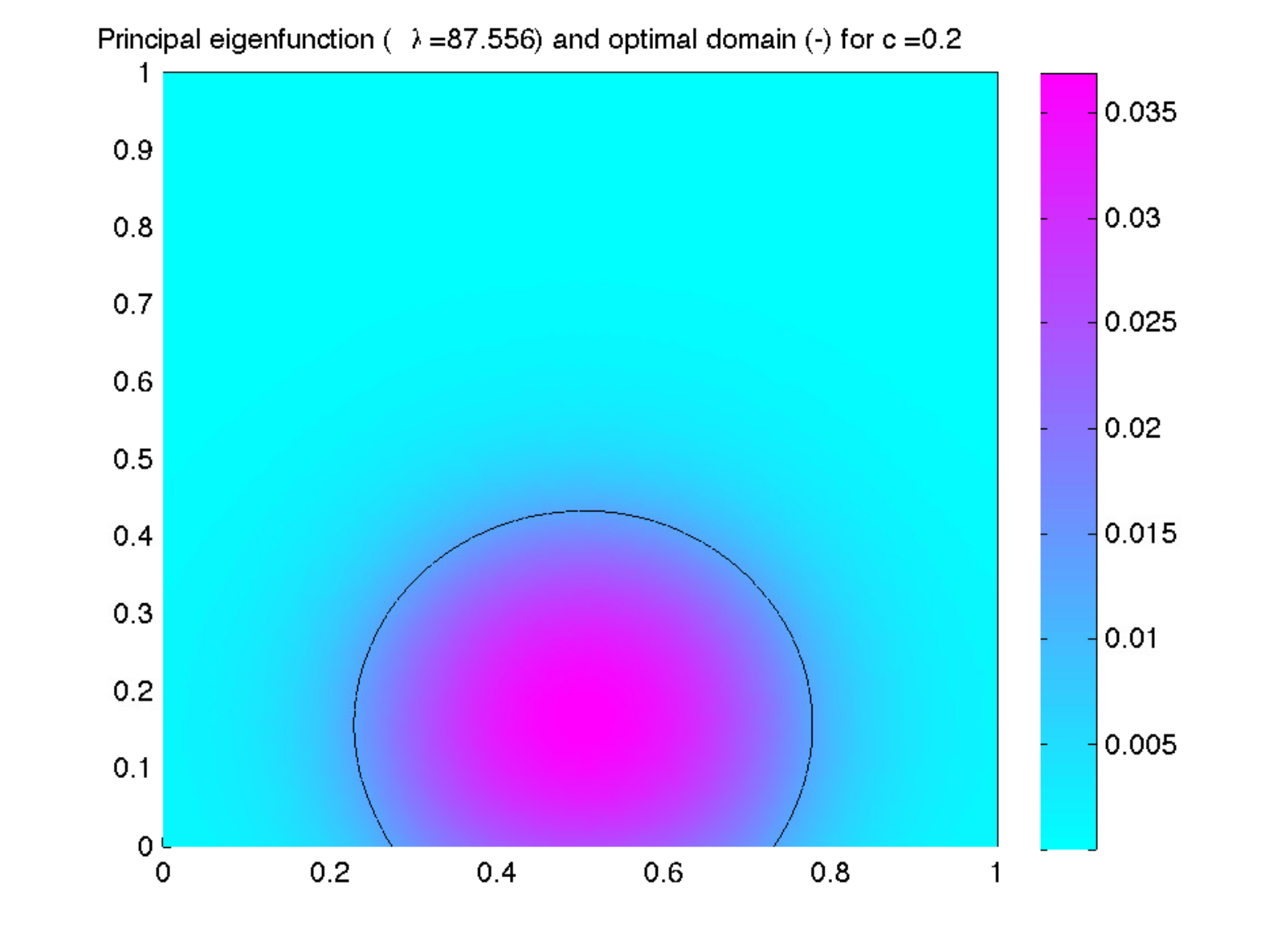}}
\subfigure[$c=0.2$ - $\beta=50$]{\includegraphics[scale=0.35]{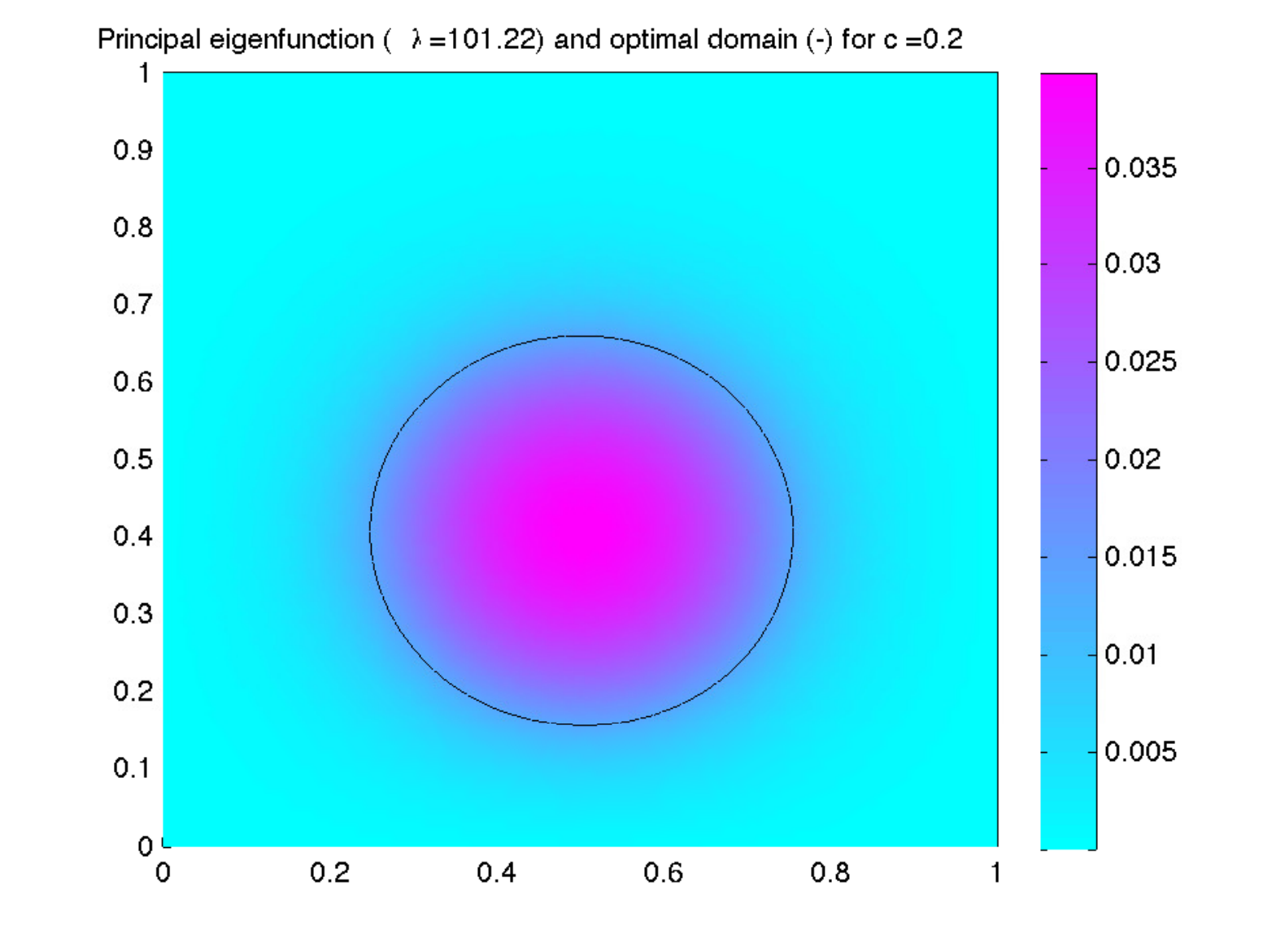}}
\subfigure[$c=0.2$ - $\beta=1000$]{\includegraphics[scale=0.35]{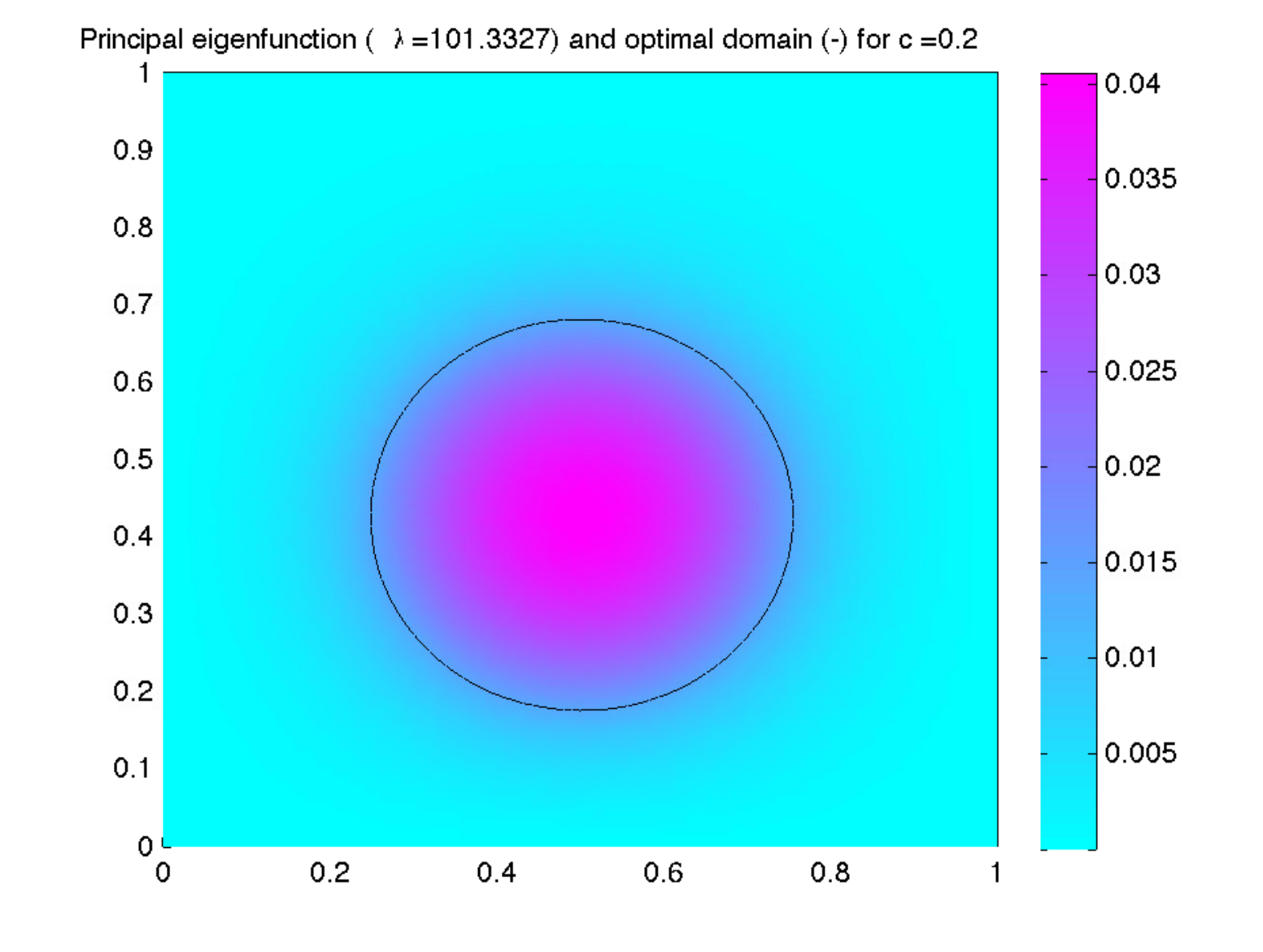}}
\subfigure[$c=0.2$ - $\beta=1$]{\includegraphics[scale=0.35]{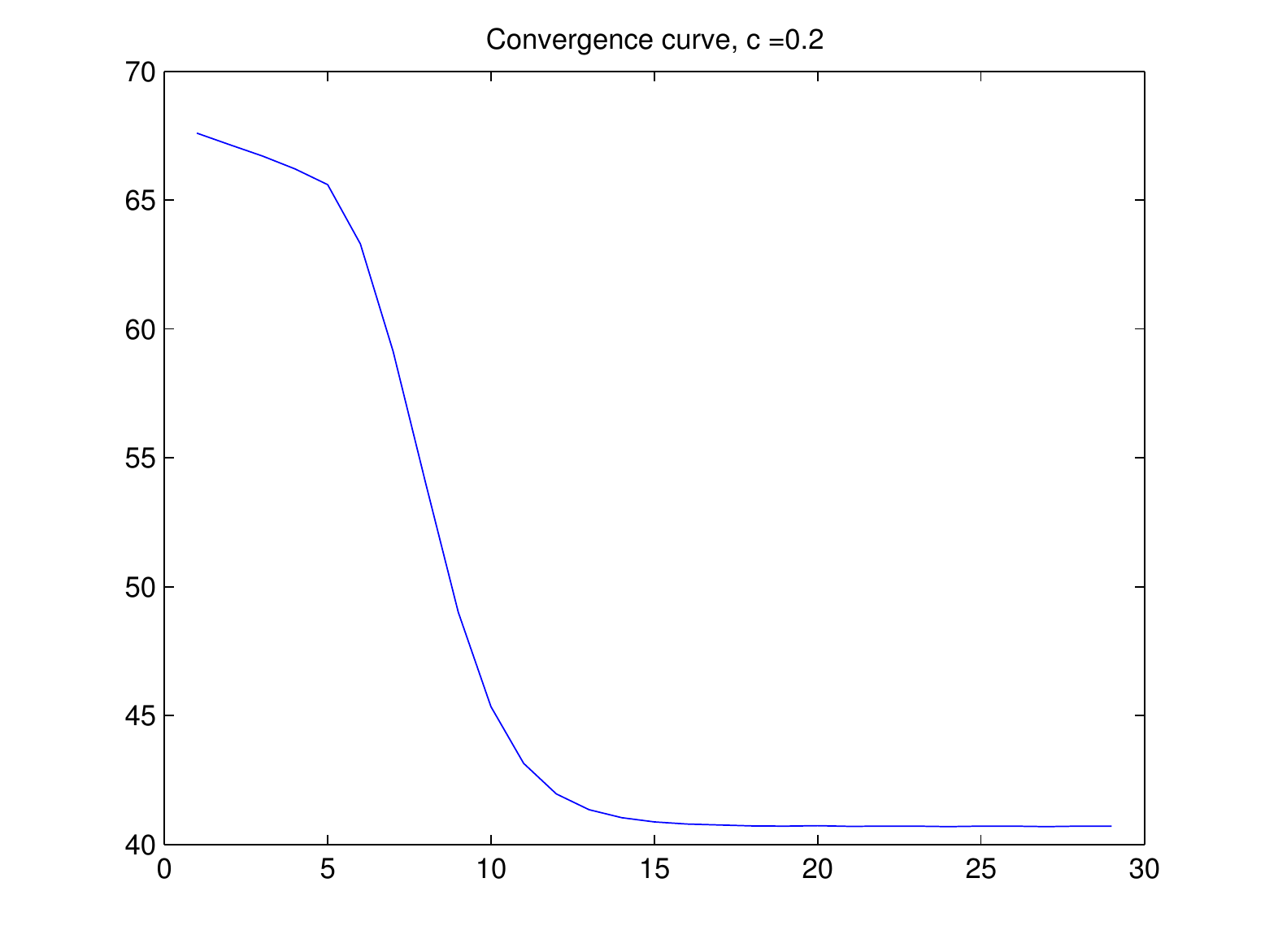}}
\subfigure[$c=0.2$ - $\beta=1000$]{\includegraphics[scale=0.35]{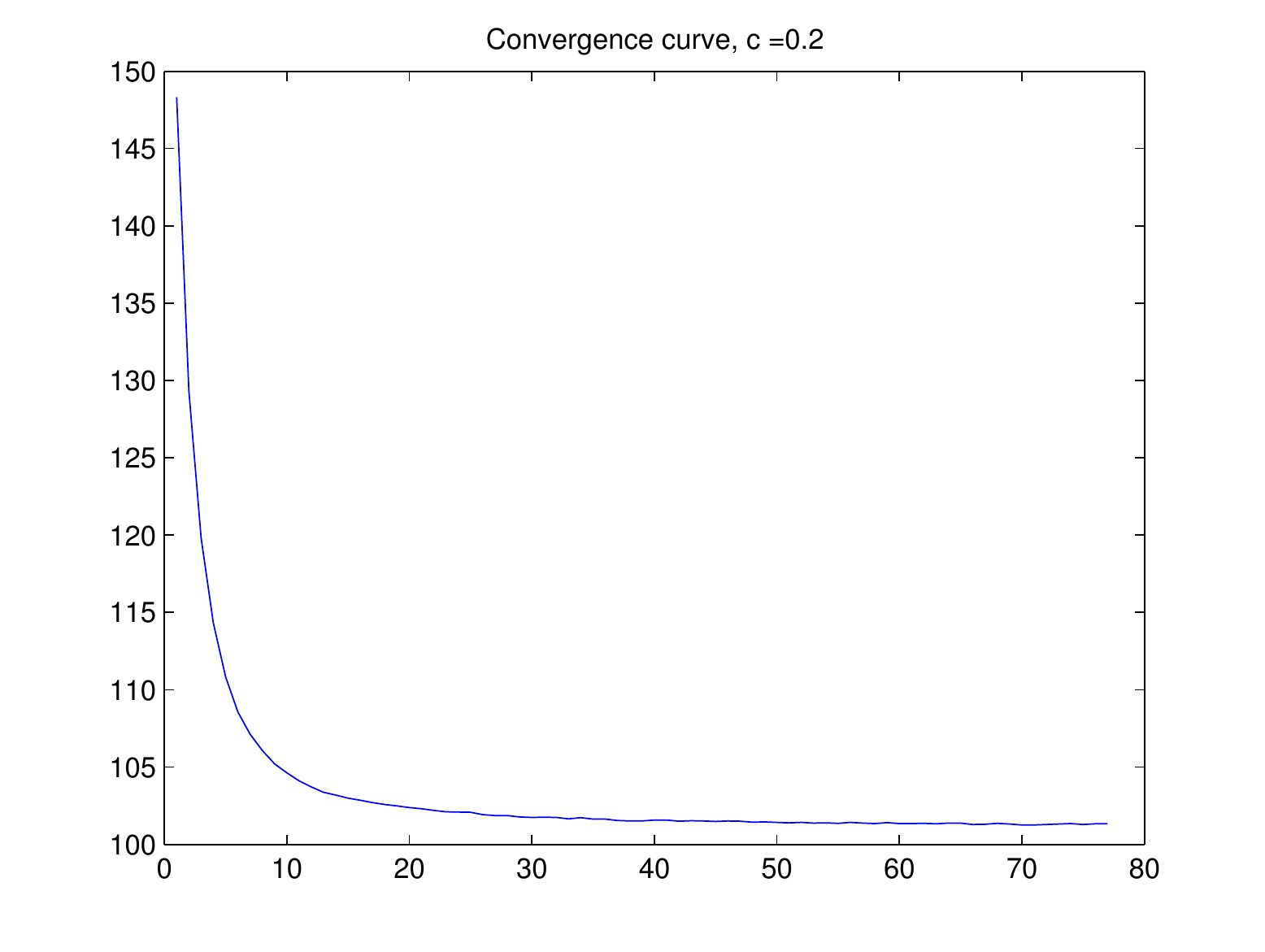}}
\caption{$\Omega=(0,1)^2$. Optimal domains for $\kappa=0.5$, $c=0.2$ and $\beta\in \{1,5,50,1000\}$ and two examples of convergence curves for $\kappa=0.5$, $c=0.2$ and $\beta\in \{1,1000\}$\label{fig:squarebeta}} 
\end{figure}


\begin{figure}[h!]
\centering
\subfigure[$c=0.2$ - $\beta=1$]{\includegraphics[scale=0.35]{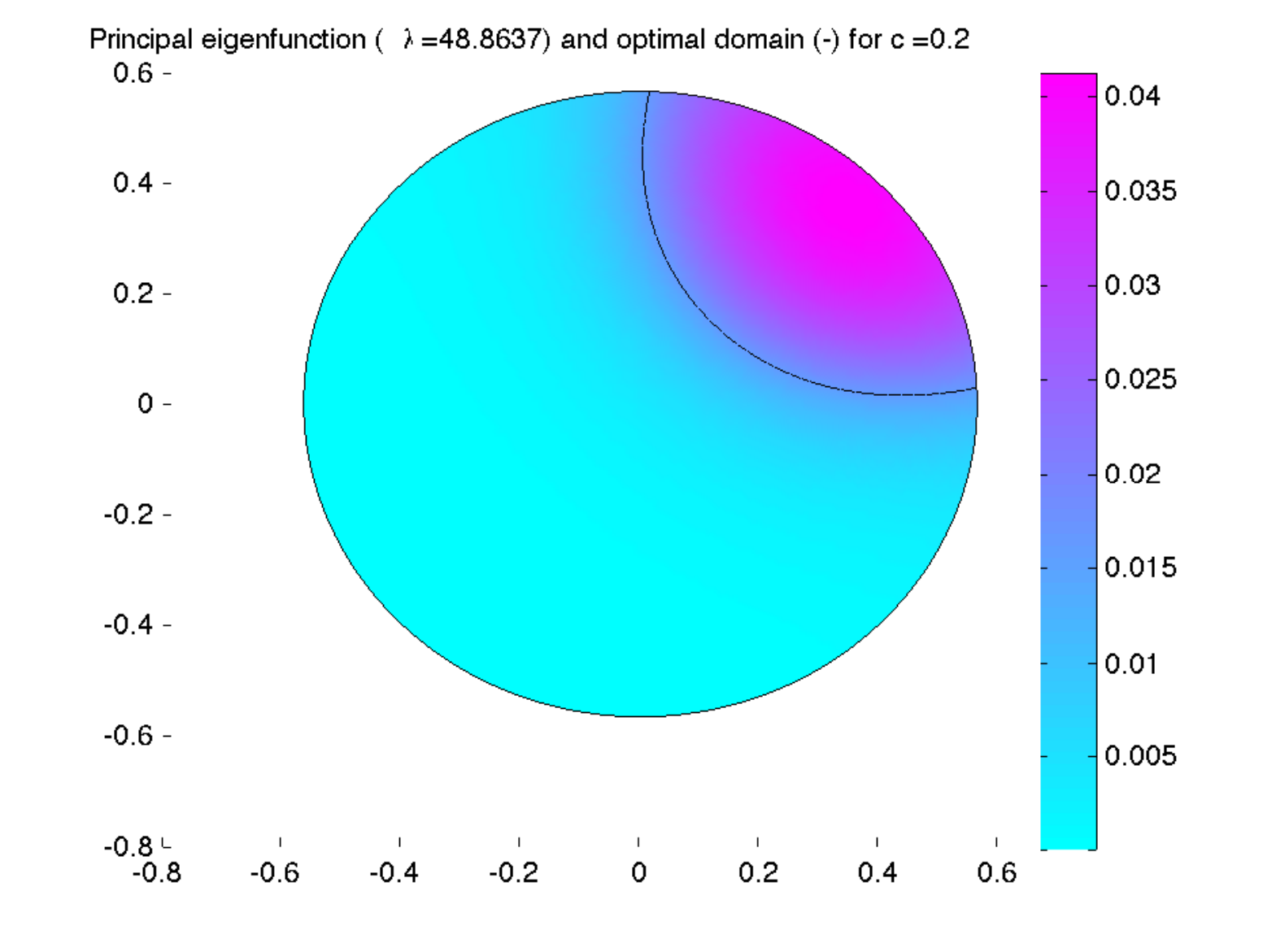}}
\subfigure[$c=0.2$ - $\beta=5$]{\includegraphics[scale=0.35]{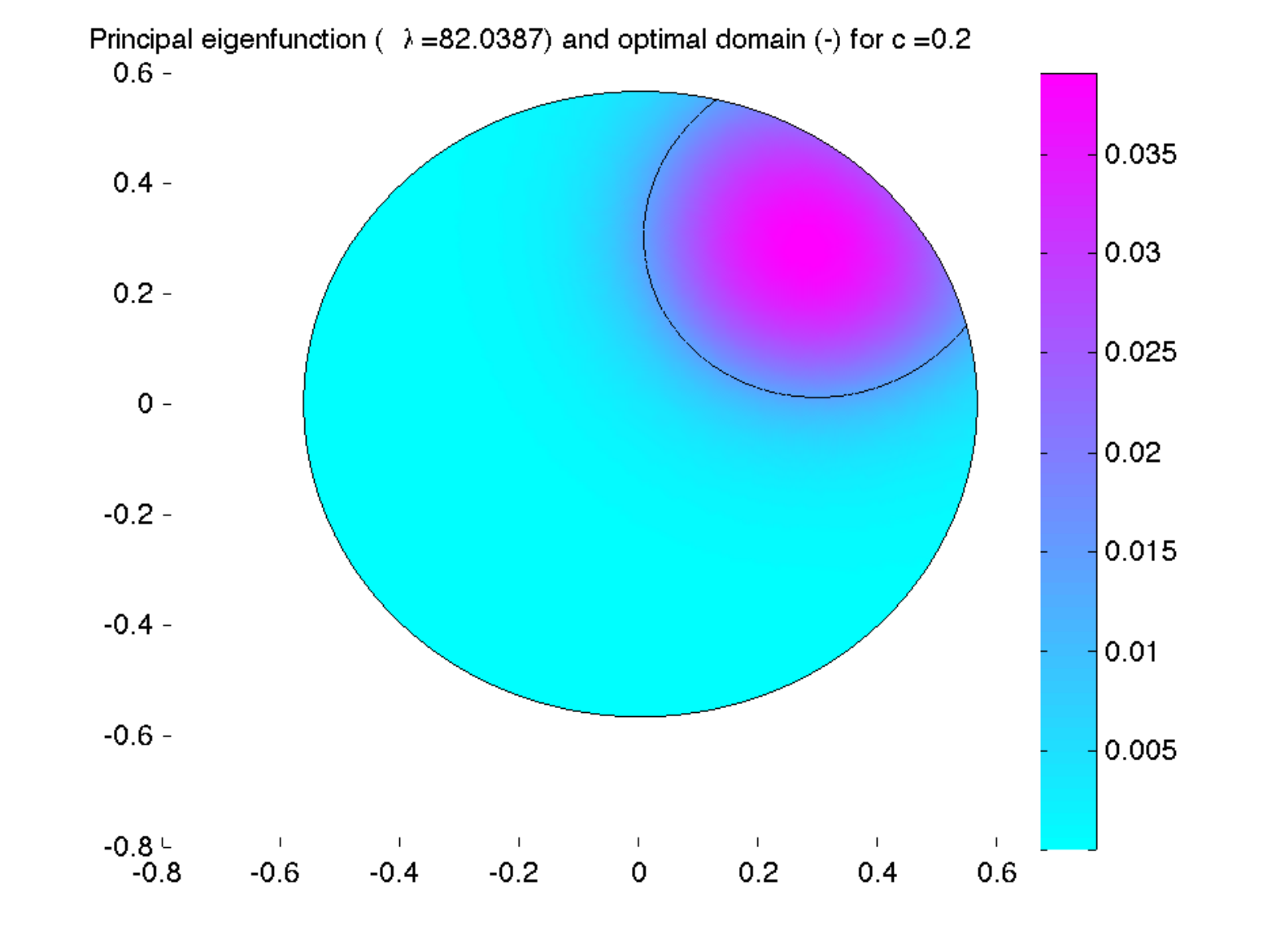}}
\subfigure[$c=0.2$ - $\beta=50$]{\includegraphics[scale=0.35]{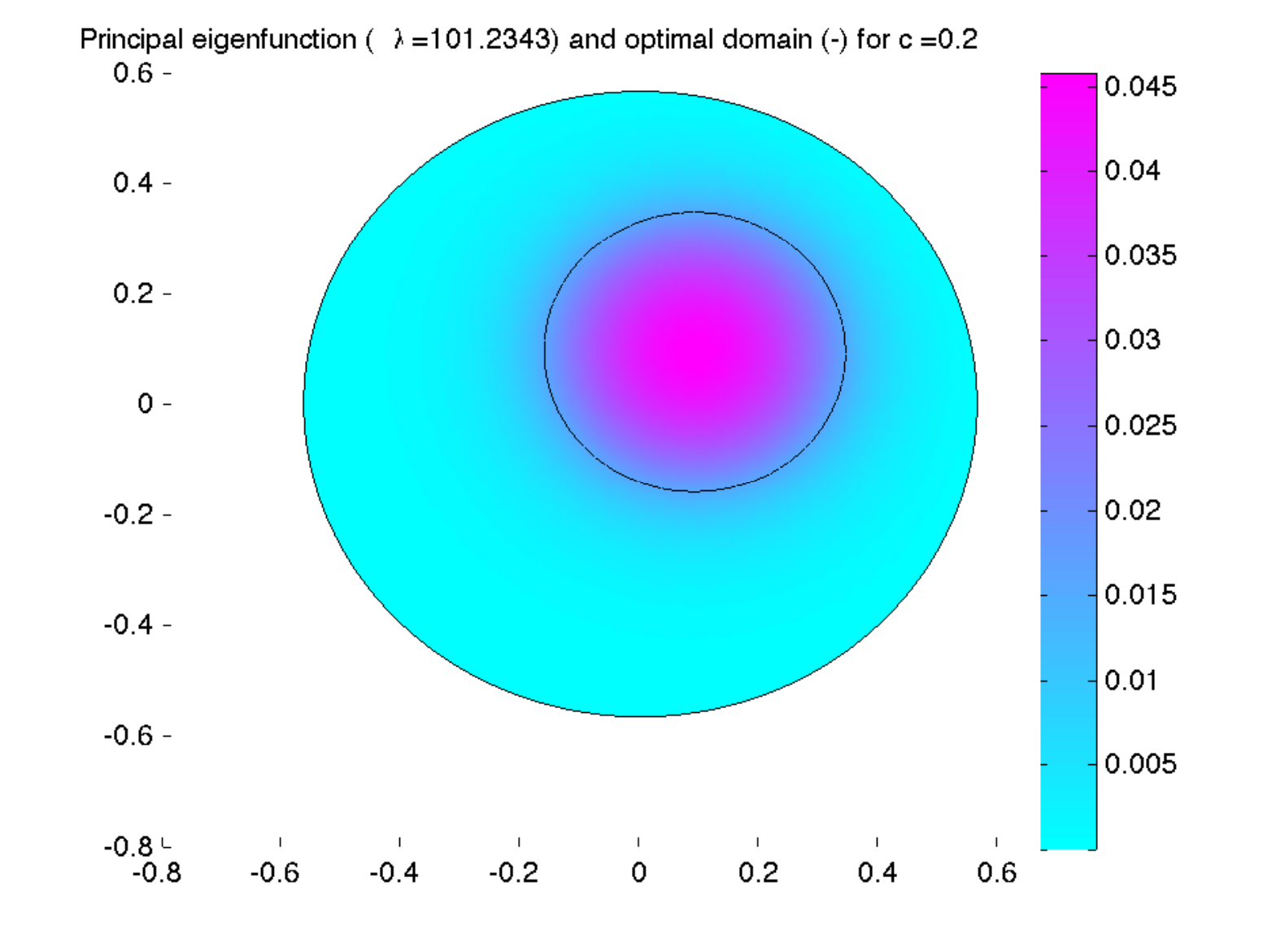}}
\subfigure[$c=0.2$ - $\beta=1000$]{\includegraphics[scale=0.35]{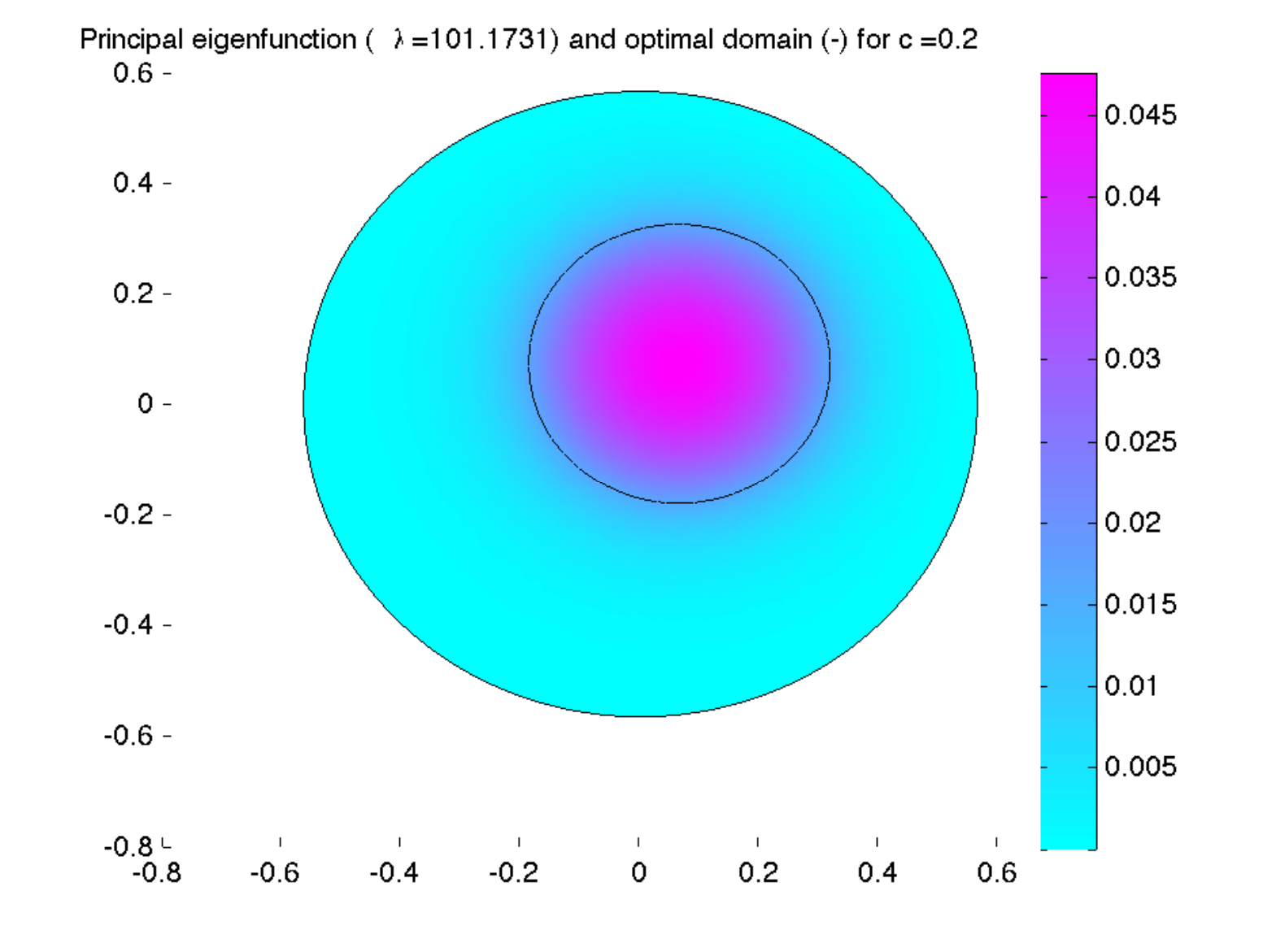}}
\subfigure[$c=0.2$ - $\beta=1$]{\includegraphics[scale=0.35]{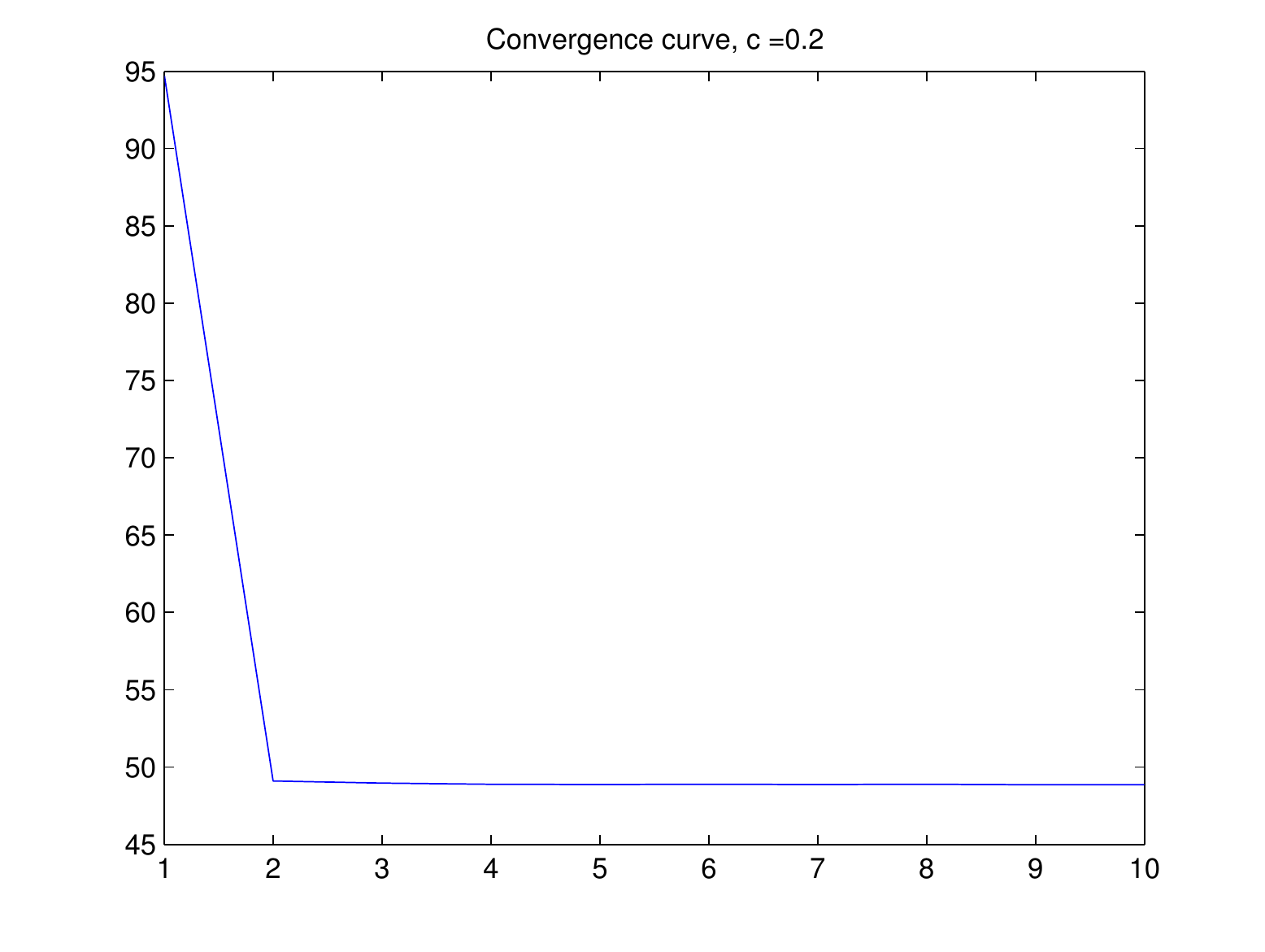}}
\subfigure[$c=0.2$ - $\beta=1000$]{\includegraphics[scale=0.35]{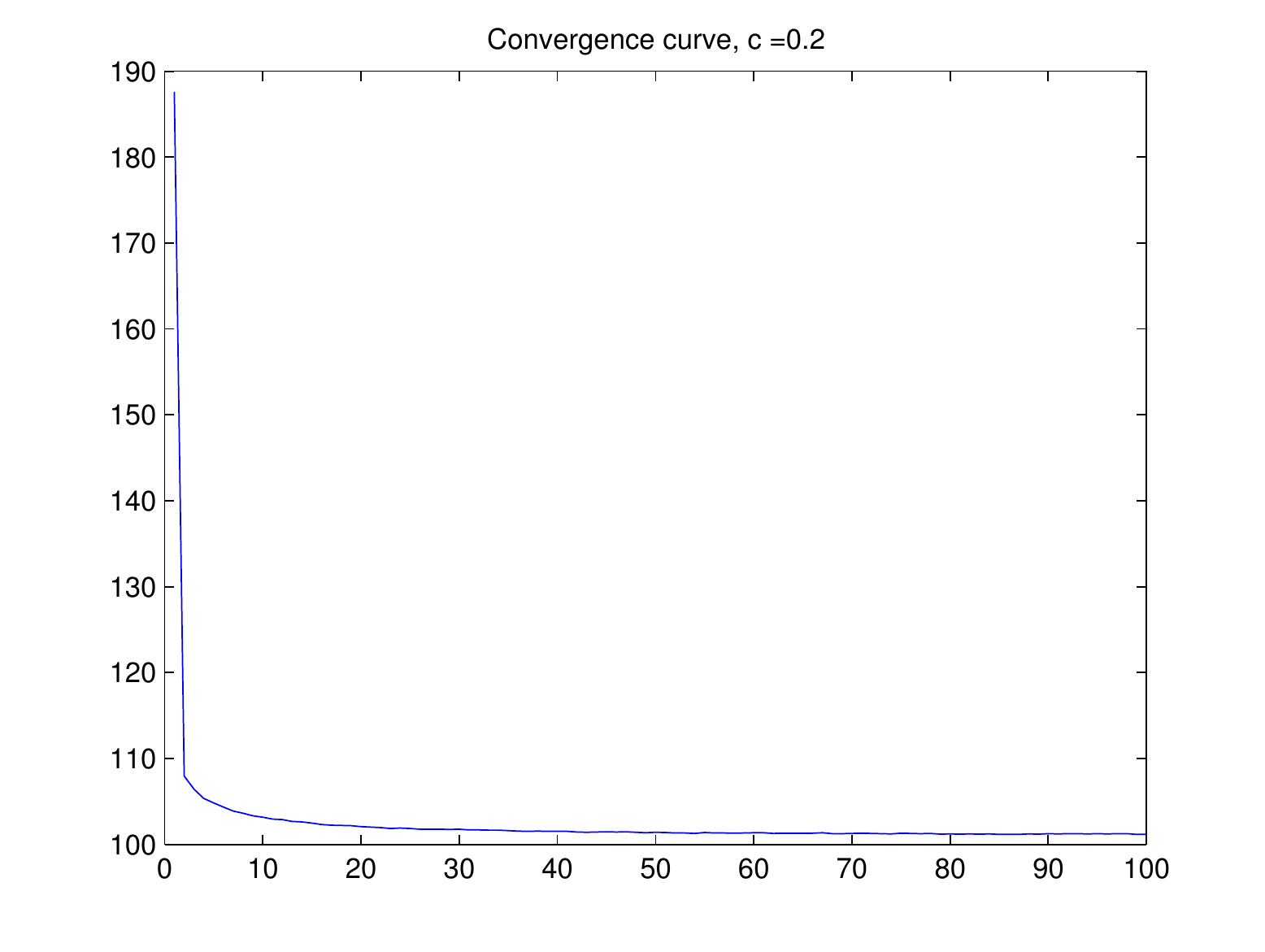}}
\caption{$\Omega=B(0,1)$. Optimal domains for $\kappa=0.5$, $c=0.2$ and $\beta\in \{1,5,50,1000\}$ and two examples of convergence curves for $\kappa=0.5$, $c=0.2$ and $\beta\in \{1,1000\}$\label{fig:diskbeta}} 
\end{figure}

\bibliographystyle{abbrv}
\bibliography{eig_indefinite.bib}
\bigskip

\end{document}